\documentclass[12pt,a4paper,leqno]{amsart}
\usepackage{color,tikz}
\usepackage{amsmath,amsfonts,amsthm} 
\usepackage{amssymb}
\usetikzlibrary{cd}
\usetikzlibrary{decorations.markings}

\usepackage[utf8]{inputenc}
\usepackage{verbatim}
\usepackage{comment}
\usepackage[ body={470pt,688pt} ]{geometry}
\usepackage{url,colonequals}
\usepackage{enumerate,mathrsfs}

\let\nc\newcommand \let\rnc\renewcommand \let\dmo\DeclareMathOperator

\nc\prodl[1]{{\prod\limits_{#1}}}
\nc\bunderline[1]{\underline{#1\mkern-3mu}\mkern3mu }
\nc\boverline[1]{\overline{#1\mkern-3mu}\mkern3mu}
\nc\hooklongrightarrow{\lhook\joinrel\longrightarrow}
\nc\hooklongleftarrow{\longleftarrow\joinrel\rhook}
\DeclareRobustCommand\twoheadlongrightarrow{\relbar\joinrel\twoheadrightarrow}
\nc\cL{\mathcal L}
\nc\cO{\mathcal O}
\nc\cK{\mathcal K}
\nc\cF{\mathcal F}
\nc\cE{\mathcal E}
\nc\kk{\kappa} 
\rnc\Box{\square}
\nc\isom{\simeq} 
\nc\isomarrow{\xrightarrow{\,\lowsim\,}}
\nc\isomarrowby[1]{
  \xrightarrow[\smash{\raise0.75ex\hbox{$\ \scriptstyle\sim\ $}}]{#1}}
\let\surject\twoheadlongrightarrow
\nc\Gm{\mathbb{G}_{\mathrm m}}
\nc\ZZ{\mathbb{Z}}
\nc\ZZpm{\ZZ_\fm} 
\nc\QQ{\mathbb{Q}}
\nc\RR{\mathbb{R}}
\nc\CC{\mathbb{C}}
\nc\Ha{\mathbb{H}}
\nc\comp{\chi} 
\nc\id{\mathrm{id}}
\nc\triv{\mathrm{triv}}
\nc\sgn{\mathrm{sgn}}
\nc\eval{\mathrm{eval}}
\nc\tors{\mathrm{tors}}
\nc\sat{\mathrm{sat}}
\nc\diff{\mathrm{diff}}    
\nc\ediff{\mathrm{diff}_E} 
\nc\incl{\mathrm{incl}}
\nc\Pont{\mathrm{Pont}}
\nc\KAb{\mathcal{K}(\mathbf{Ab})}
\nc\fm{\mathfrak m}
\nc\gfm{\fm} 
\nc\Gr{{Gr}} 
\nc\Grg{\abs{\Gr}}
\nc\Grm{{\Gr_\fm}}
\nc\Zdual{^\vee}  
\nc\bull{\bullet}
\nc\defeq{\colonequals}

\let\oBox\Box
\rnc\Box{\raisebox{-0.25pt}{$\oBox$}}
\dmo\coker{coker}
\dmo\im{im}
\dmo\coim{coim}
\let\div\undefined
\dmo\div{div}
\dmo\supp{supp}
\dmo\cone{Cone}
\dmo\tot{Tot}
\dmo\proj{proj}
\dmo\AJ{AJ}
\dmo\AJt{\widetilde{\mathrm{AJ}}}
\dmo\T{T}
\dmo\Hom{Hom}
\dmo\Ext{Ext}
\dmo\End{End}
\dmo\J{J}
\dmo\Pic{Pic}
\let\P\undefined 
\dmo\P{P}
\dmo\Prin{Prin}
\dmo\Prinhat{\widehat{\Prin}}
\dmo\Div{Div}
\dmo\Divhat{\widehat{Div}}
\dmo\Cl{Cl}
\dmo\Clhat{\widehat{Cl}} 

\nc\Cech{\check{\mathcal{C}}}
\nc\Clb{\mathbf{Cl}}
\nc\Clhb{\widehat{\mathbf{Cl}}{}} 
\nc\Jb{\mathbf{J}}
\nc\Pb{\mathbf{P}}
\nc\Pbbar{\overline{\Pb}}
\nc\f{\phi}
\nc\lp{{}^\backprime}

\nc\PL{\underline{\mathrm{PL}}}
\nc\PLm{\PL_\fm} 
\nc\PLp{\PL'_\fm}
\nc\Ling{\mathrm{Lin}}
\nc\Lins{\underline{\mathrm{Lin}}}
\nc\Harm{\underline{\mathrm{Harm}}}
\nc\Harmp{\Harm_\fm}
\nc\Harmpv[1][v]{\Harm_{\fm,#1}}

\nc\dual{^\vee} 
\nc\tr[1]{{}^t#1} 
\nc\adj{^\dagger_{\vphantom{a}}} 
\nc\adjp{^\dagger} 
\def\lowerit#1#2{\smash{\mathrel{\hbox{\lower.35ex\hbox{$#1#2$}}}}}
\nc\lowsim{\mathpalette\lowerit\sim}
\def\raiseit#1#2{\smash{\mathrel{\hbox{\raise.75ex\hbox{$#1#2$}}}}}
\nc\highsim{\mathpalette\raiseit\sim}
\nc\mapsfrom{\leftarrow\!\mapstochar\,}

\let\setminus\smallsetminus 
\nc\abs[1]{\left|#1\right|}
\nc\apairing[3][]{\left\langle#2,#3\right\rangle_{#1}}
\nc\pairing[3][]{\left(#2,#3\right)_{#1}}
\nc\inner[2]{(#1,#2)}
\nc\stab{\vspace*{1ex}$\bull$\enspace}

\nc\marp[1]{\marginpar{\raggedright\color{blue}#1}}
\nc\marpb[1]{\marginpar{\raggedright\color{red}Bruce: #1}}
\nc\done{\marp{CORRECTED OR ANNOTATED UP TO HERE}}
\nc\bruce[1]{\textsf{\color{blue} $\spadesuit$ Bruce: #1}}
\nc\tony[1]{\textsf{\color{red} $\clubsuit$ Tony: #1}}
\nc\ken[1]{\textsf{\color{red} $\heartsuit$ Ken: #1}}

\definecolor{refkey}{HTML}{0000ff}
\definecolor{labelkey}{HTML}{0000ff}

\usepackage[
  pdfauthor={Bruce W. Jordan, Kenneth A. Ribet, and Anthony J. Scholl},
  pdftitle={Generalized Jacobians of graphs}, ]{hyperref}

\numberwithin{equation}{subsection} 

\theoremstyle{plain}
\newtheorem{thm}{Theorem}[subsection]
\newtheorem*{thm*}{Theorem}
\newtheorem{lem}[thm]{Lemma}
\newtheorem{cor}[thm]{Corollary}
\newtheorem{prop}[thm]{Proposition}
\theoremstyle{definition}
\newtheorem{defn}[thm]{Definition}
\theoremstyle{remark}
\newtheorem{rem}[thm]{Remark}
\newtheorem*{rem*}{Remark}

\rnc\i{^{-1}} 

\usepackage[alphabetic,lite,nobysame]{amsrefs} 

\title{Generalized Jacobians of graphs}

\begin{document}


\author{Bruce~W.~Jordan}
\address{Department of Mathematics, Baruch
  College, The City University of New York, One Bernard Baruch Way,
  New York, NY 10010-5526, USA}
\email{bruce.jordan@baruch.cuny.edu}

\author{Kenneth~A.~Ribet}
\address{Mathematics Department, University of
  California, Berkeley, CA 94720, USA}
\email{ribet@math.berkeley.edu}

\author{Anthony~J.~Scholl}
\address{Department of Pure Mathematics and Mathematical Statistics,
  Centre for Mathematical Sciences, Wilberforce Road, Cambridge CB3
  0WB, England}
\email{a.j.scholl@dpmms.cam.ac.uk}

\begin{abstract}
  We define a generalized Jacobian $\J_{\fm}(\Gr)$ and a generalized
  Picard group $\P_\fm(\Gr)$ of a graph $\Gr$ with respect to a
  modulus $\fm=\sum_{i=1}^sm_iw_i$ with $w_i$ vertices of $\Gr$ and
  $m_i\geq 1$. These groups occur as the component groups of the N\'eron
  models of generalized Jacobians. We prove a universal mapping property for
  $\J_\fm(\Gr)$ and show that an Abel--Jacobi map in this context
  induces an isomorphism from $\P_\fm(\Gr)$ to
  $\J_\fm(\Gr)$. We also reinterpret $\P_\fm(\Gr)$ in terms of
  sheaves on the geometric realization of $\Grg$, making a connection
  with tropical geometry.
\end{abstract}

\maketitle
\tableofcontents
\section*{Introduction}

The Jacobian of a graph has been independently discovered numerous
times during the last 60 years.  It arises in chip-firing in discrete
dynamical systems where it is called the sandpile group; see, for
example, \cite{Bz} or \cite{Kli}.  It also arises from the
well-known analogy between an algebraic curve $X/\CC$ and a finite
graph $\Gr$.  An example of this is the foundational work \cite{BdN}
of Bacher, de la Harpe, and Nagnibeda.

In this paper we further develop the analogy between curves and
graphs, striving to make the analogy tight, precise, and complete.
The classical theory of curves (summarized in Section \ref{LABEL000})
attaches to a curve $X/\CC$ two (isomorphic) abstract groups and two
(isomorphic) complex abelian varieties. The abstract groups are the
divisor class group (together with its degree-zero subgroup)
\[
\Cl(X)\colonequals \Div(X)/\Prin(X)\supset \Cl^{0}(X)=\Div^0(X)/\Prin(X)
\]
and the Picard group of isomorphism classes of line bundles
\[
\Pic(X)\colonequals H^{1}(X,\cO_X^\times)\supset\Pic^{0}(X)=H^1(X,\cO_X^\times)
^{{\rm deg} = 0}.
\]
The complex abelian varieties are the Jacobian (or Albanese) variety
$\J(X)$ and the Picard variety
$\P(X)\colonequals H^1(X,\cO_X)/H^1(X,\ZZ(1))$.  Basic classical
results of the subject concern these four groups and the maps between
them:
  \begin{enumerate}[\upshape (a)]
  \item 
    \textup{(Abel's theorem)}
    The Abel-Jacobi map $\AJt:\Div^{0}(X)\rightarrow \J(X)$ induces
    an isomorphism $\AJ:\Cl^{0}(X)\isomarrow\J(X)$.
  \item
    \textup{(Universal mapping property)}
    Suppose $h:X\rightarrow A$ maps $X$ into an abelian variety $A/\CC$
    with $P_0\in X$ satisfying $h(P_0)=0$. Then there is a unique
    homomorphism of abelian varieties $h':\J(X)\rightarrow A$
    with $h(P)=h'(\AJt(P-P_0))$.
  \item The exponential sequence induces an isomorphism
    $\P(X)\isomarrow \Pic^0(X)$.
  \item
    \textup{(Duality)}
    The Jacobian $\J(X)$ is self-dual.  There is an isomorphism
    $\J(X)\isomarrow \P(X)$ induced by the Serre duality
    pairing $H^0(X,\Omega^1_X)\times H^1(X,\cO_X)\rightarrow \CC$.
  \end{enumerate}
We begin our treatment of graphs in Section \ref{LABEL200}.  We give
the analogues for a graph $\Gr$ of the four groups $\J(X)$, $\P(X)$,
$\Cl^0(X)$, and $\Pic^0(X)$ attached to a curve $X/\CC$.  For
$\J(\Gr)$ and $\Cl^0(\Gr)$ our definition agrees with \cite{BdN}, but
note that \cite{BdN} calls $\Cl^0(\Gr)$ the Picard group of the graph
in contrast with our terminology.  The Picard group $\P(Gr)$ is the
analogue of the Picard variety $\P(X)$ and the geometric Picard group
$\Pic^{0}(\left|\Gr\right| )$ is the analogue of the group $\Pic^0(X)$
of isomorphism classes of degree-zero line bundles on $X$. The
definition of $\P(\Gr)$ (Definition \ref{LABEL405}) is elementary. In
contrast, defining $\Pic^0(\left|\Gr\right|)$ and making precise the
analogies to curves, which is the content of Section \ref{LABEL470},
is more complicated, relying on the geometric realization
$\left|\Gr\right|$ of the graph $\Gr$ and certain sheaves on it
defined by Molcho and Wise \cite{MW1} in the context of tropical
geometry.  We prove the analogues for graphs of Abel's theorem (a) and the
universal mapping property (b) following \cite{BdN}. We also prove 
results for graphs analogous to (c) and (d).
Many results continue to hold for infinite graphs under the assumption
that they are locally finite (every vertex has finite degree).

We now turn to the main topic of our paper---extending the theory to
add a modulus $\fm$.  For curves this was done in the 1950's by
Rosenlicht, Lang, and Serre -- we review this case in Section
\ref{LABEL080}. A curve $X/\CC$ has modulus $\fm=\sum_{i=1}^{s}m_iP_i$
where the $P_i$ are points on $X$ and $m_i\geq 1$. There are two
abstract groups attached to $(X,\fm)$: the divisor ray class group
$\Cl_\fm(X)$ and the generalized Picard group $\Pic_\fm(X)$
classifying isomorphism classes of line bundles on $X$ together with
an $\fm$-rigidification. Likewise there are two algebraic groups over
$\CC$ attached to $(X,\fm)$: the generalized Jacobian variety
$\J_\fm(X)$ and the generalized Picard variety $\P_\fm(X)$. If $\fm$
is reduced (i.e., all $m_i=1$), then $\J_\fm(X)$ is an extension of
the $g$-dimensional abelian variety $\J(X)$ by a torus of dimension
$s-1$. Abel's theorem extends to generalized Jacobians (Theorem
\ref{LABEL105}): the Abel-Jacobi map induces an isomorphism
$\AJ_\fm: \Cl^0_\fm(X)\isomarrow \J_\fm(X)$.  There is the universal
mapping property for generalized Jacobians given in Theorem
\ref{LABEL150} and once again Serre Duality induces an isomorphism
$\J_{\fm}(X)\rightarrow \P_{\fm}(X)$.  However the generalized
Jacobian $\J_\fm(X)$ is not self-dual.

A finite graph $\Gr$ with vertices $V$ and edges $E$ has modulus
$\fm=\sum_{i\in I}w_i$ with $w_i\in V$.  Our task is to give
constructions and theorems which recast the theory of curves with
modulus for graphs. We begin in Section \ref{LABEL620} by developing
$\fm$-harmonic theory for graphs.  This is done using the auxiliary
construction of the {\emph{extended graph}} $\Gr_\fm$ which contains
$\Gr$ as a subgraph. Assume $\Gr$ is connected. The $\fm$-divisors on
$\Gr$ are $\Div_\fm^0(\Gr)\colonequals \ZZ[V]_0\oplus \ZZ[I]$ and we
define $\fm$-principal divisors
$\Prin_\fm(\Gr)\subseteq \Div_\fm^0(\Gr)$.  The divisor ray class
group is then $\Cl^{0}_\fm(\Gr)=\Div_\fm^0(\Gr)/ \Prin_\fm(\Gr)$.
Integral Hodge theory on $\Grm$ is used to define the $\fm$-Jacobian
$\J_\fm(\Gr)$ and the $\fm$-Picard group $\P_\fm(\Gr)$.  Analogous to
curves, the $\fm$-Picard group $\P_\fm(\Gr)$ is an extension of
$\P(\Gr)$ by a free abelian group of rank $\#I-1$; this is proved in
Proposition \ref{LABEL1055}.  We use duality to give the class of the
extension and to give a canonical isomorphism
$\zeta_\fm:\J_\fm(\Gr)\isomarrow \P_\fm(\Gr)$ (Theorems
\ref{LABEL1070}, \ref{LABEL830}).  The Abel--Jacobi map induces an
isomorphism $\AJ_\fm:\Cl^0_\fm(\Gr)\isomarrow \J_\fm(\Gr)$. Hence
$\J_\fm(\Gr)$ is an extension of $\J(\Gr)$ by a free abelian group of
rank $\#I-1$.  As was the case with empty modulus, the graph analogue
of $\Pic_{\fm'}(X)$ for a curve $X$ with modulus $\fm'$ is more
coomplicated, involving constructions from tropical geometry and the
geometric realization $\Grg$ of $\Gr$.  We define $\Pic_\fm(\Grg)$ for
a modulus $\fm$ on $\Gr$ in terms of $\fm$-rigidified $\Harm$--torsors
on $\Grg$ and show that we have an isomorphism
$\P_\fm(\Gr)\isomarrow \Pic^0_\fm(\Grg)$. We also extend most of these
results to infinite graphs.

These constructions for a graph $\Gr$ with modulus $\fm$ have their
origin in arithmetic geometry, just as the presentations of $\J(\Gr)$
and $\Cl^0(\Gr)\isom \P(\Gr)$ appeared in the work of Grothendieck and
Raynaud, respectively. In \cite{jrs} we extended the work of Raynaud
by giving two presentations of the component group $\Phi_{\fm '}(X)$
of the N\'{e}ron model of the generalized Jacobian of a proper regular
curve $X$ over a suitable discrete valuation ring with reduced modulus
$\fm '$.  In terms of the dual graph $\Gr$ of the special fiber with
modulus $\fm$ defined from the modulus $\fm '$ on $X$, our first
presentation of $\Phi_{\fm '}(X)$ was $\Cl^0_\fm(\Gr)$.  We introduced
the extended graph $\Grm$ in \cite{jrs}, and the second presentation
we gave of $\Phi_{\fm '}(X)$ is $\J_\fm(\Gr)$, which was defined using
$\Grm$. Since both constructions present the same group, we must have
$\Cl^0_\fm(\Gr)\simeq \J_\fm(Gr)$, which is Abel's theorem in the
context of graphs.

In \cite{GJRWW} the authors study which finite abelian groups can
occur as the Jacobians of graphs. It would be interesting to
investigate the corresponding problem for generalized Jacobians.

We would like to thank Dhruv Ranganathan for drawing our attention to
\cite{MW1}, which led us to the definitions of the generalized Picard
groups $\Pic_\fm(\Grg)$ of the geometric realization of $\Gr$, and to
the results of \cite{GJRWW}.

\noindent \emph{Notations:} For a set $S$, we write $\ZZ[S]$ for the
free abelian group on $A$ and $\ZZ[S]_0$ for the kernel of the
augmentation $\deg\colon \ZZ[S] \to \ZZ$. If $S$ is finite, we write
$\ZZ^{S,0}$ for the group of functions $S \to \ZZ$ whose sum is zero. For an
abelian group $M$ we write $M\Zdual=\Hom(M,\ZZ)$.

\section{Jacobians and generalized Jacobians of curves}
\label{LABEL000}

In this section, $X/\CC$ will be a nonsingular projective irreducible
curve of genus $g$ with function field $\kk(X)$, $\Div(X)$ its
divisor group, $\Div^0(X)$ the subgroup of divisors of degree $0$,
and $\Prin(X)=\div(\kk(X)^\times)$ the subgroup of principal
divisors. All sheaf cohomology groups will be for the classical
(analytic) topology on $X$.

\subsection{Jacobians of curves}
\label{LABEL010}

The classical theory of curves attaches to $X$ two (isomorphic)
abstract groups and two (isomorphic) complex abelian varieties. The
abstract groups are the divisor class group (together with its degree
zero subgroup)
\[
  \Cl(X)\defeq \Div(X)/\Prin(X)\supset
  \Cl^0(X)=\Div^0(X)/\Prin(X) 
\]
and the Picard group of isomorphism classes of line bundles
\[
  \Pic(X)\defeq H^1(X,\cO_X^\times)\supset
  \Pic^0(X)=H^1(X,\cO_X^\times)^{\mathrm{deg}=0}.
\]
For $D\in \Div(X)$, denote by $[D]$ the equivalence class in $\Cl(X)$
containing $D$. The map $[D] \to \cO_X(D)$ gives a
degree-preserving isomorphism $\Cl(X) \isomarrow \Pic(X)$, which up to sign
is the boundary map for the cohomology of the exact sequence of
sheaves
\begin{equation}
  \label{LABEL020}
  0 \to \cO_X^\times \to \cK_X^\times \xrightarrow{\div}
  \bigoplus_{x\in X} x_*\ZZ \to 0.
\end{equation}
(Here $x_*\ZZ$ is the skyscraper sheaf supported on $x\in X$.)
The abelian varieties are the Jacobian (or Albanese) variety $\J(X)$
and the Picard variety $\P(X)$, which as complex tori are:
\begin{equation}
  \label{LABEL030}
  \J(X)\defeq\frac{H^0(X,\Omega^1_X)^\ast}{H_1(X,\ZZ)}
  \quad\text{and}
  \quad
  \P(X)\defeq \frac{H^1(X,\cO_X)}{H^1(X,\ZZ(1))}
\end{equation}
(where $(-)^\ast$ here denotes $\CC$-dual).  On the one hand, we have
the Abel--Jacobi map
\begin{equation}
  \label{LABEL040}
  \AJt\colon\Div^0(X)\to \J(X)
\end{equation}
induced by $P-Q\mapsto \left(\omega\mapsto \int_{Q}^{P}\omega\right)$
for $P$,$Q$ points on $X$ and $\omega\in H^0(X,\Omega^1_X)$, which
satisfies:

\begin{thm}[Abel's Theorem]
  $\AJt$ is surjective with kernel $\Prin(X)$, inducing an
  isomorphism $\AJ\colon\Cl^0(X)\isomarrow \J(X)$.
\end{thm}

On the other hand, the exponential sequence of sheaves on $X$
\begin{equation}
  \label{LABEL050}
  0 \to \ZZ(1) \to \cO_X \xrightarrow{\exp} \cO_X^\times
  \to 0
\end{equation}
induces an isomorphism $\P(X) \isomarrow \Pic^0(X)$.  One then has 
compatibility with duality:
\begin{thm}\label{LABEL060}
  The following diagram is commutative up to sign:
  \begin{equation}
    \label{LABEL070}
    \begin{tikzcd}
      \Cl(X) \arrow[rrr, "{[D]\mapsto \cO_X(D)}", "{\sim}" below] &&& \Pic(X)
      \\
      \Cl^0(X) \arrow[u, hook] \arrow[rrr, "{\sim}"]
      \arrow[d, "{\mathrm{AJ}}" left, "{\wr}"]
      &&& \Pic^0(X) \arrow[u, hook] 
      \\
      \J(X) \arrow[rrr, "{\sim}", "{\text{\upshape Serre/Poincar\'e duality}}" below]
      &&& \P(X) \arrow[u, "{\wr}", "{\mathrm{exp}}" right]
    \end{tikzcd}
  \end{equation}
  where the bottom map is induced by the Serre duality pairing
  $H^0(X,\Omega_X^1) \times H^1(X,\cO_X)\to \CC$.
\end{thm}

The constructions in the left column of the diagram \eqref{LABEL070} are
covariant in $X$ (Albanese functoriality), whereas those in the right
column of \eqref{LABEL070} are contravariant in $X$ (Picard
functoriality).
  
Finally, we have the universal property of the Jacobian:

\begin{thm}
  Let $P_0\in X$, and let $h\colon X \to A$ be a morphism from $X$ to
  an abelian variety $A/\CC$, with $h(P_0)=0$. Then there is a unique
  homomorphism of abelian varieties $h' \colon \J(X) \to A$ such that
  $h(P)=h'(\AJt(P-P_0))$.
\end{thm}

\subsection{Generalized Jacobians of curves}
\label{LABEL080}

Next we recall the extension of these results to generalized
Jacobians, due to Rosenlicht \cite{Ro1}, Lang \cite{Lang1,Lang2} and
Serre \cite{SeAG}.  Let $S=\{P_i\mid 1\le i\le s\}$ be a nonempty
finite set of points of $X$, and $\fm=\sum m_iP_i$ , $m_i\geq 1$, an
effective divisor (``modulus'') with support $S$. Write $j\colon X
\setminus S\hookrightarrow X$ for the inclusion.

One has two abstract groups attached to the data
$(X,\fm)$ and two algebraic groups over $\CC$.  Set
\begin{equation}
  \Prin_{\fm}(X)=\{\div(\varphi)\mid \varphi\in \kk(X)^\times, \,
  \varphi\equiv1\text{ mod }\fm\}\subset \Div^0(X\setminus S).
\end{equation}
The first abstract group is the divisor ray class group 
\[
  \Cl_\fm(X)\defeq\frac{\Div(X\setminus
    S)}{\Prin_{\fm}(X)}\supset \Cl_\fm^0(X) = \frac{\Div^0(X\setminus
    S)}{\Prin_{\fm}(X)}.
\]
Write $\cK_X^{\times,\fm}$ for the subsheaf of
$\cK_X^\times$ whose sections are nonzero rational functions
which are locally $\equiv 1$~mod~$\fm$, and
$\cO_X^{\times,\fm}=\cO_X^\times \cap \cK_X^{\times,\fm}$.
The generalized Picard group
\[
  \Pic_\fm(X)=H^1(X,\cO_X^{\times,\fm})\supset \Pic_\fm^0(X)
\]
classifies isomorphism classes of line bundles $\mathcal L$ on $X$
together with an $\fm$-rigidification, i.e., a trivialization of the
restriction of $\mathcal L$ to the closed subscheme $T\subset X$
determined by $\fm$.  If $D\in\Div(X\setminus S)$, then $\cO_X(D)$ has
a canonical trivialization $\triv$ along $T$ given by
$\cO_X(D)|_{T}=\cO_X|_{T}=\cO_T$. The map $D \mapsto (\cO_X(D),\triv)$
then gives a degree-preserving isomorphism
$\Cl_\fm(X) \isomarrow \Pic_\fm(X)$, which (up to sign) equals the
boundary map for the exact sequence of Zariski sheaves
\begin{equation}
  \label{LABEL090}
  0 \to \cO_X^{\times,\fm} \to \cK_X^{\times,\fm}
  \xrightarrow{\mathrm{div}} \smash[b]{\coprod_{x\in X \setminus S} x_*\ZZ}
  \to 0
\end{equation}
on $X$.

The generalized Jacobian variety
$\J_{\fm}(X)$ is an algebraic group over $\CC$ which is analytically 
given by \cite[Chap.~V, \S 3, 19]{SeAG}
\begin{equation}
  \label{LABEL100}
  \J_{\fm}(X)=\frac{H^0(X, \Omega^1_X(-\fm))^*}{H_1(X\setminus S,\ZZ)}.
\end{equation}
(But recall that in general the analytic structure on $\J_\fm(X)$ does not
determine its algebraic structure.) If $\fm$ is reduced
(i.e.,~all $m_i=1$), then $\J_\fm(X)$ is an extension of the
$g$-dimensional abelian variety $\J(X)$ by a torus of dimension $s-1$.

Integration gives a homomorphism
$\AJt_\fm \colon \Div^0(X\setminus S) \to \J_\fm(X)$, and one
has the extension of Abel's Theorem to generalized Jacobians:

\begin{thm}
  \label{LABEL105}
  $\AJt_\fm$ induces an isomorphism
  $\AJ_\fm\colon \Cl^0_\fm(X)\isomarrow \J_\fm(X)$.
\end{thm}

We also have an analogue of the exponential sequence \eqref{LABEL050}
with modulus
\begin{equation}
  \label{LABEL110}
  0\to j_!\ZZ(1)\to \cO_X(-\fm)\xrightarrow{\exp}
  \cO_X^{\times,\fm}\to 0 
\end{equation}
inducing an inclusion $H^1_c(X\setminus S,\ZZ(1))=H^1(X,j_!\ZZ(1))
\hookrightarrow H^1(X,\cO_X(-\fm))$. The quotient is then the
underlying complex Lie group of the algebraic group over $\CC$
\begin{equation}
  \label{LABEL120}
  \P_{\fm}(X)=\frac{H^1(X,\cO_X(-\fm))}{H^1_c(X\setminus S,\ZZ(1))}
\end{equation}
and the exponential sequence \eqref{LABEL110} gives an
isomorphism $\P_\fm(X)\isom \Pic^0_\fm(X)$.
\begin{thm}
  \label{LABEL130}
  The following diagram is commutative up to sign
  \begin{equation}
    \label{LABEL140}
    \begin{tikzcd}
      \Cl_\fm(X) \arrow[rrr, "{[D]\mapsto \cO_X(D)}", "{\sim}" below] &&& \Pic_\fm(X)
      \\
      \Cl^0_\fm(X) \arrow[u, hook] \arrow[rrr, "{\sim}"]
      \arrow[d, "{\mathrm{AJ}}" left, "{\wr}"]
      &&& \Pic^0_\fm(X) \arrow[u, hook] 
      \\
      \J_\fm(X) \arrow[rrr, "{\sim}", "{\text{\upshape Serre/Poincar\'e duality}}" below]
      &&& \P_\fm(X) \arrow[u, "{\wr}", "{\mathrm{exp}}" right]
    \end{tikzcd}
  \end{equation}
  where the bottom map is induced by Serre
  duality $H^0(X,\Omega^1_X(\fm)) \times H^1(X,\cO_X(-\fm)) \to \CC$.
\end{thm}
The constructions in the left-hand column of \eqref{LABEL390} are
covariant functorial for morphisms $f\colon X'\to X$ and moduli
$\fm'$, $\fm$ such that $\fm'-f^*\fm\ge0$. Those in the right-hand
column are contravariant functorial for maps $f$ such that
$f^*\fm-\fm'\ge0$.

Both Theorems \ref{LABEL060} and \ref{LABEL130} are well known, but
we are not aware of a direct proof in the literature.  For
completeness we give one in the Appendix.

Finally we have the universal property for generalized Jacobians: 
\begin{thm}
  \label{LABEL150}
  Let $P_0\in X$, $S\subset X\setminus\{P_0\}$ be a finite set of
  points, and $h\colon X\setminus S \to G$ a morphism with $G$ an
  algebraic group over $\CC$ and $h(P_0)=0$. Then there exists a
  modulus $\fm$ supported on $S$, and a unique homomorphism of
  algebraic groups $h'\colon \J_\fm\to G$ such that
  $h(P)=h'(\AJ_\fm([P-P_0]))$.
\end{thm}

The rest of this paper is devoted to the analogue of Theorem
\ref{LABEL130} for graphs. In Section \ref{LABEL620} we will attach to
a ``graph with modulus'' groups analogous to those in the Theorem, and
prove a corresponding Theorem \ref{LABEL840} relating them. The
analogy is perfect except that there is no obvious analogue of the
exponential sequences \eqref{LABEL050}, \eqref{LABEL110} for
graphs; however the following alternative description of 
$\P_\fm(X)$ does have a direct analogue. Let us assume that
$\fm$ is reduced. We have exact sequences
\begin{gather*}
  0 \to H^0(X,\Omega_X^1) \to H^1(X,\CC) \to H^1(X,\cO_X) \to 0
  \\
  0 \to H^0(X,\Omega_X^1) \to H^1_c(X\setminus S,\CC) \to
  H^1(X,\cO_X(-\fm)) \to 0,
\end{gather*}
associated to the exact sequences of sheaves
\begin{gather*}
  0 \to \CC \to \cO_X \xrightarrow{\ d\ } \Omega^1_X \to 0 \\
  0 \to j_!\CC \to \cO_X(-\fm) \xrightarrow{\ d\ } \Omega^1_X \to 0 \\  
\end{gather*}
and we may then rewrite
\begin{equation}
  \label{LABEL160}
  \P(X)=\frac{H^1(X,\CC)}{H^0(X,\Omega_X^1)+H^1(X,\ZZ(1))}
  \xrightarrow[\exp]{\lowsim} \frac{H^1(X,\CC^\times)}{H^0(X,\Omega_X^1)}.
\end{equation}
The map $\P(X) \to \Pic(X)$ is now given simply by the
inclusion of sheaves $q\colon \CC^\times \hookrightarrow
\cO_X^\times$.  In place of the exponential sequence \eqref{LABEL050}
we have the exact sequence
\begin{equation}
  \label{LABEL170}
  0 \to \CC^\times \xrightarrow{\ q\ } \cO_X^\times
  \overset{\mathrm{dlog}}\longrightarrow \Omega_X^1 \to 0
\end{equation}
whose long exact cohomology sequence induces the isomorphism
$q_*\colon \P(X)\isomarrow \Pic^0(X)$.

Likewise we may rewrite
\begin{equation}
  \label{LABEL180}
  \P_\fm(X)=\frac{H^1_c(X\setminus S,\CC)}{H^0(X,\Omega_X^1)+H^1_c(X,\ZZ(1))}
  \xrightarrow[\exp]{\lowsim} \frac{H^1_c(X\setminus S,\CC^\times)}{H^0(X,\Omega_X^1)}.
\end{equation}
and consider the analogue of \eqref{LABEL170} with modulus:
\begin{equation}
  \label{LABEL190}
  0 \to j_!\CC^\times \xrightarrow{\ q\ } \cO_X^{\times,\fm}
  \xrightarrow{\mathrm{dlog}} \Omega_X^1 \to 0.
\end{equation}
The corresponding long exact sequence of cohomology then gives an
isomorphism $q_*\colon \P_\fm(X) \isomarrow
\Pic^0_\fm(X)$. The exact sequences \eqref{LABEL170} and
\eqref{LABEL190} will turn out to have precise analogues
\eqref{LABEL580}, \eqref{LABEL880} for graphs.

\begin{rem}
  We give another description of $\Cl_\fm(X)$ which will be a
  motivation for the definition of the class group $\Cl_\fm(\Gr)$ of a
  graph with modulus (see section \ref{LABEL620}).
  Suppose that $\fm =\sum P_i$ is reduced.  Set
  \[
    \Div_\fm(X)\defeq \Div(X\setminus S)\times\prod_{i=1}^{s}\CC^\times .
  \]
  Define a homomorphism 
  \[
    \{f\in\kk(X)^\times\mid f(P_i)\notin \{0,\infty\}\}\longrightarrow
    \Div_\fm(X)
  \]
  by $f\mapsto (\div(f),(f(P_i))_i)$; the cokernel of this map is $\Cl_\fm(X)$.
\end{rem}


\section{Jacobians and Picard groups of graphs}
\label{LABEL200}
In this section we review definitions and mostly known results on Jacobians
and Picard groups of graphs. Much of these can be found in the papers
\cite{BdN}, \cite{bn1} and \cite{bn2}. Our exposition is designed to
parallel the theory of curves as much as possible --- in particular
Theorem \ref{LABEL460} is a direct analogue for graphs of Theorem
\ref{LABEL060}. Where possible we have also given appropriate
generalizations to infinite graphs.

\subsection{Lattices and their duals}
\label{LABEL210}

Let $L$ be a free $\ZZ$-module of finite rank, and set
$V=L\otimes\QQ$. Let $\pairing{-}{-}\colon L\times L \to \ZZ$ be a
positive definite symmetric bilinear form (``$\ZZ$-valued inner
product"), which we extend by linearity to a$\QQ$-valued bilinear form on $V$.  The \emph{dual
  lattice} is
$L^\#=\{v\in V\mid \pairing vw\in\ZZ \text{ for all } w\in L\}$; it
contains $L$ as a subgroup of finite index, and we have the usual
isomorphism $L^\#\isomarrow L\Zdual \defeq \Hom(L,\ZZ)$ by
$m\mapsto \pairing m-$ for $m\in L^\#$.

The inner product then induces a perfect pairing of finite abelian groups
\begin{equation}
  \label{LABEL220}
  L^\#/L \times L^\#/L\longrightarrow \QQ/\ZZ
\end{equation}
so that there is a canonical isomorphism from $L^\#/L$ to its
Pontryagin dual $\Hom(L^\#/L,\QQ/\ZZ)$.

If $M\subset L$ is a submodule of finite index, then $\pairing{-}{-}$
induces a perfect pairing $L/M \times M^\#/L^\# \to \QQ/\ZZ$.


\subsection{Integral Hodge theory on graphs}
\label{LABEL230}

We begin by recalling familiar elementary facts about graphs and their
homology.  By a graph we mean a tuple $\Gr=(V,E,o,t)$ where $V$, $E$
are sets (of vertices and edges) with $V\ne\emptyset$, and
$o,\,t\colon E \to V$ are maps (origin and terminus). So our graphs
are always oriented, and loops and multiple edges are
permitted\footnote{This is the definition of \cite[\S2.1]{SeT}.}.  To
a graph $\Gr$ we can attach its geometric realization $\Grg$, which is
the CW complex obtained from $E\times [0,1] \,\sqcup\, V$ by
identifying vertices in the obvious way. Graphs will always be assumed
to be locally finite (meaning that all fibres of the maps $o$ and $t$
are finite) and countable. Often we shall require $\Gr$ to be finite
(i.e.~$V$ and $E$ are finite) and connected. We write $\pi_0(\Gr)$ for
the set of conected components of $\Gr$.

We define the degree of a vertex $v\in V$ to be
$\deg(v)=\# o\i(v) + \# t\i(v)$ (so that a loop contributes $2$
to the degree). A vertex $v$ is \emph{isolated} if $\deg(v)=0$.

As $\Gr$ is a complex of dimension 1, the integral homology
$H_*(\Grg)=H_*(\Gr)$ is the homology of the standard chain complex
\[
  C_0(\Gr) = \ZZ[V] \xleftarrow{\ \partial\ } C_1(\Gr) = \ZZ[E],\qquad
  \partial=t-o
\]
and its cohomology $H^*(\Grg,A)=H^*(\Gr,A)$ with coefficients in any
abelian group $A$ is the cohomology of
\begin{equation}
  \label{LABEL240}
  C^0(\Gr,A) = A^V \xrightarrow{d=\tr\partial} C^1(\Gr,A) = A^E,\qquad
  df=f\circ\partial.
\end{equation}
We will simply write $C^\bullet(\Gr)=C^\bullet(\Gr,\ZZ)$, and
similarly for cohomology groups. We write $\apairing[p]{-}{-} \colon
C_p(\Gr)\times C^p(\Gr) \to \ZZ$ for the pairing.

The complex $C_\bullet(\Gr)$ splits in degree $0$, since for any
subset $S\subset V$ such that the map $S \to \pi_0(\Gr)$ is bijective,
$\ZZ[S]$ is a complement to $\im(\partial)$. It splits in degree $1$
since $\im(\partial)$ is free, whence $\partial \colon C_1(\Gr) \to
\im(\partial)$ has a section. In other words, we have splittings
\begin{align}
  \label{LABEL250}
  C_\bullet(\Gr) &= [\, \ZZ[S]\oplus \im(\partial) \longleftarrow
                   \coim(\partial)\oplus H_1(\Gr)\,] \\
  \label{LABEL260}
  C^\bullet(\Gr,A) &= [\, A^S \oplus \coim(d) \longrightarrow
                     \im(d)\oplus H^1(\Gr,A)\,]
\end{align}
and therefore the abelian groups $H_p(\Gr)$ are free, and for every $A$,
$H^p(\Gr,A)=\Hom(H_p(\Gr),A)$. If $\Gr$ is finite, then moreover
$H^p(\Gr,A)=H^p(\Gr)\otimes_\ZZ A$.

In particular, since both $H_0(\Gr)$ and $H_1(\Gr)$ are free of
countable rank, we have by Specker's theorem \cite{Sp} the reflexivity
$H_p(\Gr) \isomarrow H^p(\Gr)\Zdual$.

There is a canonical $\ZZ$-valued inner product
$\pairing{-}{-}=\pairing[p]{-}{-}$ on $C_p(\Gr)$ for which an
orthonormal basis is $V$ (resp.~$E$) if $p=0$ (resp.~$1$). We write
$\iota=\iota_p \colon C_p(\Gr) \hookrightarrow C^p(\Gr)$ for the
homomorphisms taking a vertex or edge to its characteristic
function. We often write $\hat x$ for $\iota(x)$. Then for every $x$,
$y\in C_p(\Gr)$,
\begin{equation}
  \label{LABEL270}
  \pairing[p]{x}{y} = \apairing[p]{x}{\hat y}.
\end{equation}
If $\Gr$ is finite then the maps $\iota_p$ are isomorphisms, and
induce inner products on the groups $C^p(\Gr)$.

Since $\Gr$ is locally finite, the homomorphism
\[
  \partial\adj \colon C_0(\Gr) \to C_1(\Gr),\quad v \mapsto
  \sum_{t(e)=v}e-\sum_{o(e)=v}e
\]
is well defined, and is adjoint to $\partial$, in the sense
that for all $D\in C_0(\Gr)$, $\gamma\in C_1(\Gr)$, we have
$\pairing[0]{\partial \gamma}{D}=\pairing[1]{\gamma}{\partial\adj D}$.
Write $d\adj=\tr{(\partial\adj)}$ for its transpose (which is adjoint
to $d$ is $\Gr$ is finite).

The homological and cohomological Laplacians
$\Delta_p\in\End C_p(\Gr)$, $\Box_p\in \End C^p(\Gr)$ are defined as
\[
  \Delta_0 = \partial\partial\adj,\quad \Delta_1=\partial\adj\partial,\quad
  \Box_0 = d\adj d =\tr{\Delta_0},\quad \Box_1=dd\adj =\tr{\Delta_1}.
\]
Explicitly, if $v\in V$, then
$\Delta_0(v)=\sum_{t(e)=v}(v-o(e)) +\sum_{o(e)=v}(v-t(v))$. In
particular, $\Delta_0$ and its transpose $\Box_0$ are independent of
the orientation of $\Gr$.

Since $\iota$ takes adjoint to transpose, we have
\begin{lem}
  \label{LABEL280}
  \begin{enumerate}[\upshape(a)]
  \item $\iota\circ \partial = d\adj\circ\iota$,
    $\iota\circ\partial\adj = d\circ\iota$, and
    $\iota\circ\Delta_p=\Box_p\circ\iota$ \textup($p=0$, $1$\textup).
  \item If $f\in C^p(\Gr)$ and $x\in C_p(\Gr)$, then
    $\Box_pf(x)=f(\Delta_px)$.
    \qed
  \end{enumerate}    
\end{lem}

\begin{defn}
  \begin{enumerate}[\upshape(a)]
  \item The group of harmonic functions on $\Gr$ is
    $\Ha^0(\Gr)=\ker(\Box_0) \subset \ZZ^V$.
  \item \cite[\S4.3]{bn2} A 1-cochain $\omega\in C^1(\Gr)$ is said to
    be harmonic if $d\adj\omega=0$. The group of harmonic $1$-cochains
    (or 1-forms) is denoted $\Ha^1(\Gr)$.
  \end{enumerate}
\end{defn}

\begin{rem}
  Obviously $\Ha^1(\Gr) \subset \ker(\Box_1)$, and equality holds when
  $\Gr$ is a finite graph, by Hodge theory (Proposition \ref{LABEL320}(a)
  below). It seems that for infinite graphs, it is $\ker(d\adj)$ rather
  than $\ker(\Box_1)$ which has better properties, and so following
  \cite{bn2} we have taken this
  as our definition for $\Ha^1$ in general.
\end{rem}

\begin{lem}\label{LABEL290}
  The subgroup $\im (\partial\adj) \subset C_1(\Gr)$ is a direct
  summand, and the complexes $(C_\bullet(\Gr), \partial\adj)$ and
  $(C^\bullet(\Gr), d\adj)$ are split.
\end{lem}

\begin{proof}
  If $\Gr$ is finite then $\iota$ gives isomorphisms of complexes
  \[
    (C_\bullet(\Gr),\partial\adj) \isomarrow (C^\bullet(\Gr), d), \text{
      and } (C_\bullet(\Gr),\partial) \isomarrow (C^\bullet(\Gr), d\adj)
  \]
  by Lemma \ref{LABEL280}, so by \eqref{LABEL250} and \eqref{LABEL260} we
  obtain the desired splittings. In general, Theorem \ref{LABEL1160}
  below (which does not depend on this Lemma) shows that there is a
  subset $E'\subset E$ such that
  $C_1(\Gr)=\im(\partial\adj) \oplus \ZZ[E']$. As $\im(\partial\adj)$
  is free, the complex $(C_\bullet(\Gr),\partial\adj)$ also splits in
  degree $0$, and hence the dual complex $(C^\bullet(\Gr), d\adj)$
  splits as well.
\end{proof}

\begin{lem}
  \label{LABEL300}
  \begin{enumerate}[\upshape(a)]
  \item $\partial C_1(\Gr)$ and $\ker (d)=H^0(\Gr)$ are mutual
    orthogonal complements with respect to $\apairing[0]{-}{-}$.
  \item $\partial\adj C_0(\Gr)$
    and $\ker (d\adj)\subset C^1(\Gr)$ are mutual orthogonal
    complements with respect to $\apairing[1]{-}{-}$.
  \end{enumerate}
\end{lem}

\begin{proof}
  (a) Since $\apairing[0]{\partial x}{f}=\apairing[1]{x}{df}$ for
  $x\in C_1(\Gr)$, $f\in C^0(\Gr)$, the orthogonal complement to
  $\im(\partial)$ is $\ker(d)$. As $\im(\partial)$ is a direct summand
  of $C_0(\Gr)$, the orthogonal complement of $\ker(d)$ is
  $\im(\partial)$.

  (b) Same argument with $(\partial\adj,d\adj=\tr{(\partial\adj)})$,
  using Lemma \ref{LABEL290}.
\end{proof}

\begin{defn}
  \label{LABEL310}
  Write $\beta\colon \Ha^1(\Gr) \to H^1(\Gr)$ for the composite
  $\Ha^1(\Gr) \hooklongrightarrow C^1(\Gr) \to H^1(\Gr)$. Write
  $\alpha\colon H_1(\Gr) \to \Ha^1(\Gr)\Zdual$ for the transpose of
  $\beta$.
\end{defn}

Suppose now that $\Gr$ is finite. Then $\iota_p\colon C_p \to C^p$ is an isomorphism for $p=0$,
$1$, and we have inner products $\pairing{-}{-}=\pairing{-}{-}^p$ on
$C^p(\Gr)$ such that
\[
  \pairing[p] xy = \pairing{\hat x}{\hat y}^p\qquad(x,\ y\in C_p(\Gr))
\]
and for which the sets $\iota(V)$, $\iota(E)$ are orthonormal bases.

Let $(-)^\perp$ denote orthogonal complement with respect to
$\pairing{-}{-}$.

\begin{prop}
  \label{LABEL320}
  Assume that $\Gr$ is finite. Then:
  \begin{enumerate}[\upshape(a)]
  \item $\ker (\Box_0)=\ker(d)=H^0(\Gr)\subset C^0(\Gr)$, and
    $\Ha^1(\Gr) \defeq \ker (d\adj) = \ker(\Box_1) \subset C^1(\Gr)$.
  \item $\iota_1$ induces an isomorphism
    $\iota\colon H_1(\Gr) \isomarrow \Ha^1(\Gr)$.
  \item $(d\adj C^1(\Gr))^\perp = H^0(\Gr)$ and
    $H^0(\Gr)^\perp= d\adj C^1(\Gr)$.
  \item $(d C^0(\Gr))^\perp = \Ha^1(\Gr)$ and
    $\Ha^1(\Gr)^\perp= d C^0(\Gr)$.
  \item $\beta$ is injective with finite cokernel.
  \item $\alpha$ is injective with finite cokernel.
  \item The map $H_1(\Gr) \to H^1(\Gr)$, taking $\gamma\in
    H_1(\Gr)\subset C_1(\Gr)$ to the class of $\hat\gamma$, is
    injective with finite cokernel, and with image $\beta(\Ha^1(\Gr))$.
  \item If $\Gr$ is connected, then $\im (d\adj)=C^0(\Gr)^0$ and
    $\coker (d\adj) \isom \ZZ$.
  \end{enumerate}
\end{prop}

\begin{proof}
  (a) is standard Hodge theory: for the first statement, if $f\in
  C^0(\Gr)$,  then
  \[
    \Box_0 f= d\adj df=0
    \implies 0=\inner{d\adj df}f = \inner{df}{df} \iff df =0.
  \]
  (The other inclusion is trivial.) The proof of the second statement
  is the same.

  Then $\iota\colon H_1(\Gr)=\ker (\partial) \isomarrow \ker (d\adj)
  =\Ha^1(\Gr)$ by Lemma \ref{LABEL280}(i), giving (b).

  By (a) and Lemma \ref{LABEL300}(a),  $\partial C_1(\Gr)$
  and $H^0(\Gr)$ are mutual orthogonal complements with respect to
  $\pairing{-}{-}$, and so by Lemma \ref{LABEL280}, (c) follows.

  The same argument using Lemma \ref{LABEL300}(b) gives (d). In
  particular,
  \[
  C^1(\Gr,\QQ) = dC^0(\Gr,\QQ)\oplus (\Ha^1(\Gr)\otimes\QQ),
  \]
  so $\Ha^1(\Gr)\otimes\QQ\isomarrow H^1(\Gr,\QQ)$.  Since $C^1(\Gr)$
  is finitely generated, (e) then follows, and by duality we have
  (f). For (g), combine (b) and (e). Finally, (h) follows from Lemma
  \ref{LABEL280}(a) and the corresponding statements for $\partial$.
\end{proof}


\subsection{Class groups and Jacobians of graphs; Abel's theorem}
\label{LABEL330}
 
By analogy with curves, the divisor group of the graph $\Gr$ is
defined \cite{bn1} to be
\[
  \Div(\Gr)\defeq C_0(\Gr)=\ZZ[V]
\]
and if $\Gr$ is connected, the degree zero divisor group to be
\[
  \Div^0(\Gr)\defeq \ZZ[V]_0.
\]
We will also define the codivisor group to be
$\Divhat(\Gr)=C^0(\Gr)=\ZZ^V$. If $\Gr$ is finite and connected, we set
$\Divhat^0(\Gr)=\ZZ^{V,0}$, so that
\[
  \deg\colon
  \Divhat(\Gr)/\Divhat^0(\Gr)\isomarrow\ZZ.
\]
In general, we define
$\Divhat^0(\Gr)=\im d\adj \subset \ZZ^V$. By Proposition
\ref{LABEL320}(h) these definitions agree for $\Gr$ finite and
connected.

Write $\Prin(\Gr)=\Delta_0(\Div(\Gr))$. From the explicit formula for
$\Delta_0$, if $\Gr$ is connected then $\Prin(\Gr)\subset \Div^0(\Gr)$.

If $h\colon V \to A$ is any map to an abelian group, we denote also by $h$ its
$\ZZ$-linear extension $\ZZ[V] \to A$.

\begin{defn}
  \begin{enumerate}[(a)]
  \item A \emph{harmonic map} from $\Gr$ to an abelian group $A$ is a map
    $h\colon V \to A$ such that $h\circ\Delta_0=0$.
  \item The \emph{divisor class group} of $\Gr$ is the quotient
    $\Cl(\Gr)\defeq\Div(\Gr)/\Prin(\Gr)$. If $\Gr$ is
    connected, the \emph{degree zero divisor class group}
    is $\Cl^0(\Gr)\defeq \Div^0(\Gr)/\Prin(\Gr)$. 
  \end{enumerate}
\end{defn}

Equivalently, $h\colon V \to A$ is a harmonic map if $h$ is an element
of
\[
  \Ha^0(\Gr,A)\defeq \ker(\Box_0\otimes\id_A\colon A^V \to
  A^V).
\]
Observe that $\Ha^0(\Gr,A)\ne \Ha^0(\Gr,\ZZ)\otimes_\ZZ A$ in
general. For example, if $\Gr$ is an $n$-gon with $n\ge 3$, and
$A=\ZZ/n\ZZ$, then
$\Ha^0(\Gr,A)\isom A^{n-1} \ne \Ha^0(\Gr,\ZZ)\otimes_\ZZ A=A$.

Likewise, we define $\Prinhat(\Gr)=\Box_0(\Divhat(\Gr))$ and
$\Clhat(\Gr)=\Divhat(\Gr)/\Prinhat(\Gr)$,
$\Clhat^0(\Gr)=\Divhat^0(\Gr)/\Prinhat(\Gr)$. The map $\iota_0$
induces by Lemma \ref{LABEL280} a homomorphism
$\iota\colon \Cl(\Gr) \to \Clhat(\Gr)$, which is an isomorphism if
$\Gr$ is finite.  If $\Gr$ is connected, then $\Div^0(\Gr)=\im\partial$
so by Lemma \ref{LABEL280}, $\iota$ restricts to a homomorphism
\begin{equation}
  \label{LABEL340}
  \iota^0\colon \Cl^0(\Gr) \to \Clhat^0(\Gr)
\end{equation}
which is an isomorphism for $\Gr$ finite and connected.

\begin{rem*}
  The introduction of $\Clhat(\Gr)$ may seem pedantic and superfluous,
  but for infinite graphs, it is $\Clhat(\Gr)$, rather than
  $\Cl(\Gr)$, that is isomorphic to the Picard group of harmonic
  torsors on $\Grg$ --- see section \ref{LABEL470} below).
\end{rem*}

Essentially by definition, $\Cl(\Gr)$ and $\Cl^0(\Gr)$ are
universal for harmonic maps to abelian groups:

\begin{prop}
  \label{LABEL350}
  Let $v_0\in V$. Then there are bijections
  \[
    \{\text{harmonic  }h\colon V\to A\} \isomarrow
    \Hom(\Cl(\Gr),A)
  \]
  and, for $\Gr$ connected,
  \[
    \pushQED{\qed}
    \{\text{harmonic  }h\colon V\to A \mid h(v_0)=0\} \isomarrow
    \Hom(\Cl^0(\Gr),A)\,. 
    \qedhere
    \popQED
  \]
\end{prop}

\begin{rem}
  The component group as determined by Raynaud \cite{Ra} is $\Cl^0(\Gr)$ for
  $\Gr$ the dual graph of the special fibre of a semistable curve (see Introduction).
\end{rem}

\begin{defn}
\label{LABEL360}
The \emph{Jacobian} $\J(\Gr)$ of the graph $\Gr$ is the group
\[ 
\J(\Gr)=\Ha^1(\Gr)\Zdual/\alpha(H_1(\Gr))
\]
with $\alpha$ as in Definition \ref{LABEL310}.
\end{defn}

In the next Proposition, $\Ha^1(\Gr)^\#\subset
\Ha^1(\Gr)\otimes\QQ\subset C^1(\Gr,\QQ)$ denotes the dual lattice
(with respect to the restriction of the inner product
$\pairing{-}{-}^1$ on $C^1(\Gr)$ to $\Ha^1(\Gr)$) as in Section
\ref{LABEL210}.

\begin{prop}
  \label{LABEL370}
  Assume that $\Gr$ is finite.
  \begin{enumerate}[\upshape(a)]
  \item $\J(\Gr)$ is finite.
  \item The isomorphism $\Ha^1(\Gr)^\# \isomarrow \Ha^1(\Gr)\Zdual$
    induces an isomorphism
    \[
      \Ha^1(\Gr)^\#/\Ha^1(\Gr) \isomarrow \J(\Gr).
    \]
  \item There is a canonical isomorphism $\J(\Gr)\isomarrow
    \Hom(\J(\Gr),\QQ/\ZZ)$.
  \end{enumerate}
\end{prop}

\begin{proof}
By  Proposition \ref{LABEL320}(b,f), $\alpha$ is injective with
finite cokernel and the following
diagram commutes:
\[
  \begin{tikzcd}
    \Ha^1(\Gr)^\# \arrow[r, "x\mapsto \pairing x-", "\sim"'] &
    \Ha^1(\Gr)\Zdual \\
    \Ha^1(\Gr) \arrow[u, hook] & H_1(\Gr) \arrow[l, "\iota"', "\sim"]
    \arrow[u, hook, "\alpha"]
  \end{tikzcd}
\]
from which (a) and (b) follow. By \eqref{LABEL220}, we obtain the autoduality
(c).
\end{proof}

\begin{rem}
  There are various, different but obviously equivalent, definitions of
  $\J(\Gr)$. In \cite[p.177]{BdN} it is defined to be
  $\Ha^1(\Gr)^\#/\Ha^1(\Gr)$ (cf.~Proposition \ref{LABEL370}(a)). The
  component group as defined by Grothendieck \cite{Gr} (see Introduction) is the
  quotient $H^1(\Gr)/H_1(\Gr)$ (cf.~Proposition \ref{LABEL320}(g)).
\end{rem}

\textbf{Until the end of this subsection, we assume that $\Gr$ is
  connected}.

Let $D\in\Div^0(\Gr)$. Then there exists $\gamma_D\in C_1(\Gr)$,
unique modulo~$H_1(\Gr)$, with $\partial(\gamma_D)=D$.  Let
$\lambda_D\in\Ha^1(\Gr)\Zdual$ be the map
$\omega\mapsto\omega(\gamma_D)$; $\lambda_D$ is then well-defined
modulo~$\alpha(H_1(\Gr))\subset \Ha^1(\Gr)\Zdual$.

\begin{defn}
\label{LABEL380}
The \emph{Abel--Jacobi map} $\AJt\colon \Div^0(\Gr)\to \J(\Gr)$ is
the map
\[
  \AJt\colon D\mapsto \lambda_D\ \mathrm{mod}\ \alpha(H_1(\Gr)).
\]
\end{defn}

We have a commutative diagram with exact rows
\\[-2mm]
\begin{equation}
  \label{LABEL390}
  \begin{tikzcd}
    0 \arrow[r] & \ker(\partial)=H_1(\Gr) \arrow[d, twoheadrightarrow] \arrow[r] &
    C_1(\Gr)=C^1(\Gr)\Zdual \arrow[d, twoheadrightarrow, "\pi"]
    \arrow[r, "\partial"] &
    \Div^0(\Gr) \arrow[r] \arrow[d, "\AJt"] & 0
    \\
     0 \arrow[r] & H_1(\Gr)/\ker(\alpha) \arrow[r, "\alpha"] & \Ha^1(\Gr)\Zdual \arrow[r] &
    \J(\Gr)=\dfrac{\Ha^1(\Gr)\Zdual}{\alpha(H_1(\Gr))} \arrow[r] & 0
  \end{tikzcd}
\end{equation}
where $\pi$ is induced by the inclusion $\Ha^1(\Gr)\subset
C^1(\Gr)$. By Lemma \ref{LABEL290}, $\Ha^1(\Gr)$ is a direct summand
of $C^1(\Gr)$, and so $\pi$ is surjective.

The next result is Abel's Theorem for graphs, proved in the finite
case in \cite[Prop.~7(iii)]{BdN}.
\begin{thm}
\label{LABEL400}
The map $\AJt\colon \Div^0(\Gr)\to \J(\Gr)$ is
surjective with kernel $\Prin(\Gr)$, inducing an isomorphism
$\AJ\colon \Cl^0(\Gr)\isomarrow \J(\Gr)$.
\end{thm}

\begin{proof}
By the commutativity of diagram \ref{LABEL390}, $\AJt$ is
surjective, and by the snake lemma, the map $\partial\colon \ker(\pi) \to
\ker(\AJt)$ is surjective. Now
\[
  \ker(\pi) =
  \{ \gamma\in C_1(\Gr)\mid \apairing[1] \gamma {\Ha^1(\Gr)} = 0 \} 
\]
and since $\Ha^1(\Gr)=\ker (d\adj)$, by
Lemma \ref{LABEL300}(b) we have $\ker (\pi) = \im (\partial\adj)$. Therefore
$\partial(\ker \pi)=\Delta_0(C_0(\Gr))= \Prin(\Gr)$. 
\end{proof}


\subsection{Picard groups and duality}

We now describe the analogue of the Picard variety and Picard group
for a graph $\Gr$. Both of these are abelian groups, finite if $\Gr$
is finite.

First we define the analogue of the Picard variety of a curve, which
is what is called in \cite{BdN} the Picard group of $\Gr$.

Recall from Definition \ref{LABEL310} the  map $\beta\colon
\Ha^1(\Gr) \to H^1(\Gr)$. By Proposition \ref{LABEL320}(e), for $\Gr$
finite, $\beta$ is injective with finite cokernel.

\begin{defn}
  \label{LABEL405}
  The Picard group of $\Gr$ is the group
  $\P(\Gr)\defeq H^1(\Gr)/\beta(\Ha^1(\Gr))$.
\end{defn}

\begin{rem}
  This is the definition of the group given in
  \cite[Defn.~2.12]{jl}. By Proposition \ref{LABEL320}(b), for $\Gr$
  finite an equivalent definition is $H^1(\Gr)/\iota(H_1(\Gr))$.
\end{rem}

\begin{prop}
  \label{LABEL410}
  The homomorphism $d\adj\colon C^1(\Gr) \to C^0(\Gr)$ induces
    an isomorphism $\comp\colon \P(\Gr) \isomarrow \Clhat^0(\Gr)\subset
    \Clhat(\Gr)$.
\end{prop}

\begin{proof}
  Since $\im (d\adj)=C^0(\Gr)$ and $\Ha^1(\Gr)=\ker (d\adj)$, we have
  \[
    \P(\Gr) = \frac{C^1(\Gr)}{dC^0(\Gr) + \ker (d\adj)}
    \isomarrowby{\comp}
    \frac{C^0(\Gr)^0}{d\adj dC^0(\Gr)} = \Clhat^0(\Gr).
    \qedhere
  \]
\end{proof}

\begin{prop}
  \label{LABEL420}
  There is a canonical homomorphism
  $\zeta\colon \J(\Gr)\to \P(\Gr)$ \textup(described in the proof
  below\textup). If $\Gr$ is finite, $\zeta$ is an isomorphism.
\end{prop}

\begin{proof}
  By Lemma \ref{LABEL280}, $\iota_1$ induces a map $\iota\colon
  H_1(\Gr) \to \Ha^1(\Gr)$. The diagram
  \[
    \begin{tikzcd}
      \Ha^1(\Gr)\Zdual \arrow[r, "\tr\iota"] & H^1(\Gr)=H_1(\Gr)\Zdual
      \\
      H_1(\Gr) \arrow[u, "\alpha"] \arrow[r, "\iota"] &
      \Ha^1(\Gr) \arrow[u, "\beta"]
    \end{tikzcd}
  \]
  is easily seen to commute, and therefore induces the desired
  homomorphism $\zeta$,  which by Proposition \ref{LABEL320}(b) is an
  isomorphism for $\Gr$ finite.
\end{proof}

There an alternative description of $\zeta$ when $\Gr$ is
finite, using the autoduality of Proposition \ref{LABEL370}. Consider the pairing
\[
\Ha^1(\Gr)\dual \times C^1(\Gr) \isom \Ha^1(\Gr)^\# \times C^1(\Gr)
\subset C^1(\Gr)_\QQ \times C^1(\Gr)_\QQ 
  \xrightarrow{\pairing{-}{-}^1\otimes\QQ} \QQ.
\]
Since $\alpha$ and $\beta$ are injections (Proposition \ref{LABEL320}),
this pairing induces (see Section \ref{LABEL210}) a perfect pairing
$\J(\Gr) \times \P(\Gr) \to \QQ/\ZZ$. Combining this with the
autoduality of $\J(\Gr)$ from Proposition \ref{LABEL370}(b) gives an isomorphism
$\theta\colon \P(\Gr) \isomarrow \J(\Gr)$.

Explicitly: the composite map
\begin{equation}
\label{LABEL430}
C^1(\Gr)\isomarrowby{\iota\i}C_1(\Gr) = C^1(\Gr)\Zdual
\twoheadrightarrow \Ha^1(\Gr)\Zdual
\end{equation}
is surjective (since $\Ha^1(\Gr)\subset C^1(\Gr)$ is a direct
summand).  Its kernel is the orthogonal complement of $\Ha^1(\Gr)$ in
$C^1(\Gr)$, which is $dC^0(\Gr)$ by Proposition \ref{LABEL320}(d).
Therefore the map \eqref{LABEL430} induces an isomorphism
$\tilde\theta\colon H^1(\Gr)\isomarrow \Ha^1(\Gr)\Zdual$, inducing the
isomorphism $\theta\colon \P(\Gr) \isomarrow \J(\Gr)$.

It is clear from this description that for $\Gr$ finite, $\zeta=\theta\i$.

\begin{rem}
  Observe the analogy between $\P(\Gr)$ and the alternative
  description \eqref{LABEL160} of the Picard variety of a curve.
\end{rem}

The next result is a partial analogue of Theorem \ref{LABEL060} for
graphs.

\begin{thm}
  \label{LABEL440}
  Let $\Gr$ be a connected graph.
  The diagram
  \[
    \begin{tikzcd}
       \Cl(\Gr) \arrow[r, "\iota"] & \Clhat(\Gr)
      \\
      \Cl^0(\Gr) \arrow[u, hook] \arrow[d, "\AJ"', "\wr"]
      \arrow[r, "\iota^0"]
      & \Clhat{}^0(\Gr) \arrow[u, hook]
      \\
      \J(\Gr) \arrow[r, "\zeta"]
      & \P(\Gr) \arrow[u, "\wr", "\comp"']
    \end{tikzcd}
  \]
  commutes, and if $\Gr$ is finite, all the horizontal maps are isomorphisms.
\end{thm}

\begin{proof}
  The commutativity of the top square follows from the definition
  \eqref{LABEL340}. For the bottom square, let $D\in\Div^0(\Gr)$,
  $D=\partial \gamma$ for $\gamma\in C_1(\Gr)$. By Proposition
  \ref{LABEL280}(a), $\hat D\defeq \iota(D)=d\adj(\hat\gamma)\in\Divhat^0(\Gr)$. By
  definition, $\AJ(D)$ is the image of
  $(\omega\mapsto \omega(\gamma))\in\Ha^1(\Gr)\Zdual$.  So we have
  \[
    \begin{tikzcd}
      D\ \mathrm{mod} \ \Prin(\Gr)\arrow[r, mapsto, "\iota^0"] 
      \arrow[d, mapsto, "\AJ"]
      &\hat{D}\ \mathrm{mod}\ \Prinhat(\Gr)\\
      (\omega\mapsto \omega(\gamma)) \ \mathrm{mod}\ \alpha(H^1(\Gr))
      \arrow[r, mapsto, "\zeta"]
      &
      \hat\gamma \ \mathrm{mod}\ \beta(\Ha^1(\Gr))
      \arrow[u, mapsto, "\comp"'] 
    \end{tikzcd}
  \]
  and the square commutes. If $\Gr$ is finite then
  $\iota$ and $\iota^0$ are isomorphisms, hence so is $\zeta$.
\end{proof}

\begin{rem}
  The reader will have noticed that the fact that the complex
  $C_\bullet(\Gr)$ is the chain complex of a graph barely enters the
  proof. The same arguments apply almost unaltered to any 2-term
  complex $C_0\xleftarrow{\ \partial\ } C_1$, where $C_i$ are
  unimodular lattices and $\partial$ is any homomorphism,
  and indeed in greater generality. This is discussed further in
  section \ref{LABEL1080} below.
\end{rem}

For a complete analogue of Theorem \ref{LABEL060}, we need an analogue
of the Picard group of line bundles on a curve. In the next section we
will define, for any locally finite graph $\Gr$ with no isolated vertices, groups $\Pic^0(\Grg) \subset \Pic(\Grg)=H^1(\Grg,\Harm)$
classifying torsors for a sheaf $\Harm$ on the geometric realization
$\Grg$ of $\Gr$. There are canonical
isomorphisms $q_*\colon \P(\Gr) \isomarrow \Pic^0(\Grg)$ and
$\bar\delta\colon \Clhat(\Gr) \isomarrow \Pic(\Grg)$ fitting into a commutative
diagram (Theorem \ref{LABEL600} below)
\begin{equation}
  \label{LABEL450}
  \begin{tikzcd}
    \Clhat(\Gr) \arrow[r, "-\bar\delta", "\sim"']& \Pic(\Grg)
    \\
    \Clhat^0(\Gr) \arrow[u, hook]\arrow[r, "\lowsim"]
    &\Pic^0(\Grg) \arrow[u, hook] 
    \\
    & \P(\Gr) \arrow[u, "\wr", "q_*"']
    \arrow[ul, "\comp", "\isom"']
  \end{tikzcd}
\end{equation}
Combining \eqref{LABEL450} and Theorem \ref{LABEL440}, we obtain the
following analogue of Theorem \ref{LABEL060}.

\begin{thm}
  \label{LABEL460}
  Let $\Gr$ be a connected graph. The diagram
  \[
    \begin{tikzcd}
      \Cl(\Gr) \arrow[r, "\iota"] & \Clhat(\Gr)  \arrow[r, "-\bar\delta", "\sim"']
      & \Pic(\Grg)
      \\
      \Cl^0(\Gr) \arrow[u, hook] \arrow[d, "\AJ" left, "\wr"]
      \arrow[r, "\iota^0"]
      & \Clhat{}^0(\Gr) \arrow[u, hook]  \arrow[r, "\lowsim"]
      &\Pic^0(\Grg) \arrow[u, hook]
      \\
      \J(\Gr) \arrow[rr, "\zeta"]
      && \P(\Gr) \arrow[u, "\wr", "q_*"']
      \arrow[ul, "\comp", "\isom"']
    \end{tikzcd}
  \]
  commutes. If $\Gr$ is finite then the horizontal arrows are all
  isomorphisms.\qed
\end{thm}

\subsection{Sheaves and the geometric Picard group}
\label{LABEL470}

In this section we assume that $\Gr=(V,E,o,t)$ is a locally finite graph with no
isolated vertices.

Recall that the geometric realization $\Grg$ is the quotient of
the space $V\,\sqcup\,(E\times[0,1])$ obtained by identifying, for every
$e\in E$, $o(e)\in V$ with $(e,0)\in E\times [0,1]$ and $t(e)\in V$ with
$(e,1)$. Write $\tilde e\colon [0,1] \to \Grg$ for the obvious
map (which is an inclusion unless $e$ is a loop), and $x$ for the
coordinate function on $[0,1]$. We endow $\Grg$ with the obvious
metric for which the distance between adjacent vertices is $1$, and
for $r\in (0,1]$ and $v\in V$ write $U_v(r)$ for the open ball of
radius $r$ about $v$. If $\Gr'$ is the graph obtained from $\Gr$ by
reversing the orientation of some of the edges of $\Gr$, then there is
a canonical isometry $\Grg \isomarrow \abs{\Gr'}$. 

Following the discussion in \cite{MW1}, especially \S2.3 and \S3, we
define sheaves $\Harm\subset \PL$ on the topological space $\Grg$, analogous to the
sheaves $\cO^\times_X\subset \cK^\times_X$ on a curve $X$.

We will only need to consider sheaves $\cF$ on $\Grg$ satisfying the
following property:
\begin{itemize}
\item[(C)] For every $e\in E$, the restriction of $\cF$ to the open edge
  $\{e\}\times (0,1) \subset \Grg$ is constant.
\end{itemize}
(In the setting of simplicial complexes these are sometimes called
``cellular sheaves'', see for example \cite{She}, where they are
extensively studied.) Any such sheaf $\cF$ admits a simple
combinatorial description. For $v\in V$, let $\cF_v$ be the stalk of
$\cF$ at $v$, and for $e\in E$, let $\cF_e = H^0(\{e\}\times (0,1),
\cF)$.  Assumption (C) implies:
\begin{itemize}
\item for any $r\in (0,1/2)$, the restriction map
  $H^0(U_v(r),\cF) \to \cF_v$ is an isomorphism; and
\item for every $a$, $b$ with $0\le a<b\le 1$, the restriction map
  $\cF_e \to H^0(\{e\}\times (a,b),\cF)$ is an isomorphism.
\end{itemize}
Then for each $e\in E$, the inclusions\footnote{We use radius
  $1/3$ rather than the more natural $1$ to deal with loops.}
\[
  U_{o(e)}(1/3) \hookleftarrow \{e\}\times (0,1/3) \hookrightarrow 
  \{e\}\times (0,1)  \hookleftarrow \{e\}\times (2/3,1) \hookrightarrow
  U_{t(e)}(1/3)
\]
induce maps
\begin{equation}
  \label{LABEL480}
  \xi^0_e \colon \cF_{o(e)} \to \cF_e,\quad
  \xi^1_e \colon \cF_{t(e)} \to  \cF_e.
\end{equation}

\begin{prop}
  \label{LABEL490}
  \begin{enumerate}[\upshape(a)]
  \item The above construction gives an equivalence between the
    category of abelian sheaves on $\Grg$ satisfying \textup{(C)}, and the
    category of families of abelian groups $(\cF_v)_{v\in V}$,
    $(\cF_e)_{e\in E}$ and maps \eqref{LABEL480}.
  \item Let $\cF$ be a sheaf on $\Grg$ satisfying \textup{(C)}. There are
    functorial isomorphisms for $i=0$, $1$
    \[
      H^i(\Grg, \cF) \isomarrow H^i(C^\bullet(\Gr, \cF),\, d_{\cF}),
    \]
    where $(C^\bullet(\Gr, \cF),\, d_{\cF})$ is the complex
    \begin{align*}
      \prod _{v\in V} \cF_v &\xrightarrow{\,d_{\cF}\,} \prod_{e\in E}
                              \cF_e\\
      (a_v)_{v\in V} &\ \mapsto\ (b_e)_{e\in E},\quad b_e =
                       \xi^1_e(a_{t(e)}) - \xi^0_e(a_{o(e)}),
    \end{align*}
    compatible with the long exact sequence of cohomology.
  \end{enumerate}
\end{prop}

\begin{proof}
  (i) is straightforward.
  
  (ii) If $\Gr$ has no loops, then $C^\bullet(\Gr,\cF)$ is isomorphic
  to the \v Cech complex of $\cF$ for the covering by opens
  $U_v=U_v(2/3)$, $v\in V$. In more detail, one has
  $H^0(U_v,\cF)=\cF_v$ and for $v\ne w$,
  $H^0(U_v\cap U_w,\cF) = \prod_e \cF_e$, where the product is over
  all edges $e$ with $\{o(e),t(e)\}=\{v,w\}$.  Pick a total order on
  $V$, and for $e\in E$, write $\sigma_e=1$ if $o(e)<t(e)$ and $-1$
  otherwise. Let $\sigma$ be the induced involution of
  $\prod_{e\in E} \cF_e$. Then the \v Cech resolution $\Cech(\cF)$ of
  $\cF$ for the covering $(U_v)$ has global sections
  $C^\bullet(\Gr, \cF)$ and differential $\sigma\circ d_{\cF}$. A
  similar argument treats with the general case.
\end{proof}
If $A$ is any abelian group and $v\in V$, we denote by $v_*A$ the
skyscraper sheaf supported on $\{v\}\subset \Grg$ with stalk $A$.

The sheaf $\PL$ is defined to be the subsheaf of the sheaf of
continuous real-valued functions on $\Grg$, comprising all functions
whose restriction to every edge is locally a linear function $a+bx$,
$a$, $b\in \ZZ$.  So for every edge $e$, pullback gives a map $\tilde
e^*\PL \to \Lins$, where $\Lins$ is the constant sheaf $\ZZ\oplus\ZZ
x$ on $[0,1]$. The continuity condition shows therefore that the sheaf
$\PL$ is the kernel in the short exact sequence 
\begin{equation}
  \label{LABEL500}
  0 \to  \PL \to \prod_{e\in E} \tilde e_*\Lins \xrightarrow{\eval}
  \prod_{v\in V} v_*\bigl((\ZZ^{o\i(v)}\oplus \ZZ^{t\i(v)})/\mathrm{diag}(\ZZ)\bigr)
  \to 0.
\end{equation}
Here the map $\eval$ is determined by its stalks at vertices $v\in V$
\[
  \eval_v \colon\Bigl(\prod_{e\in E}\tilde e_*\Lins\Bigr)_{\!\!v}
  = (\ZZ\oplus \ZZ x)^{o\i(v)} \oplus (\ZZ\oplus\ZZ x)^{t\i(v)}
  \xrightarrow{\eval_0\oplus\eval_1}
  (\ZZ^{o\i(v)} \oplus \ZZ^{t\i(v)})/\mathrm{diag}(\ZZ)
\]
which is clearly surjective. (Recall that $\Gr$ is assumed to have no
isolated vertices, so that $o\i(v)\cup t\i(v)$ is nonempty.)

\begin{prop}
  Restriction to vertices gives an isomorphism
  $H^0(\Grg,\PL)\isomarrow \ZZ^V$, and $H^1(\Grg,\PL)=0$.
\end{prop}

\begin{proof}
  The first statement is clear from the definition of $\PL$. 

  The sheaves $\tilde e_*\ZZ$ and $v_*\ZZ$, along with their products,
  have no $H^1$, and so taking global sections of \eqref{LABEL500}
  gives an exact sequence
  \begin{equation}
    \label{LABEL510}
    (\ZZ\oplus\ZZ x)^E \to
    \prod_{v\in V} (\ZZ^{o\i(v)}\oplus \ZZ^{t\i(v)})/\mathrm{diag}(\ZZ)
    \to H^1(\Grg,\PL) \to 0.
  \end{equation}
  Let $(\phi_e)_{e\in E} \in (\ZZ\oplus \ZZ x)^E$ be a family of
  linear polynomials. Its image under the first map is a family of
  pairs $(a_v,b_v)_{v\in V}$ where
  \begin{align*}
    a_v\colon o\i(v) \to \ZZ,\quad & a_v(e)=\phi_e(0)\\
    b_v\colon t\i(v) \to \ZZ,\quad & b_v(e)=\phi_e(1).
  \end{align*}
  We have an isomorphism
  \[
    \ZZ^E \oplus \ZZ^E\isomarrow
    \prod_{v\in V} (\ZZ^{o\i(v)}\oplus \ZZ^{t\i(v)}) ,
    \quad
    (f,g) \mapsto (f|_{o\i(v)},g|_{t\i(v)})_{v\in V},\ f,g\colon E \to \ZZ
  \]
  in terms of which the diagonal becomes the map
  $\mu \colon \ZZ^V \to \ZZ^E\oplus \ZZ^E$,
  $\mu(h)=(h\circ o,h\circ t)$. Using this, the first map in
  \eqref{LABEL510} may be rewritten as the composite
  \[
    (\ZZ\oplus\ZZ x)^E \xrightarrow[\highsim]{(\eval_0,\ \eval_1)}
    \ZZ^E\oplus \ZZ^E \surject \coker(\mu).
  \]
  Therefore the map is surjective and so $H^1(\Grg,\PL)=0$.
\end{proof}

To fix signs, we define the incoming derivative of $h\in \PL_v$ along
the edge $e$ as follows: suppose $\tilde e^*h=a+bx$. Then the incoming
derivative is $-b$ if $v=o(e)$ and $b$ if $v=t(e)$. (If $o(e)=t(e)=v$
then there are two incoming derivatives, equal to $-b$ and $b$).

We define $\Harm\subset \PL$ to be the subsheaf of functions such that, at each
vertex, the sum of the incoming derivatives equals zero. (In \cite{MW1}
the tropical version of this sheaf is denoted $\underline L$.)  So
there is an exact sequence
\begin{equation}
  \label{LABEL520}
  0 \to \Harm \to \PL \xrightarrow{\diff}  \prod_{v\in V} v_*\ZZ
  \to 0 
\end{equation}
where $\diff=(\diff_v)_v$, $\diff_v\colon \PL_v \to \ZZ$ the sum of
the incoming derivatives at $v$. This is the graph-theoretic analogue
of the divisor sequence \eqref{LABEL020}.  Since $H^1(\Grg,\PL)=0$, we
deduce an exact sequence of cohomology
\begin{equation}
  \label{LABEL530}
  0 \to H^0(\Grg,\Harm) \to \ZZ^V=H^0(\Grg,\PL) \xrightarrow{\diff}
  \ZZ^V \xrightarrow{\delta} H^1(\Grg,\Harm) \to 0.
\end{equation}
\begin{prop}
  \label{LABEL540}
  The map $\diff\colon \ZZ^V \to \ZZ^V$ equals the Laplacian $\Box_0$,
  and $H^0(\Grg,\Harm)=\Ha^0(\Gr)=\ker\Box_0$.
\end{prop}

\begin{proof}
  If $f\in \ZZ^V$ then for $v\in V$,
  \[
    (\Box_0f)(v)=\sum_{o(e)=v}(f(v)-f(t(e)))
    +\sum_{t(e)=v}(f(v)-f(o(e)))
  \]
  which is the sum of the incoming derivatives of $f$, viewed as a PL
  function on $\Grg$, at $v$. The second equality then follows from
  the definitions.
\end{proof}

The Picard group $\Pic(X)=H^1(X,\cO_X^\times)$ of a curve $X$
classifies isomorphism classes of $\cO_X^\times$--torsors on $X$. This
motivates the following definition for graphs.

\begin{defn}
  \label{LABEL550}
  \begin{enumerate}[\upshape(a)]
  \item A $\Harm$--torsor on $\Grg$ is a sheaf of sets $\cL$
    on $\Grg$ satisfying (C), together with an action $\Harm \times \cL \to \cL$,
    making each stalk $\cL_x$, $x\in\Grg$, a principal homogeneous
    space for the abelian group $\Harm_x$.
  \item $\Pic(\Grg)$ is the group of isomorphism classes of
    $\Harm$--torsors on $\Grg$.
  \end{enumerate}
\end{defn}
\noindent(This is equivalent to the definition of torsor found in,
e.g.,~\cite[Tag 02FO]{Stacks}.)  By standard sheaf theory (which we
recall in subsection \ref{LABEL910} below), there is a canonical
isomorphism
\[
\Pic(\Grg)\isom H^1(\Grg,\Harm).
\]

By the above discussion and Proposition \ref{LABEL540}, we have
the following, which is the graph-theoretic analogue of the
isomorphism $\Cl(X)\isomarrow \Pic(X)$ of Theorem \ref{LABEL060} for a
smooth projective curve $X$.
\begin{prop}
  The map $\delta$ in \eqref{LABEL530} induces an isomorphism
  \[
    \pushQED{\qed}
    \bar\delta\colon \Clhat(\Gr)= \ZZ^V/\Box_0(\ZZ^V) \isomarrow \Pic(\Grg).
    \qedhere
    \popQED
  \]
\end{prop}
In terms of torsors on $\Grg$, $\delta$ is given as follows: let
$c\in \ZZ^V=H^0(\Grg,\prod_v v_*\ZZ)$. Then the corresponding
$\Harm$--torsor is the subsheaf of sets $\cL_c=\diff\i(c)\subset
\PL$. Explicitly, local sections of $\cL_c$ are PL-functions $f$ with
$\diff_v(f)=c(v)$ at every vertex $v$. If $c=\Box_0(g)$, for
$g\in\ZZ^V$ (viewed as a PL-function on $\Grg$), then addition of $g$
is an isomorphism $\cL_0=\Harm \isomarrow \cL_c$ of torsors.

The remaining sheaf, which we suggestively call $\Omega$, is the
kernel in the exact sequence
\begin{equation}
  \label{LABEL560}
  0 \to \Omega \to \prod_{e\in E} \tilde e_*\ZZ \xrightarrow{\underline{d}\adj}
  \prod_{v\in V} v_*\ZZ \to 0
\end{equation}
where $\underline{d}\adj$ is the unique map of sheaves which on global
sections equals $d\adj\colon\ZZ^E\to\ZZ^V$.  (Since $\Gr$ has no
isolated vertices, the sequence is exact at the right.) Passing to
cohomology and applying Proposition \ref{LABEL320}(h) we obtain:
\begin{prop}
  \label{LABEL570}
  $H^0(\Grg,\Omega)\isom \Ha^1(\Gr)=\ker d\adj\subset \ZZ^E$ and
  $H^1(\Grg,\Omega)\isom \coker d\adj$. If $\Gr$ is finite and
  connected, then
  $H^1(\Grg,\Omega)\isom \ZZ$. \qed
\end{prop}
For $\Gr$ connected, one may view the isomorphism
$H^1(\Grg,\Omega)\isom\ZZ$ as the analogue of Serre duality
$H^1(X,\Omega_{X})\isom\CC$.

Consider the ``edge derivative'' map
\[
  \ediff \colon \PL \to \prod_{e\in E}\tilde e_*\ZZ
\]
induced by the $x$-derivative map $\ZZ\oplus \ZZ x \to \ZZ$. Its
restriction to $\Harm$ fits into  a short exact
sequence
\begin{equation}
  \label{LABEL580}
  0 \to \ZZ \xrightarrow{\ q\ } \Harm \xrightarrow{\ediff} \Omega \to 0,
\end{equation}
where $q$ is the inclusion of the constant sheaf $\ZZ$,
and passing to cohomology gives a long exact sequence
\begin{equation}
  \label{LABEL590}
  H^0(\Grg,\Omega) \to  H^1(\Gr) \xrightarrow{q_*}
  \Pic(\Grg) \to H^1(\Grg,\Omega) \to 0
\end{equation}
to be regarded as the graph-theoretic analogue of the dlog sequence
\eqref{LABEL170}.  In terms of torsors, $q_*$ takes a $\ZZ$--torsor on
$\Grg$ to the $\Harm$--torsor obtained by extension of structure group
$\ZZ \hookrightarrow\Harm$.

\begin{defn}
  $\Pic^0(\Grg) \subset \Pic(\Grg)$ is the kernel of the map
  $\Pic(\Grg) \to H^1(\Grg,\Omega)$ in equation \eqref{LABEL590}.
\end{defn}

By Proposition \ref{LABEL570}, for $\Gr$ finite and connected we have
$\Pic(\Grg)/\Pic^0(\Grg)\isom \ZZ$, just as for curves. 

\begin{thm}
  \label{LABEL600}
  $\bar\delta(\Clhat^0(\Gr))= \Pic^0(\Grg)$, and the diagram
  \[
    \begin{tikzcd}
      \Clhat(\Gr) \arrow[r, "-\bar\delta", "\sim"']& \Pic(\Grg)
      \\
      \Clhat^0(\Gr) \arrow[u, hook]\arrow[r, "\lowsim"]
      &\Pic^0(\Grg) \arrow[u, hook] 
      \\
      & \P(\Gr) \arrow[u, "\wr", "q_*"']
      \arrow[ul, "\comp", "\isom"']
    \end{tikzcd}
  \]
  commutes.
\end{thm}

\begin{proof}
  Recall from Proposition \ref{LABEL410} that the isomorphism $\comp
  \colon \P(\Gr) \isomarrow \Clhat^0(\Gr)$ is induced by $d\adj\colon
  C^1(\Gr) \to C^0(\Gr)$. By definition, $q_*\colon \P(\Gr) \isomarrow
  \Pic^0(\Grg)$. So it is enough to show that $\bar\delta\circ \comp =
  -q_*$.  Let $g\in C^1(\Gr)=\ZZ^E$.  To compute the boundary
  $\delta(d\adj g)$, consider the \v Cech resolution from Proposition
  \ref{LABEL490}(b) for the exact sequence \eqref{LABEL520}:
  \begin{equation*}
    \begin{tikzcd}
       C^0(\Harm) \arrow[r] \arrow[d] &
       C^0(\PL) =\prod_{v\in V} \PL_v \arrow[r, "\diff"]
       \arrow[d, "{d_{\PL}}"] & \ZZ^V \arrow[d] \\
       C^1(\Harm) \arrow[r, "="] & C^1(\PL) = (\ZZ\oplus\ZZ x)^E
       \arrow[r] & 0\, .  
    \end{tikzcd}
  \end{equation*}
  Let  $h_v\in \PL_v$ to be the function such that
  \[
    \tilde e^*h_v=
    \begin{cases}
      g(e)(1-x) & \text{for $e\in t\i(v)$}\\
      g(e) x &\text{for $e\in o\i(v)$.} 
    \end{cases}
  \]
  Let $h=(h_v)\in C^0(\PL)$.  Then $\diff(h)= -d\adj g$, and
  $d_{\PL}(h)=g\in C^1(\PL)=C^1(\Harm)$.  So $\delta(-d\adj g)$ is the
  image of $g\in C^1(\Gr)$ under the composite map
  \[
    C^1(\Gr) \to H^1(\Gr)=H^1(\Grg,\ZZ) \to H^1(\Grg,\Harm)
  \]
  which is $q_*(g)$ as required.  
\end{proof}

\begin{rem}
  This proof has a natural interpretation in terms of
  torsors. Let $e\in E$, and consider $\hat e\in C^1(\Gr)$. Its image
  in $H^1(\Gr)=H^1(\Grg,\ZZ)$ is represented by the $\ZZ$--torsor $T$
  obtained from the trivial torsor by twisting by 1 along $e$. Under
  $q_*$ this maps to the $\Harm$--torsor $\cL$ obtained by twisting the
  trivial torsor by the constant function $1$ along $e$.
  
  In the other direction, $d\adj(\hat e)=\widehat{\partial(e)}$.  So
  $\delta(- d\adj(\hat e))$ is the class of the $\Harm$--torsor $\cL'$, which is
  the sheaf of PL-functions $f$ with $\diff_{t(e)}(f)=-1$,
  $\diff_{o(e)}=1$ and $\diff_v(f)=0$ for all other vertices
  $v$.

  There is an explicit isomorphism between $\cL'$ and $\cL$. On the
  complement of the edge $e$, they are both the trivial torsor. Write
  $u=o(e)$, $v=t(e)$, and set $U=U_u(2/3)$, $V=U_v(2/3)$. We write
  down an isomorphism over the open neighborhood $U\cup V$ of $e$.
  If $f$ is a section of $\cL'$ on $U\cup V$, then its restriction to
  $e$ is some linear function $a+bx$, whose incoming derivative at $u$
  is $-b$ and at $v$ is $b$. Let $f_u\colon U \to \RR$ be the
  function equal to $f$ except on $e\cap U$, where it equals $f+x$,
  and similarly $f_v\colon V \to \RR$ the function equal to $f-(1-x)$
  on $e\cap V$. Then the incoming derivative at $u$ of $f_u$ (resp.~at
  $v$ of $f_v$) along $e$ is $b-1$ (resp.~$b+1$), hence $f_u$ and
  $f_v$ are harmonic, and since their difference on $U\cap V$ is the
  constant function $1$, they define a local section of the torsor
  $\cL$.  So we have an isomorphism from $\cL'$ to $\cL$ over $U\cup V$,
  compatible with the trivial isomorphism on the complement of $e$.
\end{rem}

We have the following table of analogies between sheaves on curves and
on (the geometric realizations of) graphs. (Note that since
$V\subset \Grg$ is discrete, $\prod_vv_*\ZZ=\bigoplus_v v_*\ZZ$ as sheaves on $\Grg$.)
\\[1ex]
\begin{equation}
  \label{LABEL610}
  \begin{tabular}{|c|c|}
    \hline
    Curve $\vphantom{I^{I^{I^I}}}X$&Graph $\Gr$ \\
    \hline
    $\vphantom{I^{I^{I^I}}}\cK_X^\times$ & $\PL$ \\[1ex]
    $\underline{\Div}_X=\bigoplus_{x\in X}x_*\ZZ$ & $\prod_{v\in V}
                                                    v_*\ZZ$ \\[1ex]
    $\cO^\times_X$ & $\Harm$ \\[1ex]
    $\Omega^1_{X/\CC}$ & $\Omega$ \\[1ex]
    $\ 0 \to \cO_X^\times \to \cK_X^\times \xrightarrow{\div}
    \bigoplus_{x\in X} x_*\ZZ \to 0\ $
    &   $\ 0 \to \Harm \to \PL \xrightarrow{\diff} \prod_{v\in V} v_*\ZZ
      \to 0\ $  \\[1ex] 
    $0 \to \CC^\times \to \cO_X^\times
    \xrightarrow{\mathrm{dlog}} \Omega_X^1 \to 0$
    & $  0 \to \ZZ \to \Harm \xrightarrow{\diff_E} \Omega \to 0$\\[1ex] 
    \hline
  \end{tabular}
\end{equation}
\\[2ex]


\section{Generalized Jacobians and generalized Picard groups of
  graphs}
\label{LABEL620}

\subsection{Graphs with modulus}

\label{LABEL630}
Unless indicated otherwise, in this section $\Gr=(V,E,o,t)$ is a nonempty
locally finite graph.

Let $(w_i)_{i\in I}$ be a nonempty family of elements of $V$, and
write $\rho\colon I\to V$ for the map $i\mapsto w_i$. We assume that
for every $v\in V$, the set $I(v)\defeq\rho\i(v)\subset I$ is finite.
We then call the formal sum $\fm=\sum_{i\in I}w_i$ a \emph{modulus} on
$\Gr$.  Write $S=\supp(\fm)=\{w_i\}_{i\in I}$.  We say $\fm$ is
\emph{reduced} if the $w_i$ are distinct, and \emph{finite} if $S$
(or equivalently $I$) is finite.

\begin{defn}
  \label{LABEL640}
  Let $A$ be an abelian group and $\fm$ a modulus on $\Gr$.  An
  \emph{$\fm$-harmonic map} from $\Gr$ to the abelian group $A$ is a pair of
  maps $(f,g)\in A^V\oplus A^I$ such that
\begin{enumerate}[\upshape (a)]
\item
For all $v\in V\setminus S$, $(\Box_0f)(v)=0$.
\item
For all $v\in S$, $(\Box_0f)(v)
  =-\sum_{i\in I(v)}g(i)$.
\end{enumerate}
\end{defn}

Observe that the conditions (a) and (b) are
equivalent to the single condition that for all $v\in V$,
\[
(\Box_0f)(v)+\sum_{i\in I(v)}g(i)=0.
\]
\begin{prop}
If $\fm$ is a reduced modulus, then to give an $\fm$-harmonic map from
$\Gr$ to $A$ is equivalent to giving a map $h\colon V\to A$ satisfying
condition (a) of Definition \ref{LABEL640}. \qed
\end{prop}

We next define the ray class groups of divisors and codivisors on
$\Gr$ modulo $\fm$ as follows. Define
\[
  \Div_\fm(\Gr)\defeq\Div(\Gr) \oplus \ZZ[I] = \ZZ[V]\oplus \ZZ[I]
\]
and, for $\Gr$ connected,
\[
  \Div_\fm^0(\Gr)\defeq\Div^0(\Gr) \oplus \ZZ[I] = \ZZ[V]_0\oplus \ZZ[I].
\]
Define also
\[
  \Prin_\fm(\Gr)\defeq \im\bigl( (\Delta_0,\rho^*)\colon
  \Div(\Gr) \xrightarrow{v \mapsto (\Delta_0(v),I(v))} \Div_\fm(\Gr).
\]
Here by abuse of notation we also write $I(v)\in \ZZ[V]$ for the sum $\sum_{\rho(i)=v}i$.
If $\Gr$ is connected, then $\Prin_\fm(\Gr) \subset \Div_\fm^0(\Gr)$.
\begin{defn}
  The \emph{ray class group modulo $\fm$} of divisors
  (resp.,~degree-zero divisors, for $\Gr$ connected) on $\Gr$ is
  $\Cl_\fm(\Gr) \defeq \Div_\fm(\Gr)/\Prin_\fm(\Gr)$
  (resp.~$\Cl_\fm^0(\Gr) \defeq \Div_\fm^0(\Gr)/\Prin_\fm(\Gr)$).
\end{defn}
By definition, the ray class group is universal for $\fm$-harmonic maps:

\begin{prop}
  \label{LABEL650}
  Let $A$ be an abelian group, $\fm$ a modulus on $\Gr$, and $v_0\in V$.
  There are isomorphisms
  \begin{align*}
    \{ \textup{$\fm$-harmonic maps $(f,g)$ from $\Gr$ to $A$} \}
    &\isom \Hom(\Cl_\fm(\Gr), A)
    \\\intertext{and, for $\Gr$ connected,}
    \{ \textup{$\fm$-harmonic maps $(f,g)$ with $f(v_0)=0$} \}
    &\isom \Hom(\Cl^0_\fm(\Gr), A)
  \end{align*}
  given by $(f,g)\mapsto f\oplus g$. \qed
\end{prop}

Likewise, we define
\[
  \Divhat_\fm(\Gr)=\Divhat(\Gr)\oplus \ZZ^I = \ZZ^V\oplus \ZZ^I,
\]
and for $\Gr$ connected,
\[
  \Divhat^0_\fm(\Gr) = \Divhat^0(\Gr) \oplus \ZZ^I
\]
(where $\Divhat^0(\Gr)$ is as in section \ref{LABEL330}). The corresponding
codivisor ray class groups are
\[
  \Clhat_\fm(\Gr)=\Divhat_\fm(\Gr)/\Prin_\fm(\Gr) \supset
  \Clhat^0_\fm(\Gr)=\Divhat^0_\fm(\Gr)/\Prin_\fm(\Gr),
\]
where
\[
  \Prinhat_\fm(\Gr) = \im\bigl((\Box_0,\tr\rho)\colon \ZZ^V
  \xrightarrow{f\mapsto (\oBox_0(f),f\circ\rho)}
  \Divhat_\fm(\Gr) = \ZZ^V\oplus \ZZ^I\bigr) \subset
  \Divhat^0_\fm(\Gr).
\]
There is an obvious map
$\iota_\fm\colon\Div_\fm(\Gr)\hookrightarrow\ \Divhat_\fm(\Gr)$ which
is the direct sum of $\iota_0\colon \ZZ[V] \to \ZZ^V$ and the
inclusion $\ZZ[I] \hookrightarrow \ZZ^I$. For $\Gr$ connected,
$\iota_\fm(\Div^0_\fm(\Gr))\subset \Divhat^0_\fm(\Gr)$.

\begin{prop}
  We have $\iota_\fm(\Prin_\fm(\Gr)) \subset \Prinhat_\fm(\Gr)$, inducing
  homomorphisms $\Cl_\fm(\Gr) \to \Clhat_\fm(\Gr)$ and (for $\Gr$ connected)
  $\Cl_\fm(\Gr)^0 \to \Clhat_\fm^0(\Gr)$. Both maps are isomorphisms if $\Gr$ is finite.  
\end{prop}

\begin{rem*}
  If $\Gr$ is infinite, the map $\Cl_\fm(\Gr) \to \Clhat_\fm(\Gr)$
  need not be either injective or surjective.
\end{rem*}

\begin{proof}
  Applying Lemma \ref{LABEL280}, it is enough to show that the square
  \[
  \begin{tikzcd}
    \ZZ[V] \arrow[r, hookrightarrow] \arrow[d, "{v\mapsto I(v)}"]
    & \ZZ^V \arrow[d, "\tr\rho"]
    \\
    \ZZ[I] \arrow[r, hookrightarrow]
    & \ZZ^I
  \end{tikzcd}
  \]
  commutes, and it is easy to check that the image of $v\in \ZZ[V]$ in
  either direction is the characteristic function $\mathbf{1}_{I(v)}
  \in \ZZ^V$.
\end{proof}
If $\fm$ is reduced then there are alternative descriptions of
$\Cl_\fm(\Gr)$ and $\Clhat_\fm(\Gr)$. As $\rho$ is a bijection onto
$S$, the maps $\rho^* \colon \ZZ[V] \to \ZZ[I]$,
$\tr\rho\colon \ZZ^V \to \ZZ^I$ are surjective with kernels
$\ZZ[V\setminus S]$ and $\ZZ^{V\setminus S}$ respectively. This gives:

\begin{prop}
  \label{LABEL660}
  Suppose that $\fm$ is reduced. There are isomorphisms
  \[
    \frac{\ZZ[V]}{\Delta_0(\ZZ[V\setminus S])} \isom \Cl_\fm(\Gr),\quad
    \frac{\ZZ^V}{\Box_0(\ZZ^{V\setminus S})} \isom \Clhat_\fm(\Gr),
  \]
  and (for $\Gr$ connected)  
  \[
    \frac{\ZZ[V]_0}{\Delta_0(\ZZ[V\setminus S])} \isom \Cl^0_\fm(\Gr),\quad
    \frac{\ZZ^{V,0}}{\Box_0(\ZZ^{V\setminus S})} \isom \Clhat_\fm^0(\Gr).\qed
  \]
\end{prop}



\begin{rem}
  \label{LABEL670}
  Although we allow moduli on graphs to be non-reduced (since this can
  happen for the graphs arising in the study \cite{jrs} of N\'eron models of
  generalized Jacobians) this notion does not carry any genuinely new
  information. Suppose that $\Gr$ is connected. Let $\fm$ be the modulus given by a family of vertices
  $\rho\colon I \to V$, and let $\fm_0$ be the corresponding reduced
  modulus, given by $S=\rho(I)\subset V$. Then from the definitions
  one sees easily that there is a co-Cartesian diagram
  \[
    \begin{tikzcd}
      \ZZ[S] \arrow[r] \arrow[d,hookrightarrow, "v\mapsto I(v)"'] &
      \Cl^0_{\fm_0}(\Gr) \arrow[d,hookrightarrow]\\
      \ZZ[I] \arrow[r] & \Cl^0_\fm(\Gr)
    \end{tikzcd}
  \]
  in which the horizontal arrows are induced by the inclusions
  $\ZZ[S]\hookrightarrow \Div^0_{\fm_0}(\Gr)$, $\ZZ[I] \hookrightarrow
  \Div^0_\fm(\Gr)$. So the ray class groups $\Cl^0_\fm(\Gr)$ all come
  by pushout from those with reduced modulus. Likewise for codivisor
  ray class groups: there is a co-Cartesian square
  \[
    \begin{tikzcd}
      \ZZ^S \arrow[r] \arrow[d,hookrightarrow, "\tr\rho"']&
      \Clhat^0_{\fm_0}(\Gr) \arrow[d,hookrightarrow]\\
      \ZZ^I \arrow[r] & \Clhat^0_\fm(\Gr)\,.
    \end{tikzcd}
  \]
\end{rem}

\subsection{The extended graph \texorpdfstring{$\Grm$}{Gr\_m}}
\label{LABEL680}
Let $X$ be a (projective, smooth, irreducible) complex curve, and
$\fm=\sum P_i$ a reduced modulus. Let $X_\fm$ be the singular curve obtained
from $X$ by identifying the points of $\fm$ (with linearly independent
tangent directions along the branches). Then one knows that the
generalized Jacobian of $X$ with respect to $\fm$ is the Picard
variety of $X_\fm$, and the cohomology $H^1(X_\fm,\ZZ)$  equals the
relative cohomology $H^1(X,\{ P_i \};\ZZ)$. This motivates the following construction for graphs.

For the rest of this subsection we
consider only finite moduli. (We will discuss general moduli at the end of subsection \ref{LABEL1150} below.)

\begin{defn}
\label{LABEL690}
(cf.~\cite[Sect.~2.4]{jrs}) Let $\Gr$ be a graph, $\fm$ a finite
modulus.  Let $\star$ denote an auxiliary vertex.  Set
\begin{equation*}
V_\fm=V\sqcup \{\star\}\quad\textup{and}\quad
E_\fm=E\sqcup\{e_i\mid i\in I\}
\end{equation*}
with $o(e_i)=\star$ and $t(e_i)=w_i$ for $i\in I$.  Define $\Grm$
to be the graph with vertices $V_\fm$ and edges $E_\fm$.
\end{defn}
For example, if $\Gr$ is the graph on the left in Figure 1
with modulus $\fm=w_1+w_2$, $I=\{1,2\}$, and $S=\{w_1,w_2\}$,
 then $\Gr_{\fm}$ is the graph on the right:

\tikzset{middlearrow/.style={
        decoration={markings,
            mark= at position 0.5 with {\arrow{#1}} ,
        },
        postaction={decorate}
    }
}

\begin{figure}[ht]
\scalebox{1.2}{
\begin{tikzpicture}
\draw[black, thick, middlearrow={>}] (0,0) --(3.2,0);
\draw[black, thick, middlearrow={>}] (3.2,0) -- (1.6,1.15);
\draw[black, thick, middlearrow={>}] (1.6,1.15) --(0,0);
\node [left] at (0,0) {\tiny $w_1$};
\node[right] at (3.2,0) {\tiny $w_2$};
\node[above] at (1.6,1.15) {\tiny $v$};
\fill[black] (1.6, 1.15) circle (.08cm);
\fill[black] (0,0) circle (.08cm);
\fill[black] (3.2,0) circle (.08cm);
\node[above] at (2.4, .58) {\tiny $h$};
\node[above] at (.8, .58) {\tiny $f$};
\node[below] at (1.6, 0) {\tiny $g$};
\end{tikzpicture}
}\hspace{1em}
\scalebox{1.2}{
\begin{tikzpicture}
\draw [black, thick, middlearrow={>}] (1.6,2.2) ..controls (.3,1.85) and
(0,.8) ..(0,0);
\draw [black, thick, middlearrow={>}] (1.6,2.2) ..controls (2.9,1.85) and
(3.2,.8) ..(3.2,0);
\draw[middlearrow={>}] (0,0) --(3.2,0);
\draw[middlearrow={>}] (3.2,0) -- (1.6,1.15);
\draw[middlearrow={>}] (1.6,1.15) --(0,0);
\node [left] at (0,0) {\tiny $w_1$};
\node[right] at (3.2,0) {\tiny $w_2\,\, .$};
\node[above] at (1.6,1.15) {\tiny $v$};
\node[above] at (1.6,2.2) {\footnotesize $\star$};
\fill[black] (1.6,2.2) circle (.08cm);
\fill[black] (1.6, 1.15) circle (.08cm);
\fill[black] (0,0) circle (.08cm);
\fill[black] (3.2,0) circle (.08cm);
\node[left,black] at (.4, 1.6) {\tiny $e_1$};
\node[right,black] at (2.8, 1.6) {\tiny $e_2$};
\node[above] at (2.4, .58) {\tiny $h$};
\node[above] at (.8, .58) {\tiny $f$};
\node[below] at (1.6, 0) {\tiny $g$};
\end{tikzpicture}
}
\caption{The graphs $\Gr$ and $\Gr_{\fm}$}
\end{figure}
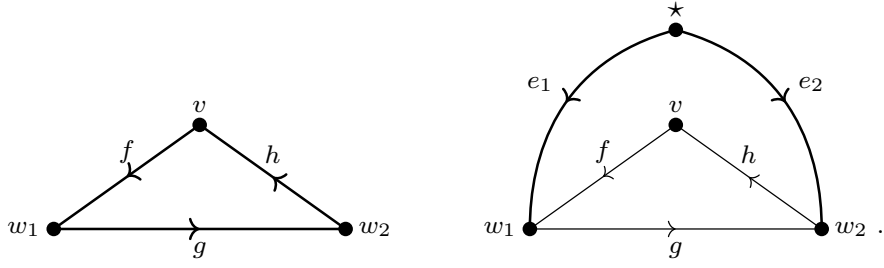
We shall write $\partial_\gfm$, $d_\gfm$, $\iota_\gfm$ etc.~for
the maps defined earlier with $\Gr$ replaced by $\Grm$. 
Explicitly, the complex $C_\bullet(\Grm)$ is
\begin{align*}
C_0(\Grm)=\ZZ[V] & \oplus \ZZ\xleftarrow{\partial_\gfm}
C_1(\Grm)=\ZZ[E]\oplus \ZZ[I]\text{ with}\\
\partial_\gfm(e,0)& =(\partial e, 0)\text{ and }
\partial_\gfm(0,i)=(w_i, -1).
\end{align*}
The adjoint $\partial_\gfm\adjp\colon C_0(\Grm)\to
C_1(\Grm)$ is given by
\begin{equation*}
\partial_\gfm\adjp (v,0)=\bigl(\partial\adj v ,\sum_{i\in I(v)}i\bigr),\quad
\partial_\gfm\adjp(0,1)=\bigl(0,-\sum_{i\in I}i\bigr)
\end{equation*}
which is well defined as $I$ is finite.

The cochain complex $C^\bull(\Grm)$ is
\begin{equation}
  \label{LABEL700}
  C^0(\Grm)=\ZZ^V \oplus \ZZ\xrightarrow{\ d_\gfm\ }
  C^1(\Grm)=\ZZ^E\oplus \ZZ^I,\quad
  d_\gfm(f,a)=(df,\tr{\rho} f - a).
\end{equation}
The associated ``adjoint'' of $d_\gfm$ is the map 
\begin{equation}
  \label{LABEL710}
  d_\gfm\adjp \defeq\tr(\partial_\gfm\adjp)\colon C^1(\Grm)\to C^0(\Grm),\quad
  d_\gfm\adjp(f,g)=\bigl(d\adj f + g \circ \rho^{-1}, -\sum_{i\in
    I}g(i)\bigr).
\end{equation}
Denote also by $\rho\colon\ZZ[I] \to \ZZ[V]$ the linear extension of
$\rho$, and write $r=\tr\rho\colon \ZZ^V \to \ZZ^I$ for its
transpose. As the fibres of $\rho$ are finite, the map
\[
  r\adj \colon \ZZ^I \to \ZZ^V,\quad
  r\adj(f)(v)= \sum_{i\in I(v)} f(i)
\]
is well defined.

\begin{prop}
 \label{LABEL720}
 \begin{enumerate}[\upshape(a)]
 \item  $\im(d_\gfm)=\{ (df,f\circ \rho) \mid f \in C^0(\Gr) \}$.
 \item Suppose that $\Gr$ is finite. Then 
   \[
     \im(\partial\adjp_\gfm) =  \sum_{v\in V} \ZZ(\partial\adj v,
     I(v)) \subset  \ZZ[E] \oplus \ZZ[I]
   \]
   and
   \[
     \ker(d_\fm\adjp) = \{ (f,g) \in \ZZ^E \oplus \ZZ^I \mid d\adj f +
     g\circ\rho\i = 0 \}.
   \]
 \end{enumerate}
\end{prop}
\begin{proof}
  (a) By \eqref{LABEL700},
  \[
    \im(d_\gfm) = \{ (df,f\circ \rho) \mid f \in C^0(\Gr) \} + \ZZ(0,1_I)
  \]
  and if $f=1_V\in C^0(\Gr)$ then $(df,f\circ\rho)=(0,1_I)$.

  (b) If $\Gr$ is finite then by Lemma \ref{LABEL280} $\iota_1$
  induces an isomorphism
  $\im\partial\adjp_\gfm \isomarrow \im d_\gfm$, and the first
  equality then follows from (a). For the second, we know that
  $\sum_{v\in V} d\adj f(v)=0$ and so if $(f,g)$ belongs to the right
  hand expression, then automatically $\sum_{i\in I}g(i)=0$.
\end{proof}

We have an inclusion of graphs $j\colon \Gr \hookrightarrow
\Grm$. On the chain and cochain complexes this induces obvious
inclusions and projections:
\begin{align*}
C_i(\Gr)\stackrel{j_\ast}{\hooklongrightarrow}&C_i(\Grm)\stackrel
{j^!}{\twoheadlongrightarrow}C_i(\Gr)\quad (i=0,1)\\
C^i(\Gr)\stackrel{j_!}{\hooklongrightarrow}&C^i(\Grm)\stackrel
{j^\ast}{\twoheadlongrightarrow}C^i(\Gr).
\end{align*}
From the formulae for the various (co-)boundary maps, we have the
following compatibilities.
\begin{lem}
  \label{LABEL730}
  \begin{enumerate}[\upshape (a)]
  \item
    $j_\ast\circ\partial =\partial_\gfm \circ j_\ast\colon C_1(\Gr)
    \to C_0(\Grm)$, 
  \item
    $j^\ast\circ d_\gfm=d\circ j^\ast\colon C^0(\Grm)\to
    C^1(\Gr)$,
  \item
    $j^!\circ\partial_\gfm\adjp=\partial\adj \circ j^!\colon 
    C_0(\Grm)\to C_1(\Gr) $,
  \item
    $j_!\circ d\adj=d_\gfm\adjp\circ j_!\colon 
    C^1(\Gr)\to C^0(\Grm)$.
  \item $j_!\circ \iota_\Gr = \iota_\gfm\circ j_* \colon C_1(\Gr) \to C^1(\Grg)$. 
    \qed
  \end{enumerate}
\end{lem}

\subsection{The Abel--Jacobi map for finite graphs with modulus}
\label{LABEL740}

In this subsection, $\Gr$ will be a connected finite graph with
modulus $\fm$.  (For infinite graphs, see Corollary \ref{LABEL1170}
and the discussion which follows it.)

Let $\alpha_\fm\colon H_1(\Gr)\hookrightarrow \Ha^1(\Grm)\Zdual$ be the
map induced by $j_\ast\colon C_1(\Gr)\hookrightarrow C_1(\Grm)=
C^1(\Grm)\Zdual$.  Note that $\alpha_\fm$ is injective 
since $\alpha_\Grm\colon H_1(\Grm)\hookrightarrow \Ha^1(\Grm)\Zdual$ is injective
by Proposition \ref{LABEL320}(f).
\begin{defn}
The generalized Jacobian of conductor $\fm$ (or $\fm$-Jacobian)
of $\Gr$ is $\J_\fm(\Gr)=\Ha^1(\Grm)\Zdual/\alpha_\fm(H_1(\Gr))$.
\end{defn}
As for the generalized Jacobians of curves, $\J_\fm(\Gr)$ is an
extension, in this case of a finite abelian group by a free abelian
group. 
\begin{prop}
  There is an exact sequence
  \[
    0 \to \ZZ \xrightarrow{1\mapsto \sum_{i\in I}i} \ZZ[I] \to \J_\fm(\Gr) \to \J(\Gr) \to 0
  \]
  in which the surjection is induced by the transpose of the inclusion
  $j_! \colon \Ha^1(\Gr) \hookrightarrow \Ha^1(\Grm)$. 
\end{prop}
This follows easily from the definitions; we prove an equivalent
result in Proposition \ref{LABEL1060} below.

For $D\in \Div^0(\Gr)$, let $\gamma_D\in C_1(\Gr)$ with
$\partial(\gamma_D)=D$ as in Section 2.3, and let $\lambda_{\fm,D}\in
\Ha^1(\Grm)\Zdual$ be the map $\omega\mapsto
\omega(j_*\gamma_D)$. Then $\lambda_{\fm,D}$ is well-defined modulo
$\alpha_\fm(H_1(\Gr))$.

For $i\in I$, let $\varepsilon_i\in \Ha^1(\Grm)\Zdual$ be the map
$\omega\mapsto \omega(e_i)$, with $e_i\in E_\fm$ as in Definition
\ref{LABEL690}. 
\begin{defn}
  \label{LABEL750}
  The \emph{Abel--Jacobi map} $\AJt_\fm\colon \Div^0_\fm(\Gr)\to
  \J_\fm(\Gr)$ is the homomorphism given by
  \[
    \AJt_\fm\colon
    \begin{cases}(D,0)\hspace*{-0.7em}&\mapsto \lambda_{\fm,D}\\
      (0,i)\hspace*{-0.7em}&\mapsto \varepsilon_i
    \end{cases}
    \quad \mathrm{mod}\ \alpha_\fm(H_1(\Gr))
  \]
  for $D\in\Div^0(\Gr)$ and $i\in I$.
\end{defn}


\begin{thm} \label{LABEL760}
  (Abel's Theorem for the generalized Jacobian.) $\AJt_\fm$ induces an
  isomorphism $\AJ_\fm\colon\Cl^0_\fm(\Gr)\isomarrow \J_\fm(\Gr)$.
\end{thm}

\begin{proof}
  We imitate the proof of Theorem \ref{LABEL400}.
  Consider the diagram with exact rows
  \begin{equation}
  \label{LABEL770}
  \begin{tikzcd}
    0 \arrow[r]
    &\ker\partial_\Gr \arrow[d, equals] \arrow[r, "{(\incl,0)}"]
    & C_1(\Gr)\oplus \ZZ[I]=C_1(\Grm)
      \arrow[d, twoheadrightarrow, "{\pi_\fm}"]
      \arrow[r, "{\partial\oplus\id}"]
    & \Div^0_\fm(\Gr) \arrow[r] \arrow[d, "{\AJt_\fm}"]
    & 0
    \\
    0 \arrow[r]
    & H_1(\Gr) \arrow[r, "{\alpha_\fm}"]
    & \Ha^1(\Grm)\Zdual \arrow[r]
    & \J_\fm(\Gr)=\dfrac{\Ha^1(\Grm)\Zdual}{\alpha_\fm(H_1(\Gr))} \arrow[r]
    & 0
  \end{tikzcd}
  \end{equation}
  where $\pi_\fm$ is induced by the split inclusion $\Ha^1(\Grm)
  \subset C^1(\Grm)= C^1(\Gr)\oplus \ZZ^I$, cf.~equations
  \eqref{LABEL700} and \eqref{LABEL710}. By the definitions, the diagram
  commutes. So $\AJt_\fm$ is surjective, and $\ker\pi_\fm \isomarrow
  \ker\AJt_\fm$. But
  \[
    \ker(\pi_\fm) = \{ \gamma \in C_1(\Grm)\mid
    \text{for all $\omega\in \Ha^1(\Grm)$,
      $\omega(\gamma)=0$.} \}
  \]
  so $\ker\pi_\fm=\im \partial\adjp_\gfm$ by Lemma
  \ref{LABEL300}(b). Therefore, applying Proposition \ref{LABEL720}(b),
  \[
    \ker(\AJt_\fm)=(\partial\oplus\id_{\ZZ[I]})(\im
    \partial\adjp_\gfm)
    =\im\Bigl(\ZZ[V] \xrightarrow{v\mapsto (\Delta_0(v),I(v))} \ZZ[V]\oplus \ZZ[I]\Bigr)
    =\Prin_\fm(\Gr)
  \]
   which gives
   the required isomorphism $\AJ_\fm$.
\end{proof}

\subsection{The \texorpdfstring{$\fm$}{m}-Picard group of a graph with
  modulus}
\label{LABEL780}

In this subsection, $\Gr$ will be a (not necessarily finite) connected graph with finite modulus $\fm$.
The extended graph $\Gr_\fm$ is then defined, and
recall that $j_!\colon C^1(\Gr)\hookrightarrow C^1(\Grm)$ is the
extension-by-zero map: if $\omega\in C^1(\Gr)$,
$(j_!\omega)(e)=\omega(e)$ if $e\in E$ and $(j_!\omega)(e_i)=0$ for
all $i\in I$. By Lemma \ref{LABEL730}(d), $j_!(\Ha^1(\Gr))\subset \Ha^1(\Gr_\fm)$.

\begin{defn}
  \label{LABEL790}
  Let $\beta_\fm$ be the composite homomorphism
  \[
    \beta_\fm\colon \Ha^1(\Gr)\xrightarrow{j_!}  \Ha^1(\Grm)
    \hookrightarrow C^1(\Gr_\fm) \to H^1(\Grm)
  \]
  and define
  \[
    \P_\fm(\Gr)\defeq H^1(\Grm)/\beta_\fm(\Ha^1(\Gr)),
  \]
  the $\fm$-Picard group of $\Gr$.
\end{defn}
If $\Gr$ is finite then $\beta_\fm$ is injective, by Proposition
\ref{LABEL320}(e) applied to $\Gr_\fm$.
\begin{prop}
  \label{LABEL800}
  The map
  \begin{equation}
    \label{LABEL810}
    d\adj\oplus \id \colon C^1(\Grm)=C^1(\Gr)\oplus \ZZ^I\to
    \Divhat^0_\fm(\Gr)=\im d\adj \oplus \ZZ^I
  \end{equation}
  induces an isomorphism $\comp_\fm\colon\P_\fm(\Gr)\isom \Clhat_\fm^0(\Gr)$.
\end{prop}

\begin{proof}
  The map \eqref{LABEL810} is surjective, with kernel
  $\ker(d\adj)\oplus \{0\}=\Ha^1(\Gr)\oplus\{0\}$. By Proposition
  \ref{LABEL720}(a) we therefore have
  \begin{align*}
  \P_\fm(\Gr)
  & = \frac{ C^1(\Grm) }{ j_!\Ha^1(\Gr)+d_\gfm C^0(\Grm) }
    =\frac{ C^1(\Gr)\oplus \ZZ^I }{ (\Ha^1(\Gr)\oplus 0)+
             \{ (df, f\circ \rho)\mid f\in C^0(\Gr) \} }
  \\
  &\xrightarrow[\ \comp_\fm\ ]{\lowsim}
    \frac{ \Divhat^0_\fm(\Gr)}{ \{ (\Box_0f, f\circ \rho)\mid f\in
                                        C^0(\Gr) \} }
    =\Clhat_\fm^0(\Gr).\qedhere
\end{align*}
\end{proof}
Suppose that the modulus $\fm$ is reduced. We then have an alternative
description of $\P_\fm(\Gr)$, analogous to Proposition \ref{LABEL660}:
there is a commutative diagram of split exact sequences
\[
\begin{tikzcd}
  0 \arrow[r] &
  \ZZ^{V\setminus S}\oplus \Ha^1(\Gr) \arrow[rr, hook]
  \arrow[d, "(d\ \mathrm{incl})"'] &&
  \ZZ^V\oplus \Ha^1(\Gr) \arrow[rr, "{(\rho^*\ 0)}"]
  \arrow[ll, bend right=15, "\mathrm{restr}\oplus id"']
  \arrow[d] &&
  \ZZ^S \arrow[d, equals] \arrow[r]
  & 0
  \\
  0 \arrow[r] & \ZZ^E \arrow[rr, hook] &&
  \ZZ^E\oplus \ZZ^S \arrow[rr]
  \arrow[ll, bend left=15, "{(id\ -d\circ \rho_*)}"]
  && \ZZ^S \arrow[r]
  & 0
\end{tikzcd}
\]
where the middle vertical map is $(f,g)\mapsto (df+g,f|_S)$,
$\mathrm{restr}$ and $\mathrm{incl}$ are restriction and inclusion
maps, and all other unlabeled maps are the obvious ones.
The following is an immediate consequence of this and Propositions
\ref{LABEL660}, \ref{LABEL720} and \ref{LABEL800}.
 
\begin{prop}
  \label{LABEL820}
  (i) The diagram above induces an isomorphism
  \[
    \frac{\ZZ^E}{d(\ZZ^{V\setminus S}) + \Ha^1(\Gr)} \isomarrow
    \frac{H^1(C^\bull_\fm)}{\im(\Ha^1(\Gr))} = \P_\fm(\Gr)
  \]
  (ii) There is a commutative diagram
  \[
    \begin{tikzcd}
      \ZZ^V \arrow[r] & \smash[b]{\dfrac{\ZZ^V}{\Box_0(\ZZ^{V\setminus S})}}
      \arrow[r, "\sim"] &
      \Clhat_\fm(\Gr) \\
      \ZZ^E \arrow[u, "d\adj"] \arrow[r] &
      \dfrac{H^1(C^\bull_\fm)}{\im(\Ha^1(\Gr))} \arrow[r, "\sim"] &
      \P_\fm(\Gr) \arrow[u, "\chi_\fm"] \ .
      \mbox{\hspace*{10em}$\qed$\hspace*{-11em}}
    \end{tikzcd}
  \]
\end{prop}
We now define, for $\Gr$ finite, a homomorphism
\[
\zeta_\fm \colon \J_\fm(\Gr) \to \P_\fm(\Gr).
\]
Consider the diagram
\[
  \begin{tikzcd}
   \Ha^1(\Grm)\Zdual \arrow[r, "\tr(\iota_\gfm^{\mbox{ }})"] & H^1(\Grm)
    \\
    H_1(\Grm) \arrow[u] \arrow[r, "\iota_\gfm^{\mbox{ }}"]
    & \Ha^1(\Grm) \arrow[u]
    \\
    H_1(\Gr) \arrow[u, "j_*"] \arrow[r, "\iota_\Gr^{\mbox{ }}"]
    \arrow[uu, bend left, shift left=1.5em, "\beta_\fm"]
    & \Ha^1(\Gr) \arrow[u, "j_!"]
    \arrow[uu, bend right, shift right=1.5em, "\alpha_\fm"']
  \end{tikzcd}
\]
in which $\alpha_\fm$ and $\beta_\fm$ are the composites of the
vertical maps. As in the proof of Proposition \ref{LABEL420}, using
Lemma \ref{LABEL280} the top square commutes. The bottom square
commutes by Lemma \ref{LABEL730}(e). So the entire diagram commutes,
inducing the desired map $\zeta_\fm$. Since $\Gr$ is finite then by
Proposition \ref{LABEL320}(b) all the
horizontal maps in the diagram are isomorphisms, and so $\zeta_\fm$ is
an isomorphism.
\begin{thm}
  \label{LABEL830}
  Let $\Gr$ be finite. The diagram
  \[
    \begin{tikzcd}
      \Cl^0_\fm(\Gr)
      \arrow[r, "\iota_\fm", "\sim"']]
      \arrow[d, "\AJ_\fm", "\simeq"'] & \Clhat^0_\fm(\Gr)
      \\
      \J_\fm(\Gr) \arrow[r, "\zeta_\fm", "\sim"']&\P_\fm(\Gr)
      \arrow[u, "\simeq", "\comp_\fm"']
    \end{tikzcd}
  \]
  commutes, and all the arrows are isomorphisms. 
\end{thm}

\begin{proof}
Let $(D,i)\in \Div^0_\fm(\Gr)$, $D=\partial \gamma$ with
$\gamma\in C_1(\Gr)$ and $i\in I$.  Then $\AJ_\fm(D,i)$ is the
class mod $\im(\alpha_\fm)$ of $\lambda_{\fm, D}+\varepsilon_i\in \Ha^1(\Grm)\Zdual$, which
is the map $\omega\mapsto (j^\ast\omega)(\gamma)+\omega(e_i)$. Under
the map $\tr(\iota_\gfm)\colon \Ha^1(\Grm)\Zdual \to H^1(\Grm)$ this
becomes the class of the cocycle $\hat\gamma+\hat i$, where $\hat
i\in\ZZ^I$ is the characteristic function of $i\in I$.

Now the
image of $(D,i)$ in $\Divhat^0_\fm(\Gr)$ is
\[
(\hat{D},\hat{i})=(\widehat{\partial\gamma},\hat{i})=
(d\adj\hat{\gamma},\hat{i})
\]
using Proposition \ref{LABEL280}.  So we have
\[
  \begin{tikzcd}
    (D,i)\ \mathrm{mod} \ \Prin_\fm(\Gr)\arrow[r, mapsto, "\iota_\fm"]
    \arrow[d, mapsto, "\AJ_\fm"] &(d\adj\hat{\gamma},\hat{i})
    \ \mathrm{mod}\ \Prinhat_\fm(\Gr)\\
    (\lambda_{D,\fm}+\varepsilon_i) \ \mathrm{mod}\
    \alpha_\fm(H^1(\Gr))
    \arrow[r, mapsto, "\zeta_\fm"]
    &
    (\hat{\gamma},\hat{i})\ \mathrm{mod}\ \im(\beta_\fm)
    \arrow[u, mapsto, "\comp_\fm"']
  \end{tikzcd}
\]
and since elements $(D,i)$ generate $\Div^0_\fm(\Gr)$,  the diagram commutes.
\end{proof}

\begin{rem}
  In the next section we will define a rigidified Picard group
  $\Pic_\fm(\Grg)$ of $\Grg$. Combining Theorem \ref{LABEL830} and
  Theorem \ref{LABEL900} (or Corollary \ref{LABEL1010}), and the
  definitions of the maps given in the next section, we then obtain
  the following analogue of Theorem \ref{LABEL130}.
\end{rem}

\begin{thm}
  \label{LABEL840}
  There is a commutative diagram
  \[
    \begin{tikzcd}
      \Cl_\fm(\Gr) \arrow[r, "\sim"] & \Clhat_\fm(\Gr)
      \arrow[r, "\sim"', "{-\bar\delta_\fm}"]
      & \Pic_\fm(\Grg)
      \\
      \Cl_\fm^0(\Gr) \arrow[u, hook] \arrow[d, "\AJ_\fm" left, "\wr"]
      \arrow[r, "\iota", "\sim"']
      & \Clhat_\fm^0(\Gr) \arrow[u, hook]  \arrow[r, "\sim"]
      &\Pic_\fm^0(\Grg) \arrow[u, hook]
      \\
      \J_\fm(\Gr) \arrow[rr, "\zeta_\fm" ]
      && \P_\fm(\Gr) \arrow[u, "\wr"] \ .
      \arrow[ul, "\chi_\fm", "\simeq"']
    \end{tikzcd}
  \]
  in which the horizontal arrows are isomorphisms.
\end{thm}


\subsection{Rigidified torsors and
  \texorpdfstring{$\Pic_\fm(\Grg)$}{Pic\_m(|Gr|)}}

\label{LABEL850}
To complete the picture, we introduce a sheaf-theoretic generalized
Picard group $\Pic_\fm(\Grg)$. As in subsection \ref{LABEL470}, we
assume that $\Gr$ is a locally finite graph with no isolated
vertices. Let $\fm$  be a modulus on $\Gr$ given by a map $\rho\colon I
\to V$. Set $S=\rho(I)\subset V$, $w_i=\rho(i)\in S$.

For any $v\in V$ we have the
evaluation map
\[
  \eval_v\colon \Harm_v \to \ZZ.
\]
Let $\cL$ be a $\Harm$--torsor on $\Grg$. We may then contract the
structure group of the stalk $\cL_v$ by
$\eval_v$ gives a $\ZZ$--torsor $\cL(v)$:
\[
\cL(v)=\cL_v\times^{\Harm_v}\ZZ = (\cL_v\times\ZZ)/\sim
\]
where for $c\in \cL_v$, $f\in\Harm_v$, $n\in \ZZ$, $(fc,n)\sim
(f,\eval_v(f)+n)$. If $c\in \cL_v$ we write $c(v)=(c,0)$ for its image
in $\cL(v)$.
\begin{defn}
  An $\fm$-rigidification of $\cL$ is a family of elements
  $s_i \in  \cL(w_i)$, $i\in I$. The group of
  isomorphism classes of $\fm$-rigidified $\Harm$--torsors on $\Grg$ is
  denoted $\Pic_\fm(\Grg)$.
\end{defn}

Any such $s_i$ determines an isomorphism $\ZZ \isom \cL(w_i)$,
$n\mapsto s_i+n$ of $\ZZ$--torsors. There is an obvious surjective map
$\Pic_\fm(\Grg) \to \Pic(\Grg)$, whose kernel is the set of
isomorphism classes of rigidifications of the trivial torsor
$\cL_{\mathrm{triv}}=\Harm$, hence equal to the cokernel of the component homomorphism
$\ZZ^{\pi_0(\Gr)} \to \ZZ^I$. We write $\Pic^0_\fm(\Grg)\subset
\Pic(\Grg)$ for the inverse image of $\Pic^0(\Grg)$.

\begin{rem}
  \label{LABEL860}
  Parallel to the discussion of Remark \ref{LABEL670}, the generality
  of non-reduced moduli is illusory: let $\fm_0$
  be the reduced modulus given by $S$.  Fix a $\Harm$--torsor
  $\cL$. Then $\ZZ^I$ (resp.~$\ZZ^S$) acts transitively on the set of
  $\fm$-rigidifications (resp.~$\fm_0$-rigidifications) of $\cL$, and
  the former set is the pushout of the latter via the map
  $\tr\rho\colon \ZZ^S \to \ZZ^I$.
\end{rem}

For the rest of this section we will assume that $\fm$ is a reduced
modulus; we postpone further discussion of the general case (which
involves more delicate homological algebra) to Section \ref{LABEL910}
below.

Let $\Harmp$ be the kernel
\[
  \Harmp = \ker\bigl( \Harm \xrightarrow{\eval_I} \prod_{i\in I}
  w_{i*}\ZZ \bigr)
\]
in degrees $0$ and $1$, where $\eval_I=(\eval_{w_i})_i$.

\begin{prop}
  There is a canonical isomorphism $\Pic_\fm(\Grg) \isom H^1(\Grg, \Harmp)$.
\end{prop}

\begin{proof}
  Elements of $H^1(\Grg, \Harmp)$ are is bijection with
  isomorphism classes of $\Harmp$--torsors on $\Grg$. Let $\cL_0$
  be a $\Harmp$--torsor.  Then the pushout
  $\cL=\cL_0\times^{\Harmp}\Harm$ is a $\Harm$--torsor, and it
  carries a canonical $\fm$-rigidification: for every $v\in S$,
  \[
  \cL(v)=(\cL_0\times^{\Harmp}\Harm)(v)
  =\cL_{0,v}\times^{\Harmpv}\ZZ = \ZZ
  \]
  since the contracting homomorphism $\Harmpv\to \Harm_v \to \ZZ$ is zero.

  In the other direction, let $\cL$ be a $\Harm$--torsor on $\Grg$ and
  $(s_i)$ an $\fm$-rigidification. Let $\cL_0\subset \cL$ be the
  subsheaf which differs from $\cL$ only at points $w_i\in S$, where
  its stalk is
  \[
  \cL_{0,w_i}=\{c \in  \cL_{w_i} \mid c(w_i)=s_i \}. 
  \]
  Then $\cL_0$ is a $\Harmp$--torsor, and the two constructions are
  easily seen to be mutual inverses.
\end{proof}
Let $k\colon \Grg\setminus S \hookrightarrow \Grg$ be the
inclusion. Define sheaves
\[
\ZZpm=k_!\ZZ =\ker\Bigl(\ZZ \to \prod_{i\in I} w_{i*}\ZZ \Bigr)
,\quad
\PLm=\ker\Bigl(\PL \to \prod_{i\in I} w_{i*}\ZZ \Bigr)
\]
(with evaluation as differential). These are sheaves of abelian groups
on $\Grg$ satisfying property (C) of Section \ref{LABEL470}. The
exact sequences \eqref{LABEL520}, \eqref{LABEL580} restrict to give exact
sequences
\begin{equation}
  \label{LABEL870}
  0 \to \Harmp \to \PLm \xrightarrow{\diff} \smash[b]{ \prod_{v\in V} v_*\ZZ}
  \to 0 \phantom{\prod}
\end{equation}
and 
\begin{equation}
  \label{LABEL880}
  0 \to \ZZpm \xrightarrow{\ q_\fm\ } \Harmp \xrightarrow{\ediff} \Omega \to 0.
\end{equation}
Since the evaluation map $H^0(\Grg,\PLm) \to \ZZ^S$ is surjective, the
cohomology $H^1(\Grg,\PLm)$ vanishes and \eqref{LABEL870} then gives
an exact sequence
\begin{equation}
  \label{LABEL890}
  0 \to H^0(\Grg,\Harmp)  \to \ZZ^{V\setminus S}=H^0(\Grg,\PLm)
  \xrightarrow{\diff} \ZZ^V \xrightarrow{\delta_\fm} H^1(\Grg,\Harmp)
  \to 0.
\end{equation}
By Proposition \ref{LABEL540}, the map $\diff$ equals $\Box_0$,
and $\delta_\fm$ therefore induces an isomorphism
$\bar\delta_\fm\colon \Clhat_\fm(\Gr)=\ZZ^V/\Box_0(\ZZ^{V\setminus S})
\isomarrow \Pic_\fm(\Grg)$ (the first equality by Proposition
\ref{LABEL660}).

The \v Cech isomorphism Proposition \ref{LABEL490}(b) for $\ZZpm$ is an isomorphism
$H^1(\Gr,\ZZpm) \isomarrow H^1(C^\bull_\fm)$, and 
the inclusion of sheaves $q_\fm$ therefore
induces, by Proposition \ref{LABEL820}(a), a homomorphism $q_{\fm *}\colon\P_\fm(\Gr)
\to \Pic_\fm(\Grg)$.

\begin{thm}
  \label{LABEL900}
  $\bar\delta_\fm(\Clhat^0_\fm(\Gr))=\Pic^0_\fm(\Grg)$, and the
  diagram
  \[
    \begin{tikzcd}
      \Clhat_\fm(\Gr) \arrow[r, "-\bar\delta_\fm", "\sim"']& \Pic_\fm(\Grg)
      \\
      \Clhat^0_\fm(\Gr) \arrow[u, hook]\arrow[r, "\lowsim"]
      &\Pic^0_\fm(\Grg) \arrow[u, hook] 
      \\
      & \P_\fm(\Gr) \arrow[u, "\wr", "q_{\fm *}"']
      \arrow[ul, "\comp_\fm", "\isom"']
    \end{tikzcd}
  \]
  commutes.  
\end{thm}
\begin{proof}
  This is exactly the same as the proof of Theorem \ref{LABEL600}, in
  view of the above Propositions.
\end{proof}

\subsection{Rigidified torsors (general case)}
\label{LABEL910}

For completeness (but see Remark \ref{LABEL860} above), and at the
expense of further homological algebra, we generalize the
results of Section \ref{LABEL850} to arbitrary moduli. So $\Gr$ will
be a locally finite graph with no isolated vertices, and $\fm$ a
modulus on $\Gr$, now not necessarily reduced.

Let $\Harm_\fm$ be the complex of sheaves
\[
  \Harm_\fm = \Bigl[\ \Harm \xrightarrow{\eval_I} \prod_{i\in I}
  w_{i*}\ZZ\ \Bigr]
\]
in degrees $0$ and $1$, where $\eval_I=(\eval_{w_i})_i$. If the
modulus is reduced, then this is quasi-isomorphic to the sheaf
$\Harmp$ considered in the previous subsection.

Standard sheaf theory gives an isomorphism
\[
  \Pic_\fm(\Grg) \isom H^1(\Grg, \Harm_\fm)
\]
extending the isomorphism $\Pic(\Grg)\isom H^1(\Grg,\Harm)$, which we now
recall.

More generally, let $X $ be a topological space, $\phi\colon H \to G$
a morphism of abelian sheaves on $X$. Then there is a canonical
isomorphism
\[
  H^1(X,[H \to G]) \isom \left\{ \parbox{24em}{isomorphism classes of pairs
    $(T,s)$, with $T$ an $H$--torsor on $X$, and $s\in H^0(X,\phi_*T)$ a
    trivialization of $\phi_*T$.} \right\}.
\]
(Recall that $\phi_*T$ is the $G$--torsor which is the quotient of
$T\times G$ by $H$, acting antidiagonally.)

If $G=\{0\}$ then this is the isomorphism following Definition
\ref{LABEL550}. The explicit construction (with correct sign) is as
follows: let $H\hookrightarrow I$ be an injection into an acyclic
sheaf. Then we have an extension of sheaves
\[
  0 \to H \to I \xrightarrow{\pi} Q=I/H \to 0
\]
whose long exact cohomology sequence gives a surjection
$\delta\colon H^0(X,Q) \to H^1(X,H)$ inducing an isomorphism
$\bar\delta\colon \im\bigl(H^0(X,I) \to H^0(X,Q)\bigr) \isom
H^1(X,H)$. Let $x\in H^0(X,Q)$ map to $\xi\in H^1(X,H)$. Then
$T_\xi=\pi\i(x)$ is the $H$--torsor (up to isomorphism) determined by
$\xi$. In the case of a graph and $H=\Harm$, we take $I=\PL$,
$Q=\prod_V v_*\ZZ$ and then $\delta$ and $\bar\delta$ are the maps in
subsection \ref{LABEL470}.

In general, let $P$ be the pushout of $\phi$ and $H\to I$, which fits
into a commutative diagram with exact rows:
\[
  \begin{tikzcd}
    0 \arrow[r] &
    H \arrow[d, "\phi"] \arrow[r, hookrightarrow] &
    I \arrow[r, "\pi"] \arrow[d] & Q \arrow[r] \arrow[d, equals] & 0
    \\
    0 \arrow[r] & G \arrow[r, hookrightarrow] &
    P \arrow[r, "\psi"] & Q \arrow[r] & 0
  \end{tikzcd}
\]
The vertical maps are a quasi-isomorphism between the two rows, so
$H^1(X,[H\to G]) \simeq H^1(X, [I \to P])$.

Write $\delta_\phi\colon H^0(X,P)\to H^1(X,[I\to
P])$ for the natural map. Then $\delta_\phi$ is surjective, since
$H^1(X,I)=0$, inducing an isomorphism
\[
  \bar\delta_\phi \colon  \coker\bigl( H^0(X,I) \to H^0(X,P)\bigr)
  \isom H^1(X,[H\to G]).
\]
Let $y\in H^0(X,P)$ be the preimage of $\eta\in
H^1(X,[H\to G])$. Let $x\in H^0(X,Q)$ be the image of $y$. Then $x$
and $\eta$ have the same image $\xi\in H^1(X,H)$. The pushout of the
torsor $T_\xi$ by $\phi$ equals the $G$--torsor $\psi\i(x)$, and $y$ is
a trivialization of it.

In the present case, the complex $H\to G$ is $\Harm_\fm$, and for
$I$ we may take $\PL$. Since the evaluation map $\eval_I\colon \Harm \to \prod_i
w_{i*}\ZZ$ extends to $\PL$, the pushout may be written as
\begin{equation}
  \label{LABEL920}
  \begin{tikzcd}
    \Harm \arrow[d, hookrightarrow] \arrow[r,"{\eval_I}"]
    & \prod_{i\in I} w_{i*}\ZZ
    \arrow[d, hookrightarrow, "{(0,\id)}"]
    \\
    \PL \arrow[r, "{(\diff,\eval_I)}"] & \prod_{v\in V}v_*\ZZ \oplus
    \prod_{i\in I} w_{i*}\ZZ\,.
  \end{tikzcd}
\end{equation}
From this and the previous discussion we obtain a surjection
$\delta_\fm \colon \ZZ^V\oplus\ZZ^I \to H^1(\Grg,\Harm_\fm)$.
\begin{prop}
  \label{LABEL930}
  $\delta_\fm$ induces an isomorphism
  \[
    \bar\delta_\fm \colon
    \frac{\ZZ^V \oplus \ZZ^I}{(\Box_0,\tr \rho)(\ZZ^V)}
    \isom H^1(\Grg,\Harm_\fm) \simeq \Pic_\fm(\Grg). \qed
  \]
\end{prop}
We have further complexes of sheaves on $\Grg$:
\begin{align*}
  \ZZ_\fm &= \bigl[\ \ZZ \to  \prod_{i\in I} w_{i*}\ZZ\ \bigr] \\
  \PL_\fm &= \bigl[\ \PL \to  \prod_{i\in I} w_{i*}\ZZ\ \bigr] 
\end{align*}
with evaluation as differential,
and exact sequences of complexes
\begin{equation}
  \label{LABEL940}
  0 \to \Harm_\fm \to \PL_\fm \to \prod_{v\in V} v_*\ZZ[0] \to 0
\end{equation}
and
\begin{equation}
  \label{LABEL950}
  0 \to \ZZ_\fm \to \Harm_\fm \to \Omega[0] \to 0.
\end{equation}
(Again, if $\fm$ is reduced these are quasi-isomorphic to the sheaves
denoted by the same symbols in Section \ref{LABEL850}.)

\begin{prop}
  \label{LABEL960}
  $H^1(\Grg,\ZZ_\fm) \isom H^1(\Gr_\fm)$.
\end{prop}

\begin{proof}
  $H^1(\Grg,\ZZ_\fm)$ is the cohomology of the mapping fibre of
  $C^\bullet(\Gr) \to \ZZ^I[0]$, which is the complex
  \begin{align*}
    \ZZ^V & \to \ZZ^E \oplus \ZZ^I \\
    f&\mapsto (df, f\circ \rho)
  \end{align*}
  whereas $C^\bullet(\Gr_\fm)$ is the complex
  \begin{align*}
    \ZZ^V\oplus\ZZ & \to \ZZ^E \oplus \ZZ^I \\
    (f,a)&\mapsto (df, f\circ \rho-a)
  \end{align*}
  and the obvious inclusion of complexes induces an isomorphism on
  $H^1$ (cf.~the proof of Proposition \ref{LABEL800}).
\end{proof}
The inclusion of complexes $\ZZ_\fm \to \Harm_\fm$ induces a map
$(**)$ on cohomology.
We now have the version with modulus of Theorem \ref{LABEL600}.
\begin{thm}
  \label{LABEL970}
  The diagram
  \[
    \begin{tikzcd}
      C^0_\fm(\Gr)=\ZZ^V\oplus\ZZ^I \arrow[r, "-\delta_\fm"] & \Pic(\Grg)
      \\
      C^1(\Gr_\fm)=\ZZ^E\oplus \ZZ^I \arrow[u, "{(d\adj,\id)}"] \arrow[r] & H^1(\Gr_\fm)
      \arrow[u, "(**)"'] 
    \end{tikzcd}
  \]
  commutes.
\end{thm}

\begin{proof}
  Morally, $\delta_\fm$ is the connecting homomorphism for the sequence
  \[
     \Harm_\fm \to \PL \to \prod_{v\in V}v_*\ZZ \oplus
    \prod_{i\in I}w_{i*}\ZZ 
  \]
  but this is not an exact sequence of complexes --- it only
  determines a triangle in the derived category. So we need to do some
  homological algebra to pin down the signs. We follow the
  sign conventions of \cite[Tag 0119]{Stacks}. First, we take the
  pushout square of complexes of abelian groups
  \begin{equation}
    \label{LABEL980}
    \begin{tikzcd}
      C^\bull(\Harm) \arrow[d, hookrightarrow] \arrow[rr,"{\eval_I}"]
      && \ZZ^I
      \arrow[d, hookrightarrow, "{(0,\id)}"]
      \\
      C^\bull(\PL) \arrow[rr, "{\gamma=(\diff,\eval_I)}"] && \ZZ^V
      \oplus \ZZ^I
    \end{tikzcd}
  \end{equation}
  obtained by applying the standard resolution to
  \eqref{LABEL920}. Recall that for a morphism
  $f\colon X^\bull \to Y^\bull$ of complexes, the total complex of $f$
  is
  \[
    (\tot(f)^\bull, d_{\tot(f)}) = \left( X^\bull\oplus Y^{\bull-1},
      \begin{pmatrix}
        d_X&0\\ f& -d_Y
      \end{pmatrix} \right).
  \]
  Denote by $\incl\colon X^\bull \to \tot(f)^\bull$, $\proj \colon
  \tot(f)^\bull \to Y[-1]^\bull$ the obvious inclusion and projection.
  By the definitions, $\tot(f)=\cone(-f)[-1]$. So applying
  \cite[Tag 014L]{Stacks} (see also the remark following Tag 014Q), the
  sequence of morphisms
  \[
    X^\bull \xrightarrow{-f} Y^\bull \xrightarrow{\incl}
    \tot(f)[1]^\bull \xrightarrow{-\proj} X[1]^\bull
  \]
  is a distinguished triangle in the homotopy category.  Then applying
  the axiom TR2 of triangulated categories (rotation of triangles), we
  see that the triangle
  \begin{equation}
    \label{LABEL990}
    \tot(f)^\bull \xrightarrow{-\proj} X^\bull \xrightarrow{f} Y^\bull
    \xrightarrow{-\incl} \tot(f)[1]^\bull
  \end{equation}
  is also distinguished.
  In the pushout diagram \eqref{LABEL980}, the total complex of the top
  row is the complex $C^\bull(\Harm_\fm)$ whose cohomology is
  $H^*(\Grg, \Harm_\fm)$.  Since \eqref{LABEL980} is a pushout square
  and the vertical maps are injective, the induced map on total
  complexes $C^\bull(\Harm_\fm) \to \tot(\gamma)^\bull$ is a
  quasi-isomorphism. This, together with the final map in
  \eqref{LABEL990}, are the solid arrows in the diagram
  \begin{equation*}
    \begin{tikzcd}[column sep=7em]
      \ZZ^V\oplus \ZZ^I \arrow[r, "{(x,y)\mapsto(0,-x,-y)}"]
      \arrow[dr, dashed]
      &
      \bigl[ \, C^0(\PL) \xrightarrow{(-d_{\PL},-\diff,-\eval_I)}
      C^1(\PL) \oplus \ZZ^V \oplus \ZZ^I \, \bigr]
      \\
      & \bigl[ \, C^0(\Harm) \xrightarrow[(-d_{\Harm}, -\eval_I)]{}
      C^1(\Harm) \oplus \ZZ^I \, \bigr] \arrow[u, "\text{q-iso}"]
    \end{tikzcd}
  \end{equation*}
  in which the complexes on the right are degrees $-1$ and $0$, and the
  vertical map is given by the obvious inclusions.  The dashed arrow
  is a morphism in the derived category, and on cohomology it induces
  $\delta_\fm\colon \ZZ^V\oplus\ZZ^I \to
  H^1(\Grg,\Harm_\fm)$. Explicitly, let $c=(a,k)\in \ZZ^V
  \oplus \ZZ^I$ and $f\in C^0(\PL)$ with $\diff(f)=a$. Then $c$ maps to
  the triple
  \[
    (0,-\diff(f),-k)= (d_{\PL}(f),0,\eval_I(f)-k)-d_{\tot(\alpha')}(f)
    \in C^1(\PL)\oplus \ZZ^V \oplus \ZZ^E
  \]
  and so $\delta_\fm\colon \ZZ^V \oplus \ZZ^E \to H^1(\Gr,\Harm_\fm)$
  takes $(a,k)$ to  $(d_{\PL}(f),\eval_I(f)-k)\in
  C^1(\Harm)\oplus\ZZ^I=C^1(\Harm_\fm)$. 
  If $a=d\adj g$ for $g\in \ZZ^E$, then as in the proof of Theorem
  \ref{LABEL600} we may take for $f$ the family $-h=(-h_v)$ which satisfies
  \[
    \diff(h)=-f,\quad d_{\PL}(h)=g,\quad \eval_I(h)=0
  \]
  and therefore  $\delta_\fm(d\adj g,k)=(-g,-k)$, which lies in the
  image of $C^1(\Grg,\ZZ_\fm)$.
\end{proof}

\begin{cor}
  \label{LABEL1000}
  The diagram
  \begin{equation*}
    \begin{tikzcd}
      \Clhat_\fm(\Gr) \arrow[r, "\sim"', "-\bar\delta_\fm"] & \Pic_\fm(\Grg) \\
      \Clhat^0_\fm(\Gr)=\dfrac{d\adj(\ZZ^E)\oplus\ZZ^I}{(\Box_0,\tr\rho)(\ZZ^V)}
      \arrow[u, hookrightarrow] &
      \dfrac{H^1(\Gr_\fm)}{\beta_\fm(\ker d\adj)} = \P_\fm(\Gr)
      \arrow[l, "\sim", "\chi_\fm"'] \arrow[u, hookrightarrow, "(**)"]
    \end{tikzcd}
  \end{equation*}
  commutes, where $\chi_\fm$ is induced by $(d\adj,\id)\colon
  C^1(\Gr_\fm)=\ZZ^E\oplus \ZZ^I \to \ZZ^V\oplus\ZZ^E=\Divhat_\fm(\Gr)$.
\end{cor}

\begin{proof}
  The proof of Proposition \ref{LABEL800} shows, without assuming $\Gr$ to
  be finite, that the lower map is an isomorphism. The rest follows from
  Theorem \ref{LABEL970}.
\end{proof}

Now assume that $\Gr$ is connected, and define
$\Pic_\fm^0(\Grg)$ to be the inverse image of $\Pic^0(\Grg)$ under the
surjection $\Pic_\fm(\Grg) \to \Pic(\Grg)$.
\begin{cor} 
  \label{LABEL1010}
  The diagram below commutes:
  \[
  \begin{tikzcd}
    \Clhat_\fm(\Gr) \arrow[r, "-\bar\delta_\fm", "\sim"']& \Pic_\fm(\Grg)
    \\
    \Clhat_\fm^0(\Gr) \arrow[u, hook]\arrow[r, "\lowsim"]
    &\Pic_\fm^0(\Grg) \arrow[u, hook] 
    \\
    & \P_\fm(\Gr) \arrow[u, "\wr"]
    \arrow[ul, "\chi_\fm"', "\simeq"]\,.
  \end{tikzcd}
  \] 
\end{cor}

\begin{proof}
  This follows since for $\Gr$ connected, $\im
  d\adj=\Divhat^0(\Gr)$ and $\ker d\adj=\Ha^1(\Gr)$.
\end{proof}

\subsection{Duality}

\label{LABEL1020}
Let $\fm=\sum_{P\in S}P$ be a reduced modulus on a curve $X$.
Then $\P_\fm(X)$ is an extension
\begin{equation}
  \label{LABEL1030}
  0\longrightarrow \T_\fm\longrightarrow \P_\fm(X)\longrightarrow \P(X)
  \longrightarrow 0
\end{equation}
of the Picard variety $\P(X)$ by the torus
$\T_\fm = (\CC^\times)^S/\CC^\times$ (each
element of $(\CC^\times)^S$ determines a trivialization along $S$
of the trivial line bundle $\cO_X$).  Applying the functor
$\Hom(\,\bullet\, ,\Gm)$ 
we get a connecting morphism
\begin{equation}
\label{LABEL1040}
\ZZ[S]_0=\Hom(\T_\fm,\Gm)\longrightarrow \Ext(\P(X),\Gm)=
\J(X)\text{ (canonically)},
\end{equation}
which has the following well-known description.
\begin{thm}
\label{LABEL1050}
The map \eqref{LABEL1040} equals \textup{(}up to sign\textup{)} the
Abel--Jacobi map restricted to $\ZZ[S]_0\subset \Div^0(X)$.\qed
\end{thm}

We have an analogue of this for (connected) graphs.
Here is the analogue of \eqref{LABEL1030}.

\begin{prop}
  \label{LABEL1055}
  The restriction map $j^*\colon C^\bullet(\Gr_\fm) \to C^\bullet(\Gr)$ induces
  an exact sequence
  \begin{equation} \label{LABEL1060}
    0 \to \ZZ^I/\ZZ \to \P(\Gr_\fm) \to \P(\Gr) \to 0.
  \end{equation}
\end{prop}

\begin{proof}
  The map of complexes $j^*$ is surjective, with kernel the complex
  $\ZZ \xrightarrow{\mathrm{diag}} \ZZ^I$, giving an exact sequence
  \[
    0 \to \ZZ^I/\ZZ \to H^1(\Gr_\fm) \to H^1(\Gr) \to 0
  \]
  from which the result follows from the definitions of $\P(\Gr)$ and
  $\P(\Gr_\fm)$. 
\end{proof}
Applying $\Hom(\,\bullet \, ,\ZZ)$ then gives a connecting 
homomorphism
\[
\ZZ[I]_0=\Hom(\ZZ^I/\ZZ,\ZZ)
\to \Ext(\P(\Gr),\ZZ)\cong
\Hom(\P(\Gr),\QQ/\ZZ)
\]
classifying the extension \eqref{LABEL1060} of abelian groups. We have the
duality $\Hom(\P(\Gr),\QQ/\ZZ)\simeq \J(\Gr)$ from
\eqref{LABEL370}(b), and composing with this gives a map
$\ZZ[I]_0\to \J(\Gr)$ analogous to \eqref{LABEL1040}.

\begin{thm}
\label{LABEL1070}
Up to sign, the map $\ZZ[I]_0\to \J(\Gr)$ above equals the
composite
\[
\ZZ[I]_0\stackrel{b}{\to}\Div^0(\Gr)\stackrel{\AJt}
{\longrightarrow} \J(\Gr).
\]
\end{thm}
\begin{proof}
First recall: let 
\[
0\to \ZZ\stackrel{i}{\to} A'\xrightarrow{\pi} A\to 0
\]
be an exact sequence of abelian groups, with $A$ torsion.  Its 
extension class in $\Hom(A,\QQ/\ZZ)\simeq \Ext(A)$ may 
(up to sign) be computed as follows:  For $x\in A$, let
$\tilde{x}\in\pi\i(x)$ and $n\geq 1$ with $nx=0$.  Then
$n\tilde{x}=i(m)$ for some $m\in\ZZ$, and $x\mapsto m/n$ is 
the desired homomorphism $A\to \QQ/\ZZ$.

For the extension \eqref{LABEL1060}, let $x\in\P(\Gr)$, $\omega\in C^1(\Gr)$
whose class in $\P(\Gr)$ equals $x$.  Let $n\geq 1$ with $nx=0$.  Then
$n\omega=\eta + df$ for some $\eta\in \Ha^1(\Gr)$, $f\in C^0(\Gr)$.
We may take $\tilde{x}$ to be the class of 
\[
j_!\omega = (\omega, 0)\in C^1(\Gr_\fm)=C^1(\Gr)\oplus \ZZ^I .
\]
Then
\[
n\tilde{x}=\textup{ class of }(\eta +df, 0)=\textup{ class of }(df,0)
\]
Since $(\eta,0)=j_!\eta\in\im(\beta_\fm)$.  As
\[
d_\fm (f,a)=(df,\tr{\rho} f-a)\in C^1(\Gr_\fm)=\ZZ^E\oplus \ZZ^I
\textup{ for all } (f,a)\in \ZZ^V\oplus \ZZ=C^0(\Gr_\fm),
\]
we see that $n\tilde{x}$ is also equal to the class of $(0,-\tr{\rho} f)$.
In other words, the connecting map
\[
\ZZ[I]_0\to \Hom(\P(\Gr), \QQ/\ZZ)\text{ is given by }
D\mapsto \bigl( x\mapsto -1/n(\tr{\rho} f)(D)\bigr)
\]
and the claim is that (up to sign) this equals the map
\[
D\mapsto \Bigl( x\mapsto \bigl(\AJt(\rho_\ast D),x\bigr)_{\mathrm{Pont}}\Bigr),
\]
where $\pairing[\mathrm{Pont}]{-}{-}\colon \J(\Gr)\times \P(\Gr)\to
\QQ/\ZZ$ is the Pontryagin duality pairing.

Compute both expressions:  let $D\in\ZZ[I]_0$, so that $\rho_\ast D= d \gamma
\in\Div^0(\Gr)$ for some $\gamma\in C_1(\Gr)$. Then
\[
\AJt(\rho_\ast D)=(\omega\mapsto \omega(\gamma))\in
\Ha^1(\Gr)\Zdual \ \text{mod}\ \alpha(H_1(\Gr)).
\]
The pairing
\[
\pairing[\mathrm{Pont}]{-}{-}\colon \J(\Gr)\times \P(\Gr)=
\frac{\Ha^1(\Gr)\Zdual}{\alpha(H_1(\Gr))}\times
\frac{H^1(\Gr)}{\beta(\Ha^1(\Gr))}\to \QQ/\ZZ
\]
is given as follows:  if $\varphi\in\Ha^1(\Gr)\Zdual$, then
\[
(\varphi\ \text{mod}\ \im(\alpha),x)_{\mathrm{Pont}}=\frac{1}{n}\varphi(\eta)
\,\,\textup{mod}\ \ZZ
\]
with $n$, $\eta$ as above.

So 
\[
(\AJt(\rho_\ast D),x)_{\mathrm{Pont}}=\frac{1}{n}\eta(\gamma)\,\,
\textup{mod}\ \ZZ\ .
\]
On the other hand, $(\tr{\rho} f)(D)=f(\partial\gamma)=df(\gamma)$
and since $n\omega =\eta + df$,
\[
  (\tr{\rho} f)(D)= - \eta(\gamma) + n\omega (\gamma),
\]
so
\[
 -\frac{1}{n}(\tr{\rho} f)(D)
               = \frac{1}{n}\eta(\gamma)\,\,\textup{mod}\ \ZZ\ .
                                             \qedhere
\]
\end{proof}

\subsection{An abstract formulation}
\label{LABEL1080}

To construct the various groups and obtain the relations between them,
we can in fact work with a rather minimal algebraic setup:

\begin{itemize}
\item $C_0 \xleftarrow{\,\partial\,} C_1$ is a $2$-term complex of
  (arbitrary) abelian groups.
\item For $p\in\{0,1\}$, $\pairing[p]--\colon C_p\times C_p\to
  \ZZ$ is a symmetric bilinear form (otherwise arbitrary).
\end{itemize}
There is just one assumption:
\begin{itemize}
  \item There exists a formal adjoint $\partial\adj \colon C_0\to C_1$
    of $\partial$, that is, a homomorphism satisfying
    $\pairing[0]x{\partial(y)} = \pairing[1]{\partial\adj(x)}y$ for
    all $x\in C_0$, $y\in C_1$.
\end{itemize}
Define $C^p=(C_p)\Zdual=\Hom(C_p,\ZZ)$, with  $\apairing[p]-- \colon
C_p \times C^p \to \ZZ$ the duality pairing. The bilinear forms
$\pairing[p]--$ induce homomorphisms $\iota_p\colon C_p \to C^p$ such
that $\pairing[p]xy=\apairing[p]x{\iota_p(y)}$.

Let $d=\tr{\partial}\colon C^0 \to C^1$, $d\adj=\tr{(\partial\adj)}\colon
C^1\to C^0$. Define ``Laplacians''
\begin{gather*}
  \Delta_0=\partial\partial\adj \in \End C_0,\quad
    \Delta_1=\partial\adj\partial \in \End C_1,
  \Box_0=d\adj d\in \End C^0,\quad \Box_1=dd\adj\in\End C^1.
\end{gather*}

\begin{lem}
  \label{LABEL1090}
  $\iota_0\circ\partial=d\adj\circ\iota_1$,
  $\iota_1\circ\partial\adj=d\circ\iota_0$, and
  $\iota_0\circ\Delta_0=\Box_0\circ\iota_0$.
\end{lem}
\begin{proof}
  For the first statement: if $x\in C_1$, $\iota_0(\partial x)$ is the map
  $y\mapsto \pairing[0]{\partial x}y$, whereas
  $d\adj(\iota_1(x))=\iota_1(x)\circ \partial\adj$ is
  the map $y\mapsto \pairing[1]x{\partial\adj y}$, and these maps are
  equal. The second statement is similar, and together they give the third.
\end{proof}

For any subgroups $A \subset C_1$, $B\subset C^1$ the duality
restricts to a pairing $A\times B \to \ZZ$. Write
\[
  \alpha \colon A \to B\Zdual,\quad \alpha' \colon B \to A\Zdual
\]
for the induced maps (for any $A$, $B$).
\begin{lem}
    \label{LABEL1100}
  Suppose $\iota_1(A)\subset B$. Then the square
  \[
    \begin{tikzcd}
      B\Zdual \arrow[r, "{\tr{(\iota_1)}}"] & A\Zdual 
      \\
      A \arrow[r, "\iota_1"'] \arrow[u, "\alpha"] &
      B \arrow[u, "{\alpha'}"']
    \end{tikzcd}
  \]
  commutes.
\end{lem}
\begin{proof}
  Each of the composite maps equals the map $A\to C_1
  \xrightarrow{\iota_1} C^1=(C_1)\Zdual \to  A\Zdual$.
\end{proof}

Define
\[
  \Clb=\dfrac{C_0}{\im \Delta_0}
  \supset \Clb^0 = \dfrac{\im \partial}{\im \Delta_0}\,,
  \quad
  \Clhb = \dfrac{C^0}{\im \Box_0}
  \supset \Clhb^0 = \dfrac{\im d\adj}{\im \Box_0}\,.
\]
By Lemma \ref{LABEL1090}, $\iota_0$ induces a map $\iota_0\colon \Clb \to \Clhb$, and
$\iota_0(\Clb^0) \subset \Clhb^0$.
\begin{lem}
  \label{LABEL1110}
  The maps $\partial$, $d\adj$ induce isomorphisms $\psi$, $\chi$ fitting into a
  commutative diagram
  \[
    \begin{tikzcd}
      \dfrac{C_1}{\ker\partial + \im \partial\adj}
      \arrow[r, "\iota_1"] \arrow[d, "\wr", "\psi"'] &
      \dfrac{C^1}{\ker d\adj + \im d} \arrow[d, "\wr"', "\chi"]
      \\
      \Clb^0 \arrow[r, "\iota_0"] &
      \Clhb^0.
    \end{tikzcd}
  \]
\end{lem}
This follows from the definitions and Lemma \ref{LABEL1090}.\qed

In Lemma \ref{LABEL1100} take $A=\ker\partial$, $B=\ker d\adj \supset
\iota_1(\ker\partial)$ (the inclusion by Lemma \ref{LABEL1090}), and define
\[
  \Jb = \dfrac{(\ker d\adj)\Zdual}{\alpha(\ker\partial)},\qquad
  \Pbbar = \dfrac{(\ker \partial)\Zdual}{\alpha'(\ker d\adj)}\,.
\]
Then $\tr{(\iota_1)}$ induces a homomorphism $\zeta\colon \Jb \to \Pbbar$.

Since $\apairing[1]{\ker\partial}{\im d}=0$, duality gives a map
$H^1 = C^1/\im d \to (\ker\partial)\Zdual = (H_1)\Zdual$. Let $\beta$
be the composite
$\beta\colon \ker d\adj \hookrightarrow C^1 \to H^1$, and define
\[
  \Pb = \dfrac{H^1}{\beta(\ker d\adj)} = \dfrac{C^1}{\ker d\adj + \im d}\,.
\]
Then $H^1\to (H_1)\dual$ gives a map $\Pb \to \Pbbar$. 

\begin{thm}
  \label{LABEL1120}
  \begin{enumerate}[\upshape(a)]
  \item There is a unique homomorphism $\AJ \colon \Clb^0 \to \Jb$
    which, for $x\in C_1$, takes the class of $\partial x$ to the
    class of the linear form $y\mapsto \apairing[1]xy$,
    $y\in \ker d\adj$.

  \item The diagram below is commutative:
    \[
      \begin{tikzcd}
        \Clb^0 \arrow[rr, "\iota_0"] \arrow[d, "\AJ"] &&
        \Clhb^0 \arrow[d, "\wr"', "{\chi^{-1}}"]
        \\
        \Jb \arrow[r,"\zeta"] &\Pbbar&
        \Pb\,. \arrow[l]
      \end{tikzcd}
    \]
  \end{enumerate}
\end{thm}
\begin{proof}
  The square
  \[
    \begin{tikzcd}
      C_1 \arrow[r, "\iota_1"]  \arrow[d, "\alpha"] &
      C^1 \arrow[d, "\alpha'"]
      \\
      (\ker d\adj)\Zdual \arrow[r, "{\tr{(\iota_1)}}"] & (\ker\partial)\Zdual
    \end{tikzcd}
  \]
  commutes by Lemma \ref{LABEL1100}. By the definitions, $\alpha(\im \partial\adj)=0$. So
  passing to the quotients we get a commutative square
  \[
    \begin{tikzcd}
      \dfrac{C_1}{\im\partial\adj+\ker\partial}
      \arrow[r, "\iota_1"]  \arrow[d] &
      \Pb \arrow[d]
      \\
      \Jb \arrow[r, "\zeta"] & \Pbbar\,.
    \end{tikzcd}
  \]
  Combining this with Lemma \ref{LABEL1110} gives the existence of $\AJ$ in (a) and the
  commutativity of the diagram in (b).
\end{proof}

Now assume that we are in addition given an abelian group $L$ together
with a homomorphism $\rho \colon L \to C_0$ and a symmetric bilinear
form $\pairing[L]--\colon L \times L \to \ZZ$. We also assume that
$\rho$ has a formal adjoint $\rho\adj \colon C_0 \to L$. Let
$M=L\Zdual$, $r=\tr \rho$ and $r\adj = \tr{(\rho\adj)}$. Let
$\iota_L\colon L \to M$ be the map induced by the pairing
$\pairing[L]--$.

Let $\tilde C_\bullet$ be
the complex
\[
  \tilde C_0 = C_0 \xleftarrow{\partial_L = (\partial\ \rho)}
  \tilde C_1 = C_1 \oplus L
\]
and equip $\tilde C_1$ with the obvious symmetric bilinear form.

Then a formal adjoint to $\partial_L$ is the map $\partial_L\adjp
\defeq (\partial\adj,\rho\adj)\colon \tilde C_0 \to \tilde C_1$. Let
\[
  \tilde C^0 = C^0 \xrightarrow{d_L = (d, r)}
  \tilde C^1 = C^1 \oplus M
\]
be the dual complex, and set $d_L\adjp \defeq \tr{(\partial_L)}= (d\adj\ r\adj) \colon
\tilde C^1 \to \tilde C^1$, which is a formal adjoint to $d_L$.

Define
\[
  \Clb_L=\dfrac{C_0\oplus L}{(\Delta_0,\rho\adj)C_0}
  \supset \Clb_L^0 = \dfrac{\im \partial\oplus L}{(\Delta_0,\rho\adj)C_0}\,,
  \quad
  \Clhb_L = \dfrac{C^0\oplus M}{(\Box_0,r)C^0}
  \supset \Clhb_L^0 = \dfrac{\im \partial\oplus M}{(\Box_0,r)C^0}\,.
\]
By Lemma \ref{LABEL1090}, we have a map $\iota_0\oplus \iota_L\colon \Clb_L \to
\Clhb_L$, and $(\iota_0\oplus\iota_L)\Clb^0_L \subset
\Clhb^0_L$. Analogous to Lemma \ref{LABEL1110}, we have (with the same proof):

\begin{lem}
  \label{LABEL1130}
  The maps
  \[
    \partial\oplus\id_L \colon \tilde C_1 = C_1\oplus L \to C_0 \oplus
    L,\quad
    d\adj\oplus \id_M \colon \tilde C^1 = C^1 \oplus M \to C^0\oplus M
  \]
  induce isomorphisms $\phi_L$, $\chi_L$ fitting into a commutative
  diagram
  \[
    \begin{tikzcd}
      \dfrac{\tilde C_1}{(\ker\partial)\oplus\{0\} + \im \partial_L\adjp}
      \arrow[r, "\iota_1"] \arrow[d, "\wr", "\phi_L"'] &
      \dfrac{\tilde C^1}{(\ker d\adj)\oplus \{0\} + \im d_L}
      \arrow[d, "\wr"', "\chi_L"]
      \\
      \Clb^0_L \arrow[r, "\iota_0\oplus \iota_L"] &
      \Clhb^0_L\,.
    \end{tikzcd} 
  \]
  \\[-5.5ex]\mbox{ }\qed
  \\[0ex]
\end{lem}
The inclusions $C_1\subset \tilde C_1$, $C^1\subset \tilde C^1$ and
the duality pairing $\tilde C_1 \times \tilde C^1\to\ZZ$ induce  maps
\[
  \tilde\alpha \colon \ker\partial \to (\ker d_L\adjp)\Zdual,\quad
  \tilde\alpha' \colon \ker d\adj \to H_1(\tilde C_\bullet)\Zdual,\quad
  \tilde\beta \colon \ker d\adj \to H^1(\tilde C^\bullet).
\]
Define
\[
  \Jb_L = \dfrac{(\ker d_L\adjp)\Zdual}{\tilde\alpha(\ker\partial)},\quad
  \Pbbar_L = \dfrac{(\ker \partial_L)\Zdual}{\tilde\alpha'(\ker
    d\adj)},\quad
  \Pb_L = \dfrac{H^1(\tilde C_L^\bullet)}{\tilde\beta(\ker
    d\adj)}  = \dfrac{\tilde C^1}{(\ker d\adj)\oplus\{0\} + \im d_L}\,.
\]
There is an obvious map $\Pb_L \to \Pbbar_L$, and
$\tr{(\iota_1\oplus\iota_L)}$ induces a homomorphism
$\zeta_L\colon \Jb_L \to \Pbbar_L$. From Lemma \ref{LABEL1130} we have an
isomorphism $\chi_L\colon \Clhb_L^0 \isom \Pb_L$. A similar
diagram chase as in the proof of Theorem \ref{LABEL1120} gives:

\begin{thm}
  \label{LABEL1140}
  \textup{(a)} There is a unique homomorphism
  $\AJ_L \colon \Clb_L^0 \to \Jb_L$ which takes the class of
  $(\partial x, l)$ (for $x\in C_1$, $l\in L$) to the class of the linear form
  $(y, m)\mapsto \apairing[1]xy+\apairing[L]lm$, $(y,m)\in \ker d_L\adjp\subset
  C^1\oplus M$.

  \textup{(b)} The diagram below is commutative:
  \[
    \begin{tikzcd}
      \Clb_L^0 \arrow[rr, "\iota_0\oplus\iota_L"] \arrow[d, "\AJ_L"] &&
      \Clhb_L^0 \arrow[d, "\wr"', "{\chi_L^{-1}}"]
      \\
      \Jb \arrow[r,"\zeta_L"] &\Pbbar&
      \Pb\,. \arrow[l]
    \end{tikzcd}
  \]
  \\[-5.5ex]\mbox{ }\qed
  \\[0ex]
\end{thm}

\subsection{Generalized Jacobian and Picard groups, II}
\label{LABEL1150}

Now let $\Gr=(V,E)$ be a connected finite graph, and
let $C_\bullet = C_\bullet(\Gr)$ be the chain complex for $\Gr$. Let
$I$ be a finite set and $\rho\colon I \to V$ be a map, which we extend
by linearity to a homomorphism $\rho\colon L\defeq \ZZ[I] \to
C_0$. Write $w_i=\rho(i)$, $\fm=\sum_{i\in I}w_i\in \ZZ[V]$
the associated modulus. For the pairings we take the positive definite
forms for which $V$, $E$ and $I$ are respectively orthonormal bases of
$C_0$, $C_1$ and $L$. The maps $\iota_0$, $\iota_1$ and $\iota_L$ are
isomorphisms, and so $\Clb^{(0)}_L \isom \Clhb^{(0)}_L$ and
$\zeta_L\colon \Jb_L\isom \Pbbar_L$. Moreover
$H_1(\tilde C_\bullet)\times H^1(\tilde C^\bullet)\to\ZZ$ is a perfect
pairing of free abelian groups of finite rank, and so
$\Pb_L \isom \Pbbar_L$. We conclude that all the maps in the
diagram of Theorem \ref{LABEL1140} are isomorphisms.

Let us reinterpret the various groups in this case. We have
$\ker d\adj = \Ha^1(\Gr)$ by definition and
$\ker d_L\adjp = \Ha^1(\Gr_\fm)$ by Proposition \ref{LABEL720}(b). By
Proposition \ref{LABEL320} the maps
$\alpha \colon H_1=\ker\partial \to \Ha^1(\Gr)\Zdual$,
$\alpha' \colon \Ha^1(\Gr) \to (H_1)\Zdual$ are injective.  We also
have
\[
  \im(\partial\oplus \id_L) = \ZZ[V]^0\oplus \ZZ[I] = \mathrm{Div}^0_\fm(\Gr)
\]
and so
\[
  \Clb_L^{(0)} = \mathrm{Cl}^{(0)}_\fm(\Gr),\quad
  \Clhb_L^{(0)} = \widehat{\mathrm{Cl}}^{(0)}_\fm(\Gr),\quad
  \Jb_L=\mathrm{J}_\fm(\Gr), \quad \Pb_L=\mathrm{P}_\fm(\Gr).
\]
Therefore from Theorem \ref{LABEL1140} we recover Theorem \ref{LABEL830}.

We now consider the case of an infinite (but, as always, locally finite) graph
$\Gr$. Then $\partial$ and $\partial\adj$ are homomorphisms between
free abelian groups. In particular, their images are free, and so
$\ker\partial \subset C_1$, $\ker\partial\adj \subset C_0$ are direct
summands. Moreover, $\im\partial$ is the submodule of $C_0$ consisting
of $0$-chains whose degree of each connected component of $\Gr$ is
zero, and so $C_0\simeq \im\partial \oplus \ZZ[\pi_0(\Gr)]$ is also
split. 

\begin{thm}
  \label{LABEL1160}
  Let $\Gr=(V,E)$ be a locally finite graph. Then there is a
  subset $E'\subset E$ such that
  $C_1(\Gr)=\im\partial\adj\oplus \ZZ[E']$. In particular,
  $\im\partial\adj \subset C_1$ is a direct summand.
\end{thm}

\begin{proof}
  Since the complex $(C_\bullet,\,\partial\adj)$ is additive for
  disjoint unions of graphs, we may assume that $\Gr$ is connected. If
  $\Gr'$ is obtained from $\Gr$ by removing the subset $E^*\subset E$
  of loops of $\Gr$, then $C_1(\Gr)=C_1(\Gr')\oplus \ZZ[E^*]$ and
  $\im\partial\adjp_\Gr=\im\partial\adjp_{\Gr'}$, so we may in
  addition assume that $\Gr$ has no loops. Since
  $\iota_0\colon \ker\partial\adj \hookrightarrow \ker d$, and
  $\ker d=H^0(\Gr)$ is the group of constant functions, we have that
  $\ker\partial\adj$ is isomorphic to $\ZZ$ if $\Gr$ is finite, and is
  zero otherwise.

  Let $\Gr_0=(V,E_0)$ be a spanning tree in $\Gr$. Then we have a commutative diagram:
  \[
    \begin{tikzcd}
      C_0(\Gr) \arrow[r, "{\partial\adjp_{\Gr}}"] \arrow[d, equals] &
      C_1(\Gr)=\ZZ[E_0]\oplus \ZZ[E\setminus E_0]
      \arrow[d, "\mathrm{pr}_1"] \\
      C_0(\Gr_0)  \arrow[r, "{\partial\adjp_{\Gr_0}}"] & C_1(\Gr_0)=\ZZ[E_0].
    \end{tikzcd}
  \]
  By the last observation, $\ker\partial\adjp_{\Gr}=\ker\partial\adjp_{\Gr_0}$.
  Suppose that $C_1(\Gr_0)=\im\partial\adjp_{\Gr_0} \oplus \ZZ[E_0']$
  for some $E_0'\subset E_0$.  Then
  $C_1(\Gr)=\im\partial\adjp_{\Gr}\oplus \ZZ[E_0'\cup (E\setminus
  E_0)]$. So it is sufficient to prove the Theorem for $\Gr$ a
  tree. As $\im\partial\adj$ is independent of the orientation of
  $\Gr$, we may assume that $\Gr$ is a rooted tree and that every edge
  is directed away from the root $v_0\in V$.

  Suppose that $\Gr$ is a finite tree. Then as $\ker\partial=0$ and
  $\im\partial=C_0(\Gr)_0$, we have that $\partial\adj$ is surjective
  and $\ker\partial\adj=\ZZ\cdot \sum_{v\in V}v$. So for any $v\in V$,
  $\partial\adj\colon \ZZ[V\setminus\{v\}] \isom C_1(\Gr)$ is an isomorphism.

  Now assume that $\Gr$ is infinite. Then there exists a unique
  maximal subtree $\Gr_1$ of $\Gr$ containing $v_0$, such that every
  vertex $v\ne v_0$ of $\Gr_1$ has degree $>1$. Its vertex set
  comprises all vertices $v$ of $\Gr$ for which there exists an
  infinite chain of edges originating at $v$. Then $\Gr$ is obtained
  by attaching finite trees to $\Gr_1$. Precisely, there is a family
  $(\Delta_v)_{v\in V}$ of pairwise disjoint finite subtrees of $\Gr$
  with
  \[
    \Gr = \Gr_1 \cup \bigcup_{v\in V}\Delta_v\text{ and } \Gr\cap \Delta_v=\{v\}.
  \]
  We then have a commutative diagram with exact rows
  \[
    \begin{tikzcd}
      0 \arrow[r] & \smash[b]{\displaystyle\bigoplus_v}\,
      \ZZ[V(\Delta_v)\setminus\{v\}]  
      \arrow[r] \arrow[d, "{\bigoplus_v\partial\adjp_{\Delta_v}}", "\simeq"'] &
      \ZZ[V(\Gr)] \arrow[r] \arrow[d, "\partial\adjp_\Gr"] &
      \ZZ[V(\Gr_1)] \arrow[r] \arrow[d, "\partial\adjp_{\Gr_1}"] & 0
      \\
      0 \arrow[r] & \displaystyle\bigoplus_v \ZZ[E(\Delta_v)] \arrow[r] &
      \ZZ[E(\Gr)] \arrow[r] & \ZZ[E(\Gr_1)] \arrow[r] & 0     
    \end{tikzcd}
  \]
  and by the finite case, the left-hand vertical arrow is an
  isomorphism. So it's enough to prove the Theorem for $\Gr=\Gr_1$,
  and may therefore assume that for all $v\in V$, $\deg(v)>1$.

  Let
  \[
    e(v,i),\ 0 \le i <
    \begin{cases}
      \deg(v) & \text{if $v=v_0$}\\
      \deg(v)-1 & \text{if $v\ne v_0$}
    \end{cases}
  \]
  be the edges $e$ of $\Gr$ with $o(e)=v$. For $v\ne v_0$ let $e'(v)$
  be the unique edge with $t(e)=v$. Then
  \[
    \partial\adj(v)=
    \begin{cases}
      -\sum_i e(v_0,i) & \text{if $v=v_0$} \\
      e'(v) -\sum_i e(v,i) & \text{if $v\ne v_0$} 
    \end{cases}
  \]
  from which one sees easily that
  $C_1(\Gr)=\im\partial\adj \oplus \ZZ[E_0]$, where
  \[
    E_0 = \bigcup_{v\in V} \{ e(v,i) \mid i \ne 0 \}. \qedhere
  \]
\end{proof}

Let $\fm=\sum_{i\in I}w_i$ be a (not necessarily finite) modulus on the graph
$\Gr$, as in Section \ref{LABEL630}. Taking $L=\ZZ[I]$ with the bilinear
form $\pairing[L]ij=\delta_{ij}$ we then may apply the previous
formalism, and we have the commutative diagram of Theorem
\ref{LABEL1140}.  As $H^1(\Gr)=H_1(\Gr)\Zdual$, the map
$\Pb_L\to \Pbbar_L$ is an isomorphism.

\begin{cor}
  \label{LABEL1170}
  Let $\Gr$ be a connected locally finite graph. Then
  $\AJ_L\colon \Clb^0_L \to \Jb_L$ is an isomorphism.
\end{cor}

\begin{proof}
  We may assume that $\Gr$ is infinite. As it is connected and locally
  finite, $E$ and $V$ are countably infinite sets.  From its
  definition,
  $\ker d\adj = (C_1/\im\partial\adj)\Zdual =
  \ZZ[E_0]\Zdual=\ZZ^{E_0}$, and by Specker's Theorem \cite{Sp}, the canonical
  homomorphism $\ZZ[E_0] \to (\ZZ^{E_0})\Zdual$ is an isomorphism of
  abelian groups. So
  $C_1/\im\partial\adj = \ZZ[E_0] \isom (\ker d\adj)\Zdual$.
\end{proof}

For an infinite graph, the equalities
\[
  \Clb_L^{(0)} = \mathrm{Cl}^{(0)}_\fm(\Gr),\quad \Clhb_L^{(0)} =
  \widehat{\mathrm{Cl}}^{(0)}_\fm(\Gr),\quad \Pb_L=\mathrm{P}_\fm(\Gr)
\]
still hold. The last corollary, and Corollary \ref{LABEL1000},
together suggest that for an infinite graph with arbitrary modulus,
the correct definition of the generalized Jacobian group
should be $\J_\fm(\Gr)= \Jb_L$. Combining these
results with Theorem \ref{LABEL1140}, we obtain a commutative diagram
\[
  \begin{tikzcd}
    \Cl_\fm(\Gr) \arrow[r] &
    \Clhat_\fm(\Gr) \arrow[r, "-\delta_\fm", "\sim"'] &
    \Pic_\fm(\Grg) \\
    \Cl^0_\fm(\Gr) \arrow[u, hookrightarrow] \arrow[d, "\AJ_\fm"', "\simeq"]&&\\
    \J_\fm(\Gr) \arrow[rr, "\zeta"] && \P_\fm(\Gr) \arrow[uu, hookrightarrow, "(**)"']
    \arrow[uul, "\chi"', hookrightarrow]
  \end{tikzcd}
\]
Simple examples show that in general $\zeta$ is neither injective or
surjective, and that $\ker d_L\adjp$ and $\Ha(\Gr_\fm)$ are not in
general equal.

\subsection{Functoriality}

We describe the analogues for graphs of Albanese and Picard
functoriality (compare the remarks after Theorems \ref{LABEL060} and
\ref{LABEL130}). The proofs of the compatibilities of the various maps
with functoriality are rather formal and we mostly leave them to the
reader.

We first review from \cite{bn2} (following previous work in \cite{U1,U2})
the notion of a harmonic morphism of graphs, with appropriate
modifications for oriented graphs and to allow loops.

Throughout this section, $\Gr$ and $\Gr'$ will be connected
locally finite graphs, with $E(\Gr)\ne\emptyset\ne E(\Gr')$.

\subsubsection*{Harmonic morphisms} \cite[\S2.1]{bn2}

\begin{defn}
  A morphism $\f\colon \Gr \to \Gr'$ of graphs is a pair of maps
  \[
    \f_V\colon V(\Gr) \to V(\Gr'),\quad
    \f_E \colon E(\Gr) \to E(\Gr') \sqcup  V(\Gr')
  \]
  such that, for all $e\in E(\Gr)$:
  \begin{itemize}
  \item if $\f_E(e)=v'\in V(\Gr')$ then $\f_V(o(e))=v'=\f_V(t(e))$.
  \item if $\f_E(e)=e'\in E(\Gr')$ then $\f_V(o(e))=o(e')$ and $\f_V(t(e))=t(e')$.
  \end{itemize}
\end{defn}
A morphism $\f\colon \Gr \to \Gr'$ induces a PL map
$\abs \f \colon \Grg \to \abs{\Gr'}$ on geometric realisations.

\begin{defn}
  Let $\f\colon \Gr \to \Gr'$ be a morphism of graphs, $v\in
  V(\Gr)$, $v'=\f(v)$. Say that $\f$ is \emph{harmonic} at $v$ if there exists a
  non-negative integer $m(\f,v)$, the \emph{horizontal multiplicity} of
  $\f$ at $v$ such that the fibres of the maps induced by $\phi_E$
  \begin{align*}
    \Phi_0 \colon & \{e \in E(\Gr) \mid \f_E(e)\in E(\Gr'),\ o(e)=v \} \to
                    \{e' \in E(\Gr') \mid  o(e')=v' \} \\
    \Phi_1 \colon & \{e \in E(\Gr) \mid \f_E(e)\in E(\Gr'),\ t(e)=v \} \to
                    \{e' \in E(\Gr') \mid  t(e')=v' \}
  \end{align*}
  all have cardinality $m(\f,v)$. If $\f$ is harmonic at every $v\in
  V(\Gr)$ say that $\f$ is harmonic.
\end{defn}
We may rewrite this definition in terms of geometric realisations: let
\[
  U=U_{1/3}(\Gr, v) \setminus \abs \f^{-1}(v'), \quad U'=U_{1/3}(\Gr',v')
  \setminus\{v'\}.
\]
Then $\f$ is harmonic at $v$, with horizontal multiplicity $m$, iff
the map $\abs \f \colon U \to U'$ is everywhere $m$-to-one. (This is
the analogue of the local mapping property for holomorphic functions
of a complex variable.)

Let $\f\colon \Gr \to \Gr'$ be a morphism of graphs. Since $\abs \f$ is PL, pullback of
functions determines a map of sheaves $\f^*\colon \abs \f^*\PL_{\Gr'}
\to \PL_{\Gr}$.

\begin{prop}
  \label{LABEL1180}
  \begin{enumerate}[\upshape(a)]
  \item $\f$ is harmonic at $v$ iff $\f^*_v (\Harm_{\Gr',\f(v)}) \subset
    \Harm_{\Gr,v}$.
  \item $\f$ is harmonic iff $\f^*$ maps $\abs \f^*\Harm_{\Gr'}$ to
    $\Harm_{\Gr}$.
  \end{enumerate}
\end{prop}

\begin{proof}
  (a) By definition, $\Harm_{\Gr,v}$ is the set of PL-functions in
  $\PL_{\Gr,v}$ for which the sum of the outgoing derivatives
  at $v$ vanishes.  Let $h \in\Harm_{\Gr',f(v)}$. Then the outgoing
  derivative of $\abs \f^*h$ along an edge $e$ with $o(e)=v$
  (resp.~$t(e)=v$) is zero if $\f_E(e)\in V(\Gr')$, and equals the
  outgoing derivative of $h$ along $\f_E(e)$ otherwise. So the sum
  of the outgoing derivatives of $\abs \f^*h$ vanishes for every
  $h$ iff all fibres of the maps $\Phi_0$ and $\Phi_1$ have the
  same cardinality.
  
  (b) Since $\Harm_{\Gr,e}=\PL_{\Gr,e}$ for every edge $e$, this
  follows from (a).
\end{proof}

\subsubsection*{Graphs without modulus}

The following results describe the functorial properties of the
Jacobian and Picard groups of graphs with respect to harmonic
morphisms, and are mostly contained in \cite{bn2}.  We give proofs of
more general statements in \ref{LABEL1190} below.

\begin{prop} \upshape{\cite[\S4.1]{bn2}}
  \begin{enumerate}[\upshape(a)]
  \item Let $\f\colon \Gr \to \Gr'$ be harmonic, and
    $\f_*\colon \Div(\Gr) \to \Div(\Gr')$ the linear extension of
    $\f_V$. Then $\f_*(\Prin(\Gr)) \subset \Prin(\Gr')$, inducing a
    homomorphism
    \[
      \f_*\colon \Cl(\Gr) \to \Cl(\Gr')
    \]
    which maps $\Cl^0(\Gr)$ into $\Cl^0(\Gr')$.
  \item Write $\f^m\colon \Divhat(\Gr') \to \Divhat(\Gr)$ for the
    transpose of the map $\ZZ[V(\Gr)] \to \ZZ[V(\Gr')]$,
    $v\mapsto m(\f,v)\f(v)$. Then $\f^m(\Prinhat(\Gr'))\subset
    \Prinhat(\Gr)$, inducing a homomorphism
    \[
      \f^*\colon \Clhat(\Gr') \to \Clhat(\Gr).
    \]
    If the graphs are finite, then $\f^*(\Clhat^0(\Gr')) \subset \Clhat^0(\Gr)$.
  \end{enumerate}
\end{prop}

\begin{prop}
  \begin{enumerate}[\upshape(a)]
  \item \upshape{\cite[Prop.~4.11]{bn2}} $\f$ induces a map
    $\f^*\colon \ker d_{\Gr'}\adjp \to \ker d_{\Gr}\adjp$.
  \item The following diagram commutes:
    \[
      \begin{tikzcd}
        \Cl^0(\Gr) \arrow[r, "\f_*"] \arrow[d, "\simeq", "\AJ_{\Gr}"'] &
        \Cl^0(\Gr') \arrow[d, "\simeq"', "\AJ_{\Gr'}"] \\
        \J(\Gr) = \dfrac{(\ker d_{\Gr}\adjp)\Zdual}{\alpha(H_1(\Gr))}
        \arrow[r] &
        \J(\Gr') = \dfrac{(\ker d_{\Gr'}\adjp)\Zdual}{\alpha(H_1(\Gr'))}
        \\
        (\ker d_{\Gr}\adjp)\Zdual \arrow[u, twoheadrightarrow]
        \arrow[r, "\tr{(\f^*)}"] &
        (\ker d_{\Gr'}\adjp)\Zdual \arrow[u, twoheadrightarrow]
      \end{tikzcd}
    \]
  \end{enumerate}
\end{prop}

\begin{prop}
  \label{LABEL1185}
  There is a commutative diagram
  \[
    \begin{tikzcd}
      \Clhat(\Gr') \arrow[d, "{\bar\delta_{\Gr'}}"',"\simeq"]
      \arrow[r, "\f^*"] &
      \Clhat(\Gr) \arrow[d, "{\bar\delta_{\Gr}}", "\simeq"'] \\
      \Pic(\abs{\Gr'}) \arrow[r, "\f^*"] &
      \Pic(\Grg) \\
      \P(\Gr')= \dfrac{H^1(\Gr')}{\beta(\ker d_{\Gr'}\adjp)}
      \arrow[u, "(*)"] \arrow[r] &
      \P(\Gr)= \dfrac{H^1(\Gr)}{\beta(\ker d_{\Gr}\adjp)}
      \arrow[u, "(*)"'] \\
      H^1(\Gr') \arrow[u, twoheadrightarrow]
      \arrow[r, "\f^*"] &
      H^1(\Gr) \arrow[u, twoheadrightarrow]
    \end{tikzcd}
  \]
  where the maps $\f^*$ on $\Pic$ and $H^1$ are pullback on cohomology, and the 
  vertical maps $(*)$ are those in equation \upshape{(2.5.6)}.
\end{prop}

For now, we just note that the commutativity of the top square follows
easily from the following description of the map $\bar\delta$: if
$D\in\ZZ^{V(\Gr')}$ then the corresponding $\Harm_{\Gr'}$-torsor
$\cL_D$ is the pullback
\[
  \begin{tikzcd}
    0 \arrow[r] & \Harm_{\Gr'} \arrow[r] &
    \PL_{\Gr'} \arrow[r, "(\mathrm{diff}_v)"] &
    \prod_v v_*\ZZ \arrow[r] & 0
    \\
    && \cL_D \arrow[u] \arrow[r] & \ast \arrow[u, "D"']
  \end{tikzcd}
\]
and so by the same reasoning as in the proof of Proposition
\ref{LABEL1180}(a), $\abs \f^*\cL_D  \simeq \cL_{\f^m(D)}$.

\subsubsection*{Graphs with modulus}
\label{LABEL1190}

Let $\f\colon \Gr \to \Gr'$ be a harmonic morphism of finite graphs.
Let $(w_i)_{i\in I}$, $(w'_j)_{j \in J}$ be families of vertices of
$\Gr$ and $\Gr'$ , with corresponding moduli $\fm=\sum w_i$,
$\fm'=\sum w'_j$. We give conditions under which the results of the
previous subsections can be extended to generalized Jacobians and
Picard groups. We will consider only the case of reduced modulus ---
using Remarks \ref{LABEL670} and \ref{LABEL860} one can easily extend
this to the general case. As before, write $S=\{w_i\}\subset V$,
$S'=\{w'_j\}\subset V'$. 

\begin{prop}
  \label{LABEL1200}
  Suppose that $\phi_V^{-1}(S') \supset S$, and let
  $\phi_S \colon \ZZ[S] \to \ZZ[S']$ denote the corresponding
  map. Then:
  \begin{enumerate}[\upshape(i)]
  \item $\phi_V\oplus \pi_S\colon \ZZ[V\setminus S] \to \ZZ[V'\setminus S']$
    induces a homomorphism
    \[
      \phi_*\colon \Cl_\fm(\Gr) \to \Cl_{\fm'}(\Gr').
    \]
  \item $\phi$ induces a map $\phi^* \colon \ker d_{\Gr'_{\fm'}}\adjp
    \to \ker d_{\Gr'_{\fm}}\adjp$, and the following diagram commutes:
    \[
      \begin{tikzcd}
        \Cl^0_\fm(\Gr) \arrow[r, "\f_*"] \arrow[d, "\simeq", "\AJ_{\Gr,\fm}"'] &
        \Cl^0_{\fm'}(\Gr') \arrow[d, "\simeq"', "\AJ_{\Gr',\fm'}"] \\
        \J_\fm(\Gr) = \dfrac{(\ker d_{\Gr_\fm}\adjp)\Zdual}{\alpha_\fm(H_1(\Gr))}
        \arrow[r] &
        \J_{\fm'}(\Gr') = \dfrac{(\ker
          d_{\Gr'_{\fm'}}\adjp)\Zdual}{\alpha_{\fm'}(H_1(\Gr'))} 
        \\
        (\ker d_{\Gr_\fm}\adjp)\Zdual \arrow[u, twoheadrightarrow]
        \arrow[r, "\tr{(\f^*)}"] &
        (\ker d_{\Gr'_{\fm'}}\adjp)\Zdual \arrow[u, twoheadrightarrow]
      \end{tikzcd}
    \]
  \end{enumerate}
\end{prop}

\begin{prop}
  \label{LABEL1210}
  Suppose that $\phi_V^{-1}(S') \subset S$, and let
  $\phi^S\colon \ZZ^{S'} \to \ZZ^S$ be the corresponding map. Then:
  \begin{enumerate}[\upshape(i)]
  \item
    $\phi^m\oplus \phi^S \colon \Divhat_{\fm'}(\Gr') \to
    \Divhat_\fm(\Gr)$ induces a homomorphism
    $\phi^* \colon \Clhat_{\fm'}(\Gr') \to \Clhat_\fm(\Gr)$, and
    $\phi^*(\Clhat_{\fm'}(\Gr'))\subset \Clhat_\fm(\Gr)$.
  \item
    There is a commutative diagram:
    \[
      \begin{tikzcd}
         \Clhat_{\fm'}(\Gr') \arrow[r, "\phi^*"] \arrow[d, "{\bar\delta_{\fm'}}"]
         & \Clhat_\fm(\Gr) \arrow[d, "{\bar\delta_\fm}"] \\
         \Pic_{\fm'}(\abs{\Gr'})\arrow[r, "{\abs\phi^*}"] &
         \Pic_\fm(\Grg) \\
         \P_{\fm'}(\Gr') \arrow[r] \arrow[u] & \P_\fm(\Gr) \arrow[u]\\
         H^1(\Gr'_{\fm'}) \arrow[u, twoheadrightarrow] \arrow[r, "\phi^*"] &
         H^1(\Gr_\fm)  \arrow[u, twoheadrightarrow]
      \end{tikzcd}
    \]
  \end{enumerate}
\end{prop}

Let us first deal with the top square in Proposition
\ref{LABEL1210}(b). This commutes by the same argument as given after
Proposition \ref{LABEL1185}.  For the rest of the proofs, it is
convenient to work in the abstract setting of the previous
subsection. So we have a system
$\mathcal C = (C_\bullet,\partial,\partial\adj, L, \rho,\rho\adj)$ of
homomorphisms of abelian groups with symmetric integral bilinear forms
and formal adjoints
\[
  \begin{tikzcd}
    L \arrow[r, "\rho"] &
    C_0 \arrow[l, bend left, "{\rho\adj}"]
    \arrow[r, "{\partial\adj}"', bend right] & C_1 \arrow[l, "\partial"']
  \end{tikzcd}
\]
and their transposes
\[
  \begin{tikzcd}
    M = \hspace*{-3em}&L^\vee  \arrow[r, bend right, "{r\adj}"'] &
    C^0 \arrow[l, "r"'] \arrow[r, "d"] &
    C^1 \arrow[l, bend left, "{d\adj}"].
\end{tikzcd}
\]
Let $\lp\mathcal C = (\lp C_\bullet, \lp\partial,\lp\partial\adj, \lp L,
\lp\rho, \lp\rho\adj)$ be another such system. Suppose that $\phi_0$,
$\phi_m \colon C_0 \to \lp C_0$, $\phi_1\colon C_1\to \lp C_1$ are
homomorphisms such that the diagram
\[
  \begin{tikzcd}
    C_0 \arrow[r, "{\partial\adj}"] \arrow[d, "\phi_m"] &
    C_1 \arrow[r, "\partial"] \arrow[d, "\phi_1"] &
    C_0 \arrow[d, "\phi_0"] \\
    \lp C_0 \arrow[r, "{\lp\partial\adj}"] &
    \lp C_1 \arrow[r, "{\lp\partial}"] & \lp C_0
  \end{tikzcd}
\]
commutes. (We might say that $\phi_\bullet$ is a harmonic morphism
from $\mathcal C$ to $\lp\mathcal C$.) Let $\phi^i=\tr{(\phi_i)}$ for
$i\in \{0,1,m\}$. Denote by $\lp\Clb_L$, $\lp\Jb_L$ etc.~the
corresponding groups for the system $\lp C_\bullet$. Propositions
\ref{LABEL1200} and \ref{LABEL1210} then follow from the following
compatibilities (whose proofs are straightforward).

\begin{prop}
  Suppose that there exists $\mu_L \colon L \to \lp L$ such that the
  square
  \[
    \begin{tikzcd}
      C_0 \arrow[r, "\phi\adj"] \arrow[d, "\phi_m"'] &
      L \arrow[d, "\mu_L"] \\
      \lp C_0 \arrow[r, "{\lp\rho\adj}"'] & \lp L
    \end{tikzcd}
  \]
  commutes. Let $\mu^M=\tr{(\mu_L)}\colon \lp M \to M$ be its
  transpose.
  \begin{enumerate}[\upshape(i)]
  \item $\phi_0 \oplus \mu_L$ induces a homomorphism $\phi_*\colon
    \Clb_L \xrightarrow{\phi_*} \lp\Clb_L$, and
    $\phi_*(\Clb_L^0) \subset \lp\Clb_L^0$.
  \item $\phi^1\oplus \mu^M\colon \lp\tilde C^1=\lp C^1\oplus \lp M
    \to  \tilde C^1$ induces homomorphisms $(\ker d_L\adjp)^\vee \to
    (\lp d_L\adjp)^\vee$ and $\Jb_L \xrightarrow{\phi_{J,L}}
    \lp\Jb_L$.
  \item The diagram
    \[
      \begin{tikzcd}
        \Clb_L^0 \arrow[r, "\phi_*"] \arrow[d, "\AJ"', "\simeq"] &
        \lp\Clb_L^0 \arrow[d, "\AJ", "\simeq"'] \\
        \Jb_L \arrow[r, "\phi_{J,L}"] & \lp\Jb_L \\
        (\ker d_L\adjp)^\vee  \arrow[r] \arrow[u, twoheadrightarrow] &
        (\ker \lp d_L\adjp)^\vee  \arrow[u, twoheadrightarrow]
      \end{tikzcd}
    \]
    commutes.
  \end{enumerate}
\end{prop}

\begin{prop}
  Suppose that there exists $\phi_L \colon L \to \lp L$ such that the
  square
  \[
    \begin{tikzcd}
      L \arrow[r, "\rho"] \arrow[d, "\phi_L"'] & C_0 \arrow[d,
      "\phi_0"] \\
      \lp L \arrow[r, "\rho"] & \lp C_0 
    \end{tikzcd}
  \]
  commutes. Let $\phi^M =\tr(\phi_L)\colon \lp M \to M$ denote its
  transpose. Then:
  \begin{enumerate}[\upshape(i)]
  \item $\phi^0\oplus \phi^M\colon \lp C^0 \oplus \lp M \to C^0 \oplus
    M$ induces a homomorphism $\phi^*\colon \lp \Clhb_L \to \Clhb_L$,
    and $\phi^*(\Clhb_L^0)\subset \Clhb_L^0$.
  \item $\phi^1\oplus \phi^M\colon \lp C^1 \oplus \lp M \to C^1 \oplus
    M$ induces a homomorphism $\lp \Pb_L \xrightarrow{\phi_{P,L}}
    \Pb_L$, and the diagram
    \[
      \begin{tikzcd}
         \lp \Clhb_L^0 \arrow[r, "\phi^*"] \arrow[d, "\lp\chi"', "\simeq"]&
         \Clhb_L \arrow[d, "\chi", "\simeq"']
        \\
        \lp \Pb_L \arrow[r, "\phi_{P,L}"] & \Pb_L
      \end{tikzcd}
    \]
  \end{enumerate}
\end{prop}

One can also deduce corresponding results for harmonic morphisms of
infinite graphs with modulus, whose statements we leave to the
interested reader.

\section*{Appendix}

We give a proof of Theorem \ref{LABEL130} (including as a
special case Theorem \ref{LABEL060}).

Let $D\in\Div^0(X)$. We first compute an explicit \v Cech $1$-cocycle for
$\cO_X(-\fm)$, representing the element of $\P_\fm(X)$ which maps to the class
of $\cO_X(D)\in\Pic_\fm(X)$.

Choose a simply-connected open neighborhood $U$
of $\supp(D)$, whose closure is disjoint from $S$. Let $X=U\cup V$
be an open cover, where $U\cap V$ is a topological annulus disjoint
from $\supp(D)$. Let $f$ be any meromorphic
function on $U$ whose divisor on $U$ equals $D$. The trivializations
\[
  f^{-1} \colon \cO_U \isom \cO_X(D)|_U, \quad
  1 \colon \cO_V \isom \cO_X(D)|_V
\]
determine the \v Cech cocycle
$f^{-1}|_{U\cap V}\in C^1(\{U,V\},\cO_X^{\times,\fm})$ representing
the class $\Pic_\fm(X)$ of $(\cO_X(D),\triv)$.

Since $D$ has degree $0$, $f$ has a single-valued logarithm $\log f$
on $U\cap V$, and the class in $H^1(X,\cO_X(-\fm))$ of the \v Cech cocycle
$-\log f|_{U\cap V}\in C^1(\{U,V\},\cO_X(-\fm))$ maps under the
exponential to the class of $(\cO_X(D),\triv)$ in $\Pic_\fm(X)$.

We next compute the image of this cocycle under Serre duality. More
generally, for any holomorphic vector bundle $\cE$ on $X$, we have the
Dolbeault isomorphism
\begin{equation} \label{LABEL1220}
  H^1(X, \cE) \isom H^1_{\bar\partial}(X,\cE)
  = A^{0,1}(X,\cE)/\bar\partial A^{0,0}(X,\cE)
\end{equation}
and Serre duality $H^0(X,\Omega_X^1\otimes\cE^\vee)^* \simeq H^1(X,\cE)$ is
induced by the pairing
\[
H^0(X,\Omega_X\otimes\cE^\vee) \times A^{0,1}(X,\cE) \to \mathbb{C},
\qquad
(\omega,\eta)\mapsto \int_X \omega\wedge\eta
\]
Let $s\in \cE(U\cap V)$ be a \v Cech 1-cocycle for $\cE$ and the covering $\{U,V\}$. Let
$b\colon X \to \RR$ be a $C^\infty$-function vanishing on
$U\setminus V$ and equal to 1 on $V\setminus U$. Then $\bar\partial b$
is a $(0,1)$-form on $X$ whose support in contained in $U\cap V$.

\noindent\textbf{Claim:} Under the isomorphism \eqref{LABEL1220}, the class
of the cocycle $s\in\cE(U\cap V)$ maps to the class of $s\,\bar\partial b\in
A^{0,1}(X,\cE)$.

To see this, recall the explicit description of
\eqref{LABEL1220}. Consider the \v Cech bicomplex
\begin{equation}
  \label{LABEL1230}
  \begin{tikzcd}
    A^{0,0}(U,\cE) \oplus A^{0,0}(V,\cE) \arrow[d, "{(u,v)\mapsto u-v}"]
    \arrow[r, "{(\bar\partial,\bar\partial)}"]
    & A^{0,1}(U,\cE)\oplus A^{0,1}(V,\cE) \arrow[d,
    "{(\alpha,\beta)\mapsto \beta-\alpha}"] \\
    A^{0,0}(U\cap V,\cE) \arrow[r, "{\bar\partial}"] & A^{0,1}(U\cap V, \cE).
  \end{tikzcd}
\end{equation}
Let $C^\bull_{tot}$ be its total complex. The inclusion of
$C^\bull(\{U,V\},\cE)=[\cE(U)\oplus\cE(V) \to \cE(U\cap V)]$ into the first column 
induces the isomorphism $H^*(X,\cE) \isom H^*(C^\bull_{tot})$, and the
diagonal inclusion of $[A^{0,0}(X,\cE) \xrightarrow{\bar\partial}
A^{0,1}(X,\cE)]$ into the top row induces the isomorphism
$H^*_{\bar\partial}(X,\cE) \isom H^*(C^\bull_{tot})$. So it's enough
to observe that 
\[
	(0,0,s), \ (s\bar\partial b,s\bar\partial b,0) \in C^1_{tot}
	= A^{0,1}(U,\cE) \oplus
  A^{0,1}(V,\cE)\oplus A^{0,0}(U\cap V,\cE)
\]
are 1-cocycles in $C^\bull_{tot}$  which are cohomologous --- indeed, their
difference $(-s\bar\partial b, -s\bar\partial b, s)$ is the coboundary of
\[
  ((1-b)s,-bs)) \in A^{0,0}(U,\cE)\oplus A^{0,0}(V,\cE) = C^0_{tot}.
\]
So the claim holds. In our particular situation, we see that the
class of the \v Cech cocycle $-\log f|_{U\cap V}\in
C^1(\{U,V\},\cO_X(-\fm))$ maps under the Dolbeault isomorphism to the
class of $(-\log f)\,\bar\partial b\in A^{0,1}(X,\cO(-\fm))$.

So to complete the proof of Theorem \ref{LABEL130}, it's enough to
show that if $\gamma\in C_1(U,\ZZ)$ is a 1-chain whose boundary is
$D$, then for every
$\omega\in H^0(X,\Omega^1(\fm))\subset H^0(U,\Omega^1)$,
\begin{equation}
  \label{LABEL1240}
  \int_\gamma\omega = \int_U \omega\wedge(- \log f)\, \bar\partial b
\end{equation}
For this we may assume that $D=P-P_0$, $P_0\in U$. Let
$w\colon U \hookrightarrow \mathbb{C}$ be a chart, $w(P)=z$,
$w(P_0)=a$. Then $\log f = \log (w-z)/(w-a) + g$ with $g$ holomorphic
on $U$. The equation \eqref{LABEL1240} holds for $P=P_0$, since the
right-hand integrand is the 2-form $-\omega\wedge g\,\bar\partial b$
which is exact, and both sides are holomorphic functions of $P\in
U$. Write $\omega=h(w)\,dw$. Then the $z$-derivatives of the left- and right-hand sides of
\eqref{LABEL1240} are respectively $h(z)$ and
\[
  \frac{1}{2\pi i} \int_U \frac{h(w)\, dw\wedge
  \bar\partial b}{w-z}.
\]
and from Cauchy's integral formula \cite[p.3]{GH} for the
$C^\infty$-function $(1-b)h$, we see that they are equal.\qed

\bibliographystyle{amsxport}
\bibliography{GJGTonyGx}
\end{document}